\RequirePackage{etex}
\pdfoutput=1
\documentclass{amsart}
\usepackage{listings}
\usepackage{enumerate, calligra}
\usepackage{amsxtra}
\usepackage{iftex}
\ifLuaTeX
\RequirePackage{luatex85}
\fi
\usepackage[all,cmtip]{xy}
\usepackage{xcolor}
\usepackage{graphicx}
\usepackage{comment}
\usepackage{mathtools}
\usepackage{amssymb}
\usepackage{xparse}
\usepackage{xr}
\ifpdf
\usepackage[hypertexnames=false]{hyperref}
\else
\usepackage[hypertexnames=false,dvipdfmx]{hyperref}
\fi
\usepackage[nameinlink]{cleveref}
\usepackage{autonum}
\usepackage{array}
\usepackage{subdepth}
\usepackage{bm,bbm}
\usepackage{mathrsfs}
\usepackage{stmaryrd}
\usepackage{rsfso}
\usepackage{latexsym}

\makeatletter
\renewcommand\th@plain{\slshape}
\makeatother
\newtheoremstyle{plain}% <name>
 {2mm}%                 <space above>
 {2mm}%                         <space below>
 {\slshape}%              <body font>
 {}%                         <indent amount>
 {\bfseries}%                <theorem head font>
 {.}%                        <punctuation after theorem head>
 {.5em}%                     <space after theorem head>
 {}%                         <theorem head spec>
\theoremstyle{plain}
\newtheorem{theorem}{Theorem}[section]
\newtheorem{corollary}[theorem]{Corollary}
\newtheorem{lemma}[theorem]{Lemma}
\newtheorem{proposition}[theorem]{Proposition}
\newtheorem{claim}[theorem]{Claim}
\newtheorem*{claim*}{Claim}

\newtheoremstyle{definition}% <name>
 {2mm}%                 <space above>
 {2mm}%                         <space below>
 {\normalfont}%              <body font>
 {}%                         <indent amount>
 {\bfseries}%                <theorem head font>
 {.}%                        <punctuation after theorem head>
 {.5em}%                     <space after theorem head>
 {}%                         <theorem head spec>
\theoremstyle{definition}

\newtheorem{definition}[theorem]{Definition}
\newtheorem{notation}[theorem]{Notation}

\newtheorem{remark}[theorem]{Remark}

\newtheorem*{acknowledgements}{Acknowledgements}

\crefname{section}{Section}{Sections}

\crefname{theorem}{Theorem}{Theorems}
\crefname{corollary}{Corollary}{Corollaries}
\crefname{lemma}{Lemma}{Lemmas}
\crefname{lemma}{Lemma}{Lemmas}
\crefname{proposition}{Proposition}{Propositions}
\crefname{claim}{Claim}{Claims}
\crefname{definition}{Definition}{Definitions}
\crefname{notation}{Notation}{Notations}
\crefname{problem}{Problem}{Problems}
\crefname{note}{Note}{Notes}
\crefname{remark}{Remark}{Remarks}
\crefname{example}{Example}{Examples}

\crefname{enumi}{}{}
\crefname{enumii}{}{}
\crefname{enumiii}{}{}

\crefformat{equation}{{\upshape(#2#1#3)}}

\numberwithin{equation}{section}

\usepackage{appendix}

\usepackage{enumitem}
\setlist{font=\upshape}

\usepackage{tikz}

\allowdisplaybreaks

\newcommand{\proofitem}[1]{\noindent\cref{#1}.}
\newcommand{\prooftitle}[1]{\noindent\textsf{#1}}

\newcommand{\textdf}[1]{\textbf{\textsl{#1}}}

\newcommand{\midmid}{\mathrel{}\middle|\mathrel{}}
\newcommand{\semicolon}{\nobreak\mskip2mu\mathpunct{}\nonscript\mkern-\thinmuskip{;}\mskip6muplus1mu\relax}
\renewcommand{\emptyset}{\varnothing}

\DeclareMathOperator{\Supp}{Supp}

\DeclareMathOperator{\GL}{GL}
\DeclareMathOperator{\Spec}{Spec} %%Yohsuke added
\newcommand{\prim}{\mathrm{prim}}
\newcommand{\loc}{\mathrm{loc}}

\newcommand{\fin}{\mathrm{fin}}
\newcommand{\NIP}{\mathrm{NIP}} %%Yohsuke added
\newcommand{\pr}{\mathrm{pr}} %%Yohsuke added

\DeclareMathOperator{\rank}{rank}
\DeclareMathOperator{\spa}{span}
\DeclareMathOperator{\spanR}{\spa_{\R}}
\DeclareMathOperator{\Ker}{Ker}
\DeclareMathOperator{\Cok}{Cok}
\DeclareMathOperator{\Img}{Im}

\renewcommand{\mod}[1]{(\mathrm{mod}\ #1)}

\DeclareMathOperator{\vol}{vol}

\newcommand{\N}{\mathbb{N}}
\newcommand{\Z}{\mathbb{Z}}
\newcommand{\Zprim}{\Z_{\prim}}
\newcommand{\Q}{\mathbb{Q}}
\newcommand{\R}{\mathbb{R}}
\newcommand{\C}{\mathbb{C}}
\renewcommand{\P}{\mathbb{P}}
\newcommand{\V}{\mathbb{V}}
\newcommand{\F}{\mathbb{F}}  %%Yohsuke added
\newcommand{\A}{\mathbb{A}} %%Yohsuke added

\newcommand{\mathvec}{\mathbf} %For universal replacement of the vector notation.
\let\aa\relax

\newcommand{\aa}{\mathvec{a}}
\newcommand{\bb}{\mathvec{b}}
\newcommand{\cc}{\mathvec{c}}
\newcommand{\ee}{\mathvec{e}}

\newcommand{\xx}{\mathvec{x}}
\newcommand{\yy}{\mathvec{y}}
\newcommand{\zz}{\mathvec{z}}
\newcommand{\uu}{\mathvec{u}}
\newcommand{\vv}{\mathvec{v}}

\newcommand{\xii}{\bm{\xi}}
\newcommand{\etaa}{\bm{\eta}}
\newcommand{\TT}{\mathbf{T}}

\renewcommand{\a}{\alpha}

\newcommand{\f}{\varphi}

\renewcommand{\l}{\lambda}

\newcommand{\s}{\sigma}

\newcommand{\G}{\Gamma}
\renewcommand{\L}{\Lambda}

\newcommand{\dd}{\mathfrak{d}}

\renewcommand{\tilde}{\widetilde}

\newcommand{\ball}[2]{\mathcal{B}_{#1}(#2)}
\newcommand{\cone}[3]{\mathcal{C}_{#1}(#2,#3)}

\newcommand{\realeps}{\sigma}
\newcommand{\realeta}{\tau}
\newcommand{\q}{\mathfrak{q}}

\DeclarePairedDelimiter\ceil{\lceil}{\rceil}       %%\ceil*{}
\DeclarePairedDelimiter\floor{\lfloor}{\rfloor}  %%\floor*{}

\renewcommand{\and}{\quad\textup{and}\quad}

\makeatletter
\NewDocumentCommand{\xsideset}{mmme{_^}}{%
\mathop{%
\settowidth{\dimen0}{$\m@th\displaystyle#3$}%
\dimen0=.5\dimen0
\settowidth{\dimen2}{$%
\m@th\displaystyle#3%
\IfValueT{#4}{_{#4}}%
\IfValueT{#5}{^{#5}}%
$}%
\dimen2=.5\dimen2
\advance\dimen2 -\dimen0
\sbox6{\scriptspace\z@$\displaystyle{\vphantom{#3}}#1$}
\sbox8{\scriptspace\z@$\displaystyle{\vphantom{#3}}#2$}
\ifdim\wd6>\dimen2 \kern\dimexpr\wd6-\dimen2\relax\fi
{%
\mathop{\llap{\copy6}{\displaystyle#3}\rlap{\copy8}}\limits%
\IfValueT{#4}{_{#4}}%
\IfValueT{#5}{^{#5}}%
}%
\ifdim\wd8>\dimen2 \kern\dimexpr\wd8-\dimen2\relax\fi
}%
}
\makeatother

\newcommand{\dsum}[1]{\xsideset{}{^{#1}}\sum}
\newcommand{\astsum}{\dsum{\smash{\ast}}}
\newcommand{\dcup}[1]{\xsideset{}{^{#1}}\bigcup}
\newcommand{\astcup}{\dcup{\smash{\ast}}}

\newcommand{\dprod}[1]{\xsideset{}{^{#1}}\prod}
\newcommand{\astprod}{\dprod{\smash{\ast}}}

\begin{document}

\title[Distribution of rational points on random Fano hypersurfaces]
{Distribution of rational points\\on random Fano hypersurfaces}
\author[Y. Matsuzawa]{Yohsuke Matsuzawa}
\author[Y. Suzuki]{Yuta Suzuki}

% \address{Department of Mathematics, Graduate School of Science, Osaka Metropolitan University, 3-3-138, Sugimoto, Sumiyoshi, Osaka, 558-8585, Japan}

% \email{matsuzaway@omu.ac.jp}
% \email{suzuyu@rikkyo.ac.jp}

\begin{abstract}
    We prove an asymptotic formula for the average number of rational points 
    on Fano hypersurfaces that are contained in a small ball centered at a given
    adelic point.
    We also prove an asymptotic formula for the number of hypersurfaces admitting adelic points
    that are contained in a small ball.
\end{abstract}


\maketitle

\tableofcontents

%%%%%%%%%%%%%%%%%%%%%%%%%%%%%%%%%%%%%%%%
\section{Introduction}
\label{sec:intro}

Let $V$ be a projective variety defined over $\Q$.
The set $V(\Q)$ of rational points on $V$ is naturally embedded in the set of adelic points of $V$:
\begin{align}
V(\Q)\hookrightarrow \prod_{p\in M_{\Q}}V(\Q_p)
\end{align} 
where $M_{\Q}$ is the set of all places of $\Q$.
To study the distribution of rational points,
it is natural to study the proportion of rational points that are contained in 
a measurable subset of $\prod_{p\in M_{\Q}}V(\Q_p)$.
There is a general equidistribution principle that such proportion should be governed by
a naturally constructed measure on $\prod_{p\in M_{\Q}}V(\Q_p)$ (cf.\ \cite{huang2021}). 
It seems very difficult to prove such equidistribution for a given variety,
but the question such as if it holds for 100\% of varieties in a given class of varieties 
might be tractable.
Indeed, recently Browning, Le Boudec, and Sawin proved that 100\% of Fano hypersurfaces (except cubic surfaces) satisfy the Hasse principle, which is believed to be very difficult to prove
a given Fano variety \cite{BBS}.

Following their idea as well as \cite{leBoudec},
we investigate the number of rational points on random hypersurfaces
that are contained in an adelic neighborhood (actually a box) of a given adelic point.
More precisely, we study the average of counting functions of those points 
with bounded height, where the average is taken over all hypersurfaces.
To state our main theorems, let us first introduce 
the set of all hypersurfaces.

\begin{definition}\label{def:monomial_homog_poly_hypsurf}
Let $n,d \in \Z_{\geq1}$ be positive integers.
\begin{enumerate}
\item
$\mathbb{V}_{d,n} = \left\{ V \subset \P^{n}_{\Q} \ \middle|\ \text{$V$ is a hypersurface of degree $d$} \right\}$.

\item
Let $R$ be a commutative ring.
Let $X_{0}, \dots, X_{n}$ be indeterminate.
The set of all degree $d$ monomials in $X_{0}, \dots, X_{n}$ is denoted by $\mathcal{M}_{d,n}$.
Ordering the members of $\mathcal{M}_{d,n}$ by the lexicographic order,
we identify
\begin{align}\label{homog-poly-RN}
\left\{  \txt{homogeneous polynomials of degree $d$\\ in $X_{0},\dots, X_{n}$ with coefficients in  $R$}    \right\}
 = R^{N_{d,n}}
\end{align}
where 
\begin{align}
N_{d,n}\coloneqq\binom{n+d}{d} = \# \mathcal{M}_{d,n}.
\end{align}
For $\aa \in R^{N_{d,n}}$, 
we denote the corresponding homogeneous polynomial by $f_{\aa}$.

\item 
Any $V \in \mathbb{V}_{d,n}$ has exactly two defining equations with coefficients in $\Z$ whose coefficient vectors
are primitive.
One of such coefficient vectors are denoted by $\aa_{V}$.
(Thus we have $V = V_+(f_{\aa})$ where $V_+(f_{\aa})$ stands for the closed subscheme of the projective space defined by $f_{\aa}$.)
For $A \geq 1$, we define
\begin{align}
{\mathbb{V}}_{d,n}(A) := \left\{ V \in {\mathbb{V}}_{d,n} \ \middle|\ \| \aa_{V}\| \leq A \right\}
\end{align}
where $ \|\cdot\|$ is the usual Euclidean norm.

\end{enumerate}
\end{definition}

We use the following metric on $\P^{n}(\Q_{p})$ to measure the $p$-adic distances between points.

\begin{definition}
Let $p \in M_{\Q}$.
For $x = (x_{0} : \cdots : x_{n}), y=(y_{0}:\cdots : y_{n}) \in \P^{n}(\Q_{p})$, we write
$\xx = (x_{0}, \dots, x_{n}), \yy = (y_{0}, \dots, y_{n}) \in \Q_{p}^{n+1}$ and set
\begin{align}
d_{p}(x,y) = \frac{ \|  \xx \wedge \yy \|_{p}}{ \|\xx  \|_{p} \| \yy \|_{p}}.
\end{align}
Here, when the place $p$ is finite,
let $\|\cdot\|_{p}$ be the max norms on $\Q_{p}^{n+1}$ and $\bigwedge^2\Q_{p}^{n+1}$
with respect to the standard basis $\ee_0,\ldots,\ee_n$
and $\ee_i\wedge\ee_j$ with $0\le i<j\le n$.
When $p=\infty$,
let $\|\cdot\|_{\infty}$ be the Euclidean norms
on $\R^{n+1}$ and $\bigwedge^2\R^{n+1}$
with respect to the standard basis $\ee_0,\ldots,\ee_n$
and $\ee_i\wedge\ee_j$ with $0\le i<j\le n$.
\end{definition}

This function $d_{p}$ is well-defined non-negative function
on $\P^{n}(\Q_p)\times\P^{n}(\Q_p)$, symmetric and 
$d_{p}(x,y) =0$ if and only if $x=y$.
One can also prove $d_{p}$ satisfies
the strong triangle inequality for $p<\infty$
and the triangle inequality for $p=\infty$
(see e.g.\ \cite[Section~1.1]{Rumely})
and so $d_p$ is a genuine metric on $\P^{n}(\Q_p)$.
Note that $d_{\infty}(x,y)=|\sin\theta|$
with the angle $\theta\in[-\frac{\pi}{2},+\frac{\pi}{2}]$ between $\xx$ and $\yy$ since
\begin{equation}
\label{wedge_to_innerproduct}
\|\xx\wedge\yy\|^2
=\|\xx\|^2\cdot\|\yy\|^2-\langle\xx,\yy\rangle^2
=\|\xx\|^2\|\yy\|^2|\sin\theta|^2
\end{equation}
where $\langle\cdot,\cdot\rangle$ is the usual Euclidean inner product.

Now let us introduce the counting function of rational points
that are contained in an adelic neighborhood of a given point.

\begin{definition}
Let $n, d \in \Z_{\geq1}$ and suppose $d \leq n$ (so that general members of $\V_{d,n}$ are smooth Fano varieties).

\begin{enumerate}
\item
For $x \in \P^{n}(\Q)$ with primitive homogeneous coordinates $\xx \in \Z^{n+1}$,
we define 
\begin{align}
H(x) = \| \xx \|^{n+1-d}.
\end{align}
Note that when restricted on a hypersurface of degree $d$, this is the multiplicative height function
associated with the anti-canonical divisor.

\item 
We write
\if0
\begin{align}
    &\astprod_{p \in M_{\Q}} (0,1]\\
    &\coloneqq \left\{ (\s_p)_p \in \prod_{p \in M_{\Q}} (0,1] \ \middle|\  
    \begin{gathered}
    \text{$\s_p=1$ for all but finitely many $p$}\\
    \text{$\sigma_{p}=p^{-e_{p}}$ with some $e_{p}\in\mathbb{Z}_{\ge0}$ for all $p<\infty$}
    \end{gathered}\right\}
\end{align}
\fi
\begin{align}
    &\astprod_{p \in M_{\Q}} (0,1]\\
    &\coloneqq \left\{ (\s_p)_p \in \prod_{p \in M_{\Q}} (0,1] \ \middle|\  
    \parbox{20em}{
    $\s_p=1$ for all but finitely many $p$\\
    $\sigma_{p}=p^{-e_{p}}$ with some $e_{p}\in\mathbb{Z}_{\ge0}$ for all $p<\infty$
    }
    \right\}
\end{align}
For $\s = (\s_p)_p \in \prod_{p \in M_{\Q}}^* (0,1]$,
we write
\[
\sigma_{p}=p^{-e_{p}}
\quad\text{for $p<\infty$}
\]
and set
\[
q=q(\sigma)\coloneqq\prod_{\substack{p\in M_{\Q}\\p<\infty}}p^{e_{p}},\quad
\q=\q(\sigma)\coloneqq\frac{q(\sigma)}{\sigma_{\infty}}=\prod_{p\in M_{\Q}}\frac{1}{\sigma_{p}}
\]
and
\[
\Supp \s \coloneqq \{ p \in M_{\Q} \mid \s_p < 1\}.
\]

\item
Let $\xi = (\xi_{p})_{p \in M_{\Q}} \in \prod_{p \in M_{\Q}} \P^{n}(\Q_{p})$
and let $\s = (\s_{p})_{p \in M_{\Q}} \in \prod_{p \in M_{\Q}}^* (0,1]$.
%Let $S = \Supp \s$.
For $V \in \V_{d,n}$ and $B \geq 1$, we define 
\begin{align}
&N_V(B;\xi,\realeps)\\
&\coloneqq
\#\{x\in V(\Q)\mid H(x)\le B\ \text{and}\ d_p(x,\xi_p)\le\realeps_p\ \text{for all $p\in M_\Q$}\}. \label{def-of-NV}
\end{align}
\end{enumerate}

\end{definition}

%%%%%%%%%%%%%%%%%%%%%%%%%%%%%%%%%%%%%%%%
The following is our first main theorem.
\begin{theorem}
\label{mainthm:numerator}
Let $n\ge d\ge2$ with $(n,d)\neq(2,2)$, $\xi \in \prod_{p \in M_{\Q}} \P^n(\Q_p)$
and $\s \in \prod_{p \in M_{\Q}}^* (0,1]$.
For $A,B\ge1$, we have
\[
\sum_{V\in\mathbb{V}_{d,n}(A)}
N_V(B;\xi,\realeps)
=
\widetilde{C}_{d,n}(\xi,\sigma)A^{N_{d,n}-1}B(1+R_{d,n}(A,B;\xi,\realeps))
\]
with the error term $R_{d,n}(A,B;\xi,\realeps)$ bounded as
\[
R_{d,n}(A,B;\xi,\realeps)
\ll
(\q A^{-1}
+
\q B^{-\frac{1}{n}}
+
A^{-1}B^{\frac{d+1}{(2n+1)(n+1-d)}}
+
\q^{n-1}A^{-(2n-1)}B^{\frac{d}{n+1-d}}
)B^{\varepsilon}
\]
provided
\[
A\ge\max(\q, B^{\frac{d+1}{(2n+1)(n+1-d)}},\q^{\frac{n-1}{2n-1}}B^{\frac{d}{(2n-1)(n+1-d)}})
\quad\text{and}\quad
B\ge\q^{n},
\]
where the coefficient $\widetilde{C}_{d,n}(\xi,\sigma)$ is explicitly given by
\begin{align}
\widetilde{C}_{d,n}(\xi,\sigma)
&\coloneqq
\frac{V_{N_{d,n}-1}}{4\zeta(N_{d,n}-1)}
\mathfrak{W}_{d,n}(\xii_{\infty};\sigma_{\infty})
\frac{\varphi(q)}{J_{n+1}(q)\zeta(n+1)},\\
\mathfrak{W}_{d,n}(\xii_{\infty},\realeps_{\infty})
&\coloneqq
\int_{\mathcal{C}_{n+1}(\xii_{\infty},\sigma_{\infty})\cap\mathcal{B}_{n+1}(1)}
\frac{d\xx}{\|\nu_{d,n}(\xx)\|}
\end{align}
with a homogeneous coordinate $\xii_{\infty}\in\R^{n+1}$ of $\xi_{\infty}$
\textup{(}for other notation, see \cref{sec:notation}\textup{)}
and the implicit constant depends only on $d,n$.
In particular, we have
\begin{align}
    \widetilde{C}_{d,n}(\xi,\sigma) \asymp \frac{\s_\infty^n \f(q)}{J_{n+1}(q)}
    \asymp  \frac{\f(q)}{q} \frac{1}{\q^n}
\end{align}
where the implicit constant depends only on $d,n$.
\end{theorem}

\begin{remark}
    For the last assertion, use \cref{lem:W_bound} to see
    $\mathfrak{W}_{d,n}(\xii_{\infty},\realeps_{\infty}) \asymp \s_\infty^n$.
\end{remark}

\begin{remark}
In \cref{thm:first_moment_local_cond} below,
we indeed have a better error term estimate than \cref{mainthm:numerator}.
We simplified the error term in \cref{mainthm:numerator}
by making some error terms larger keeping the admissible range of $A,B,\q$.
\end{remark}

When a hypersurface $V$ has no adelic points $x \in \prod_{p \in M_{\Q}} V(\Q_{p})$
that satisfies the same conditions in \cref{def-of-NV} except the height bound, 
we automatically have $N_{V}(B; \xi, \s) = 0$.
From this viewpoint, it is natural to consider the average of $N_{V}(B; \xi, \s)$
over the following set of hypersurfaces rather than $ \V_{d,n}$:
\begin{align}
\V_{d,n}^{\loc}(\xi, \s) \coloneqq \left\{ V \in   \V_{d,n} \ \middle|\ 
\parbox{20em}{$\forall p \in M_\Q, \exists \eta \in V(\Q_{p})$ such that $d_{p}(\eta, \xi_{p}) \leq \sigma_{p}$}  \right\}.
\end{align}
For $A\ge1$,
we also introduce the set of such hypersurfaces with height $\le A$:
\[
\V_{d,n}^{\loc}(A;\xi,\s)
\coloneqq
\V_{d,n}^{\loc}(\xi, \s)\cap \V_{d,n}(A).
\]
% We introduce the counting function
% \Yuta{This is not counting function itself though?} of such hypersurfaces: for $A \geq 1$, we set
% \begin{align}
% \V_{d,n}^{\loc}(A; \xi, \s) := \left\{ V \in   \V_{d,n}(A) \ \middle|\ 
% \txt{$\forall p \in S, \exists \eta \in V(\Q_{p})$ such that $d_{p}(\eta, \xi_{p}) \leq \sigma_{p}$\\
% $\forall p \in M_{\Q} \setminus S, V(\Q_{p}) \neq  \emptyset$}  \right\}.
% \end{align}
Our second main theorem is to give an asymptotic formula of the size of this set.
To state the theorem, let us introduce some notation.
The volume of unit ball in $\R^{N}$ for $N \in \Z_{\geq 1}$ is denoted by $V_N$.

A point $a \in \P^{N_{d,n}-1}(\Q_p)$ corresponds to a homogeneous polynomial
of degree $d$ up to non-zero multiple via \cref{homog-poly-RN}.
One of such polynomial is denoted by $f_a$. Then $V_+(f_a)$ depends only on $a$.
By this, we identify 
\begin{align}
    \{\text{Hypersurfaces in $\P^n_{\Q_p}$ of degree $d$} \}= \P^{N_{d,n}-1}(\Q_p).
\end{align}

Having this identification in mind, we set
\begin{align}
Z_{p}(\xi,\s) = 
\left\{ a \in \P^{N_{d,n}-1}(\Q_{p}) \ \middle| \  \txt{$\exists \eta \in V_{+}(f_{a})(\Q_{p})$ such that\\ $d_{p}(\eta, \xi_{p}) \leq \s_{p}$}\right\}.
\end{align}
This is the locus consisting of hypersurfaces that admit a $\Q_p$-point
close to the given point $\xi_p$.
It is easy to see that this set is closed with respect to 
the strong topology coming from $\Q_p$, cf.\ \cref{subsec:desc-Zp}.
The density of this set with respect to the standard probability measure on
$\P^{N_{d,n}-1}(\Q_p)$ is denoted by $\rho_{p}(\xi,\s)$.
See \cref{def:rhops}.
\if0
\begin{definition}
    We write
    \begin{align}
        &\rho_{p}(\xi,\s) = \mu_{\P^{N_{d,n}-1}_{\Q_{p}}}(Z_{p}(\xi,\s)) 
        \quad \text{for $p \in M_{\Q} \setminus \{ \infty\}$};\\
        &\rho_{\infty}(\xi,\s) = \mu_{S^{N_{d,n}-1}}(W_{\infty}(\xi,\s)).
    \end{align}
\end{definition}
\fi

\begin{theorem}
\label{mainthm:denominator}
    Let $A \geq 1$.
    Let $\xi \in \prod_{p \in M_{\Q}} \P^n(\Q_p)$
    and $\s \in \prod_{p \in M_{\Q}}^* (0,1]$.
    Let us write
    \begin{align}
    &\#\V_{d,n}^{\loc}(A;\xi, \s) = 
    \frac{V_{ N_{d,n}}\prod_{p \in M_{\Q}} \rho_{p}(\xi, \s)}{2\zeta( N_{d,n})} A^{N_{d,n}}
    \left(1 + R(A;\xi, \s)\right).
    \end{align}
    If 
    \begin{align}\label{eq:asymp-formula-vdn-condition-intro}
        A \gg \q \quad \text{and} \quad
        A \geq q \left((\log \log 3q)(\log 2A)\right)^{\frac{1}{N_{d,n}-2}},
    \end{align}
    then we have
    \begin{align}\label{eq:asymp-formula-vdn-bound-intro}
        R(A;\xi, \s) \ll \frac{1}{\log \frac{A}{\q} \log \log \frac{A}{\q}}
        + \frac{\log\log 3A}{\frac{A}{\q}}.
    \end{align}
    Here the implicit constant in \cref{eq:asymp-formula-vdn-condition-intro}
    is an absolute constant and that of \cref{eq:asymp-formula-vdn-bound-intro}
    depends only on $d,n$.
    Moreover, the size of the coefficient of the main term is
    \begin{align}
        \frac{V_{ N_{d,n}}\prod_{p \in M_{\Q}} \rho_{p}(\xi, \s)}{2\zeta( N_{d,n})} 
    \asymp \frac{1}{\q}
    \end{align}
    where the implicit constant depends only on $d$ and $n$.
\end{theorem}

\begin{remark}
    It has been proven that positive proportion of hypersurfaces admit adelic points
    \cite[Theorem 3.6]{PoonenVoloch}.
    \cref{mainthm:denominator} can be regarded as a generalization of this result.
    The idea of the proof is similar, but we have to dealt with additional parameters
    $\xii$ and $\s$ and make implicit constants independent of them.
\end{remark}

%%%%%%%%%%%%%%%%%%%%%%%%%%%%%%%%%%%%%%%%
By combining \cref{mainthm:numerator} and \cref{mainthm:denominator},
we can obtain the following corollary on the least height of the rational point
satisfying the same condition as \cref{def-of-NV}
on almost all hypersurfaces:
%%%%%%%%%%%%%%%%%%%%%%%%%%%%%%%%%%%%%%%%
\begin{corollary}
\label{cor:least_height_weak_approximation}
Let $n\ge d\ge2$ and $(n,d)\neq(2,2)$,
$\xi \in \prod_{p \in M_{\Q}} \P^n(\Q_p)$
and $\s \in \prod_{p \in M_{\Q}}^* (0,1]$.
For $A\ge2$ and $0<\delta<1$, we have
\[
\mathfrak{M}(V;\xi,\sigma)
\coloneqq
\min\{
H(x)
\mid
x\in V(\Q)
\ \text{and}\ 
d_p(x,\xi_p)\le\realeps_p\ \text{for all $p\in S$}
\}
\ge
\delta\q^{n-1} A
\]
for all $V\in V_{d,n}^{\loc}(A;\xi,\sigma)$
but at most $\ll\delta|\V_{d,n}^{\loc}(A;\xi,\sigma)|$ exceptions provided
\begin{equation}
\label{cor:least_height_weak_approximation:A_cond}
A\ge\frac{\q^{\max(\theta,1)+\varepsilon}}{\delta}
\end{equation}
with
\[
\theta
=\theta(n,d)
\coloneqq
\frac{n^2-1}{(2n-1)(n+1-d)-d},
\]
where the implicit constant depends only on $d,n,\varepsilon$.
\end{corollary}
%%%%%%%%%%%%%%%%%%%%%%%%%%%%%%%%%%%%%%%%
\begin{proof}
Assume $A\ge\q^{1+\varepsilon}/\delta$.
Take $B=\delta\q^{n-1}A$.
We then automatically have $B\ge\q^{n+\varepsilon}$ since $A\ge\q^{1+\varepsilon}/\delta$.
We use \cref{mainthm:numerator}.
To this end, we need to check
\begin{equation}
\label{cor:least_height_weak_approximation:A_cond_weak}
A\ge\max(B^{\frac{d+1}{(2n+1)(n+1-d)}},\q^{\frac{n-1}{2n-1}}B^{\frac{d}{(2n-1)(n+1-d)}}).
\end{equation}
Under our choice, we have
\[
A\ge B^{\frac{d+1}{(2n+1)(n+1-d)}}
\impliedby
A\ge (\q^{n-1}A)^{\frac{d+1}{(2n+1)(n+1-d)}}
\iff
A\ge\q^{\theta_{1}}
\]
with
\[
\theta_{1}
\coloneqq
\frac{(d+1)(n-1)}{(2n+1)(n+1-d)-(d+1)}.
\]
We also have
\[
A\ge\q^{\frac{n-1}{2n-1}}B^{\frac{d}{(2n-1)(n+1-d)}}
\impliedby
A\ge\q^{\frac{n-1}{2n-1}}(\q^{n-1}A)^{\frac{d}{(2n-1)(n+1-d)}}
\iff
A\ge\q^{\theta_{2}}
\]
with
\[
\theta_{2}
\coloneqq
\frac{n^2-1}{(2n-1)(n+1-d)-d}.
\]
Note that $n\ge d\ge 2$ implies
\begin{align}
\theta_{1}
&=\frac{(d+1)(n-1)}{(2n-1)(n+1-d)-d+2(n+1-d)-1}\\
&\le
\frac{(n+1)(n-1)}{(2n-1)(n+1-d)-d}
=
\theta_{2}
=
\theta
\end{align}
and so \cref{cor:least_height_weak_approximation:A_cond_weak}
is assured by \cref{cor:least_height_weak_approximation:A_cond}.
Also, by $A\ge\q$, we have
\[
B^{\frac{d+1}{(2n+1)(n+1-d)}}
\le
(\q^{n-1}A)^{\frac{d+1}{(2n+1)(n+1-d)}}
\le
A^{\frac{n(d+1)}{(2n+1)(n+1-d)}}
\le
A.
\]
By \cref{mainthm:numerator} and \cref{thm:denominator},
it suffices to show that the error terms in \cref{mainthm:numerator} are all $\ll B^{-\varepsilon}$,
which is assured by \cref{cor:least_height_weak_approximation:A_cond} again.
% The error terms
% $\q A^{-1}$, $\q B^{-\frac{1}{n}}$
% are automatically $\ll1$. Also, $B\ge\q^{n}$ imply
% \[
% \q B^{-\frac{1}{n+1-d}}
% \le
% \q B^{-\frac{1}{n}}
% \le
% 1.
% \]
% Furthermore, $A\ge\max(\q,B^{\frac{1}{n(n+1-d)}})$ implies
% \begin{align}
% \q^{\frac{2}{2n+1}}A^{-1}B^{\frac{1}{(2n+1)(n+1-d)}}
% &\le
% \q^{\frac{2}{2n+1}}
% A^{-\frac{n+1}{2n+1}}
% \le
% A^{-\frac{n-1}{2n+1}}
% \le
% 1.
% \end{align}
% We thus check the other two error terms.
% We have
% \[
% A^{-(2n+1)}B^{\frac{d+1}{n+1-d}}
% \le
% \q^{\frac{d(n-1)}{n+1-d}}A^{-(2n+1)+\frac{d+1}{n+1-d}}
% \]
% which is $\le1$ if and only if
% \[
% A\ge\q^{\theta_{1}+\varepsilon}
% \quad\text{with}\quad
% \theta_{1}
% \coloneqq
% \frac{(d+1)(n-1)}{(2n+1)(n+1-d)-(d+1)}.
% \]
% Also, we have
% \[
% \q^{n-1}A^{-(2n-1)}B^{\frac{d}{n+1-d}}
% \le
% \q^{\frac{n^2-1}{n+1-d}}A^{-(2n-1)+\frac{d}{n+1-d}}
% \]
% which is $\le1$ if and only if
% \[
% A\ge\q^{\theta_{2}+\varepsilon}
% \quad\text{with}\quad
% \theta_{2}
% \coloneqq
% \frac{n^2-1}{(2n-1)(n+1-d)-d}.
% \]
% Note that $n\ge d\ge 2$ implies
% \begin{align}
% \theta_{1}
% &=\frac{(d+1)(n-1)}{(2n-1)(n+1-d)-d+2(n+1-d)-1}\\
% &\le
% \frac{(n+1)(n-1)}{(2n-1)(n+1-d)-d}
% =
% \theta_{2}
% =
% \theta
% \end{align}
% and so the first part of the assertion holds.
\end{proof}

%%%%%%%%%%%%%%%%%%%%%%%%%%%%%%%%%%%%%%%%
\begin{remark}
\label{rem:least_height_weak_approximation}
In \cref{cor:least_height_weak_approximation}, we have
\begin{align}
\theta\le1
\iff
n^{2}-1
\le
(2n-1)(n+1-d)-d
\iff
n\ge 2d-1.
\end{align}
Thus, the conclusion of \cref{cor:least_height_weak_approximation}
holds only assuming $A\ge\q^{1+\varepsilon}$ if $n\ge 2d-1$.
On the other hand, since $\theta(n,n)=n+1$,
the exponent $\theta(n,d)$ is not bounded with respect to $(n,d)$
without restriction.
\end{remark}

\begin{acknowledgements}
%The first author is supported by JSPS KAKENHI Grant Number JP22K13903.
This work was supported by JSPS KAKENHI Grant Numbers
JP19K23402, JP21K13772, JP22K13903.
\end{acknowledgements}

\section{Notations and conventions}
\label{sec:notation}

%%%%%%%%%%%%%%%%%%%%%%%%%%%%%%%%%%%%%%%%
Besides notations and conventions introduced in \cref{sec:intro}
and will be introduced in the latter sections,
we use the following notations and conventions.

%%%%%%%%%%%%%%%%%%%%%%%%%%%%%%%%%%%%%%%%
Throughout the paper,
$d$ and $n$ denote positive integers.
We also use $M,N$ to denote a non-negative integer,
which is used for the dimension or the rank of certain vector spaces or modules.

%%%%%%%%%%%%%%%%%%%%%%%%%%%%%%%%%%%%%%%%
Let $\mu(n)$ be the M\"obius function,
$\tau(n)$ be the number of positive divisors of an integer $n$,
$\varphi(q)$ be the Euler totient function and
\[
J_k(q)\coloneqq q^k\prod_{p\mid q}\biggl(1-\frac{1}{p^k}\biggr)
\]
be the Jordan totient function.
For integers $a,b,\ldots,c$, we denote their greatest common divisor by $(a,b,\ldots,c)$.
When confusion with tuples may occur,
we write instead $\gcd(a,b,\ldots,c)$.

%%%%%%%%%%%%%%%%%%%%%%%%%%%%%%%%%%%%%%%%
For an integer $n$ and a positive integer $q$ such that $(n,q)=1$,
we write $\overline{n}\ \mod{q}$ be the multiplicative inverse of $n\ \mod{q}$ in $\Z/q\Z$.
The symbol
\[
\astsum_{u\ \mod{q}}
\and
\astcup_{u\ \mod{q}}
\]
denote the sum and union over all reduced residues $u\ \mod{q}$.
The symbol
\[
\dsum{(d)}_{n_1,\ldots,n_r}
\]
denote the summation where the summation variable runs over suitable dyadic sequences
of the form $n_{i}=2^{k_i}a_i$ with $k_i\in\mathbb{Z}$ and a positive real number $a_i$.

%%%%%%%%%%%%%%%%%%%%%%%%%%%%%%%%%%%%%%%%
Let $V$ be a real vector space. For $S\subset V$, we let $S_{\mathbb{R}}$ be the $\mathbb{R}$-subspace of $V$ spanned by $S$
and
$\mathbb{R}S\coloneqq\{ax\mid a\in\mathbb{R}\ \text{and}\ x\in S\}$.

%%%%%%%%%%%%%%%%%%%%%%%%%%%%%%%%%%%%%%%%
Let $V$ be a real metric vector space of dimension $r$,
i.e.\ a real vector space of dimension $r$ with an inner product.
We use the following notation.
\begin{itemize}
    \item For $R >0$, we set $\mathcal{B}_V(R):= \{\xx \in V \mid \|\xx\| \leq R\}$
    \item $\vol_{V}$ is the measure corresponding to the Lebesgue measure on $\R^r$
via an isometric isomorphism $V \simeq \R^r$. This is independent of the choice of the isometry. 
    \item For $\uu\in V$ and $0\le\realeps\le1$, we let
\begin{alignat}{3}
\mathcal{C}_{V}(\uu,\realeps)
&\coloneqq
\left\{\xx\in V\midmid
\|\xx\wedge\uu\|\le\realeps\|\xx\|\|\uu\|
\right\},\\
\mathcal{C}_{V}^{\perp}(\uu,\realeps)
&\coloneqq
\left\{\xx\in V\midmid
|\langle\xx,\uu\rangle|\le\realeps\|\xx\|\|\uu\|
\right\}.
\end{alignat}
\end{itemize}
Here note the the metric on $V$ naturally induces a metric on $\bigwedge^2 V$.

We always use the Euclidean inner product $\langle\cdot,\cdot\rangle$ and the Euclidean norm $\|\cdot\|$
to define the metric vector space structure on $\R^{N}$ and its subspaces.
For $\R^N$, we simplify the above notation as follows:
\begin{align}
    &\mathcal{B}_N(R) = \mathcal{B}_{\R^N}(R) = \{\xx \in \R^N \mid \|\xx\| \leq 1\},\\
    &\text{$\vol_N$ denotes the Lebesgue measure on $\R^N$},\\
    &\mathcal{C}_{N}(\uu,\realeps) = \mathcal{C}_{\R^N}(\uu,\realeps),\\
    &\mathcal{C}_{N}^{\perp}(\uu,\realeps) = \mathcal{C}_{\R^N}^{\perp}(\uu,\realeps).
\end{align}
We set $\vol_{0}(\R^0)=1$.
Note that if $\uu \neq 0$, we have 
\begin{align}
    \mathcal{C}_{N}(\uu,\realeps) 
    = \{\xx \in \R^N \setminus \{0\} \mid d_{\infty}(\xx, \uu) \leq \s\} \cup \{0\}.
\end{align}
This is the cone with axis $\R\uu$
of angle $\theta\in[0,\frac{\pi}{2}]$ with $\sin\theta=\realeps$.

The set of primitive vectors in $\Z^N$ is denoted by $\Z^N_{\prim}$:
\[
\Z^{N}_{\prim}
\coloneqq
\{\xx=(x_1,\ldots,x_N)\in\Z^N\mid\gcd(x_1,\ldots,x_N)=1\}.
\]
For a general lattice $\Gamma\subset\mathbb{R}^{N}$, we write
\[
\Gamma_{\prim}\coloneqq\Gamma\setminus\bigcup_{d\ge2}d\Gamma,
\]
which is the set of all primitive elements of $\Gamma$.

%%%%%%%%%%%%%%%%%%%%%%%%%%%%%%%%%%%%%%%%
We use \textbf{the Veronese embedding} given by
\[
\nu_{d,n}\colon\R^{n+1}\to\R^{\mathcal{M}_{d,n}}
\semicolon(x_0,\ldots,x_n)\mapsto(M(x_0,\ldots,x_n))_{M\in\mathcal{M}_{d,n}}.
\]
Note that we have
\begin{equation}
\label{Veronese_bound}
\|\nu_{d,n}(\xx)\|
\le\|\xx\|^d
\le d!\|\nu_{d,n}(\xx)\|,
% \le d^{d}\|\nu_{d,n}(\xx)\|,
\end{equation}
which can be seen by squaring and expanding $\|\xx\|^d$.

%%%%%%%%%%%%%%%%%%%%%%%%%%%%%%%%%%%%%%%%

Let $K$ be a field.
For $\aa \in K^{n+1} \setminus \{0\}$,
we denote the point in $\P^n(K)$ with homogeneous coordinates $\aa$
as $[\aa]$.

%%%%%%%%%%%%%%%%%%%%%%%%%%%%%%%%%%%%%%%%
We use Landau's symbol ``$O$'', Hardy's symbol ``$\asymp$''
and Vinogradov's symbol ``$\ll$'' in the standard way.
The dependence of the implicit constant on $a,b,\ldots,c$ is denoted by subscript,
e.g.\ ``$O_{a,b,\ldots,c}$'', ``$\asymp_{a,b,\ldots,c}$'' or ``$\ll_{a,b,\ldots,c}$''.
If Theorem or Lemma is stated
with the phrase ``where the implicit constant depends on $a,b,\ldots,c$'',
then every implicit constant in the corresponding proof
may also depend on $a,b,\ldots,c$ unless otherwise specified.

%%%%%%%%%%%%%%%%%%%%%%%%%%%%%%%%%%%%%%%%
\section{Preliminaries on geometry of numbers}
In this section, we recall some definitions and known results in geometry of numbers.
As a standard reference,
see e.g.\ Cassels's book~\cite{Cassels:GeometryOfNumbers} or 
\cite[Chapter 12]{Davenport:DiophantineEquation}.

%%%%%%%%%%%%%%%%%%%%%%%%%%%%%%%%%%%%%%%%
\textbf{A lattice} $\Lambda\subset\R^N$ of rank $r\ge0$
is a $\Z$-submodule of $\R^N$ of the form
$\Lambda=\Z\vv_1+\cdots+\Z\vv_r$,
where $\vv_1,\ldots,\vv_r\in\R^N$
are vectors linearly independent over $\R$.
(This is equivalent to say that $\L$ is a finitely generated subgroup of $\R^N$
of rank $r$ that has no accumulation points in $\R^N$.)
For a lattice $\Lambda\subset\R^N$,
let \textbf{the determinant} $\det(\Lambda)$
be the $r$-dimensional volume 
(with respect to the induced metric on $\spanR(\L)$)
of the fundamental parallelepiped.
When $\L$ is $0$, we set $\det (\L)=1$.
We say
a lattice $\Lambda\subset\R^N$
is \textbf{integral} if $\Lambda\subset\Z^{N}$.

%%%%%%%%%%%%%%%%%%%%%%%%%%%%%%%%%%%%%%%%
For a lattice $\Lambda\subset\R^N$,
we say a sublattice $\Gamma\subset\Lambda$ is \textbf{primitive (with respect to $\Lambda$)}
if the quotient module $\Lambda/\Gamma$ is torsion-free over $\Z$.
When $\Lambda=\Z^N$, we just say $\Gamma$ is primitive.

%%%%%%%%%%%%%%%%%%%%%%%%%%%%%%%%%%%%%%%%
Consider a lattice $\Lambda\subset\R^N$ of rank $r$.
We let
\[
\lambda_i(\Lambda)\coloneqq\inf\{\lambda\in\R_{>0}
\mid\dim\spanR(\Lambda\cap\mathcal{B}_{N}(\lambda))\ge i\}
\quad(i=1,\ldots,r),
\]
which are called \textbf{the successive minima of $\Lambda$}.

%%%%%%%%%%%%%%%%%%%%%%%%%%%%%%%%%%%%%%%%
\begin{lemma}[Minkowski's second theorem]
\label{lem:Minkowski_second}
For a lattice $\Lambda\subset\R^N$ of rank $r$,
\[
\det(\Lambda)\asymp\lambda_1(\Lambda)\cdots\lambda_r(\Lambda),
\]
where the implicit constant depends only on $r$.
\end{lemma}
%%%%%%%%%%%%%%%%%%%%%%%%%%%%%%%%%%%%%%%%
\begin{proof}
See Theorem~I of Chapter VIII of \cite[p.~205]{Cassels:GeometryOfNumbers}.
\end{proof}

%%%%%%%%%%%%%%%%%%%%%%%%%%%%%%%%%%%%%%%%
For a lattice $\L \subset \R^N$, 
its $\R$-span is denoted by $\L_{\R}$.
\begin{lemma}
\label{lem:nice_basis}
For any lattice $\Lambda\subset\R^{N}$ of rank $r\ge0$,
there is a $\Z$-basis
\[
(\vv_1,\ldots,\vv_r)
\]
of $\Lambda$ satisfying
\begin{enumerate}[label=\textup{(\roman*)}]
%%%%%
\item\label{lem:nice_basis:successive_minima}
We have $\|\vv_i\|\asymp\lambda_i(\Lambda)$ for $i=1,\ldots,r$.
%%%%%
\item\label{lem:nice_basis:partial_determinant}
Consider the sublattices
\[
\Lambda_{\nu}\coloneqq\Z\vv_1+\cdots+\Z\vv_\nu\subset\Lambda
\quad\text{for $\nu=1,\ldots,r$}.
\]
We then have
\[
\det(\Lambda_\nu)\asymp\lambda_1(\Lambda)\cdots\lambda_\nu(\Lambda)\quad\text{for $\nu=1,\ldots,r$}.
\]
%%%%%
\item\label{lem:nice_basis:projection_length}
Consider the orthogonal projection
\[
\pi_{\nu}^{\perp}\colon\R^{N}
=(\Lambda_\nu)_{\R}\oplus(\Lambda_\nu)_{\R}^{\bot}\to(\Lambda_\nu)_{\R}^{\bot}
\]
for $\nu=1,\ldots,r$. We then have
\[
\|\pi_{\nu}^{\perp}(\vv_i)\|\asymp\lambda_i(\Lambda)\quad\text{for $1\le\nu<i\le r$}.
\]
%%%%%
\item\label{lem:nice_basis:quasi_orthogonal}
For any $x_1,\ldots,x_r\in\R$, we have
\[
\biggl\|\sum_{i=1}^{r}x_i\vv_i\biggr\|
\asymp
\sum_{i=1}^{r}|x_i|\|\vv_i\|.
\]
\end{enumerate}
where the implicit constants depend at most on $r$.
\end{lemma}
%%%%%%%%%%%%%%%%%%%%%%%%%%%%%%%%%%%%%%%%
\begin{proof}
This is essentially Lemma~12.3 of \cite[p.~78]{Davenport:DiophantineEquation}.
By sending $\Lambda$ through some isometry $\R\Lambda\to\R^r$
and dilating by $\det(\Lambda)^{-\frac{1}{r}}$,
we may assume $N=r$ and $\det(\Lambda)=1$.
We may further rotate $\R^r$ without loss of generality.
By Lemma~12.3 of \cite[p.~78]{Davenport:DiophantineEquation},
after a suitable rotation,
we can take a basis $(\vv_1,\ldots,\vv_r)$
such that
\begin{equation}
\label{lem:nice_basis:successive_minima:v_form}
\setlength\arraycolsep{2pt}
\begin{pmatrix*}
\vv_1^{T}\\
\vv_2^{T}\\
\vdots\\
\vv_r^{T}
\end{pmatrix*}
=
\begin{pmatrix*}
v_{11}&&&\\
v_{21}&v_{21}&&\\
\vdots&\vdots&\ddots&\\
v_{r1}&v_{r2}&\cdots&v_{rr}
\end{pmatrix*}
% \begin{pmatrix*}
% v_{11}&0&0&\cdots&0\\
% v_{21}&v_{21}&0&\cdots&0\\
% \vdots&\vdots&\vdots&\ddots&\vdots\\
% v_{r1}&v_{r2}&v_{r3}&\cdots&v_{rr}
% \end{pmatrix*}
\end{equation}
and
\begin{equation}
\label{lem:nice_basis:successive_minima:v_initial_cond}
\|\vv_i\|\asymp|v_{ii}|\asymp\lambda_i(\Lambda)
\quad\text{for $i=1,\ldots,r$}.
\end{equation}
We check the claimed conditions for such chosen basis.
\medskip

%%%%%%%%%%%%%%%%%%%%%%%%%%%%%%%%%%%%%%%%
\proofitem{lem:nice_basis:successive_minima}
Clear by \cref{lem:nice_basis:successive_minima:v_initial_cond}.
\medskip

%%%%%%%%%%%%%%%%%%%%%%%%%%%%%%%%%%%%%%%%
\proofitem{lem:nice_basis:partial_determinant}
By \cref{lem:nice_basis:successive_minima:v_form}, we have
\[
\det(\Lambda_\nu)^2
=
\setlength\arraycolsep{2pt}
\det
\begin{pmatrix*}
v_{11}&&&\\
v_{21}&v_{21}&&\\
\vdots&\vdots&\ddots&\\
v_{\nu1}&v_{\nu2}&\cdots&v_{\nu\nu}
\end{pmatrix*}^2
=(|v_{11}|\cdots|v_{\nu\nu}|)^2.
\]
By \cref{lem:nice_basis:successive_minima:v_initial_cond}, we then have
$\det(\Lambda_\nu)^2
\asymp
(\lambda_1(\Lambda)\cdots\lambda_{\nu}(\Lambda))^2$
and so \cref{lem:nice_basis:partial_determinant} holds.
\medskip

%%%%%%%%%%%%%%%%%%%%%%%%%%%%%%%%%%%%%%%%
\proofitem{lem:nice_basis:projection_length}
By \cref{lem:nice_basis:successive_minima:v_form}, we have
\[
\begin{pmatrix*}
\pi_\nu^{\perp}(\vv_{\nu+1}^{T})\\
\pi_\nu^{\perp}(\vv_{\nu+2}^{T})\\
\vdots\\
\pi_\nu^{\perp}(\vv_{r}^{T})\\
\end{pmatrix*}
=
\begin{pmatrix*}
v_{\nu+1,\nu+1}&&&\\
v_{\nu+2,\nu+1}&v_{\nu+2,\nu+2}&&\\
\vdots&\vdots&\ddots&\\
v_{r,\nu+1}&v_{r,\nu+2}&\cdots&v_{rr}
\end{pmatrix*}.
% \begin{pmatrix*}
% v_{\nu+1,\nu+1}&0&0&\cdots&0\\
% v_{\nu+2,\nu+1}&v_{\nu+2,\nu+2}&0&\cdots&0\\
% \vdots&\vdots&\vdots&\ddots&\vdots\\
% v_{r,\nu+1}&v_{r,\nu+2}&v_{r,\nu+3}&\cdots&v_{rr}
% \end{pmatrix*},
\]
%where the components are with respect to the embedding $(\R\Lambda_{\nu})^{\perp}\to\R^N$.
By \cref{lem:nice_basis:successive_minima:v_initial_cond},
we therefore have
\[
\lambda_i(\Lambda)
\ll|v_{ii}|
\ll\|\pi_\nu^{\perp}(\vv_{i})\|
\ll\|\vv_{i}\|
\ll\lambda_i(\Lambda)
\]
for $\nu<i\le r$. This completes the proof.
\medskip

%%%%%%%%%%%%%%%%%%%%%%%%%%%%%%%%%%%%%%%%
\proofitem{lem:nice_basis:quasi_orthogonal}
The inequality
\[
\biggl\|\sum_{i=1}^{r}x_i\vv_i\biggr\|
\le
\sum_{i=1}^{r}|x_i|\|\vv_i\|
\]
follows just by the triangle inequality. We can then take $i_0\in\{1,\ldots,r\}$ such that
\[
\sum_{i=1}^{r}|x_i|\|\vv_i\|
\le
r|x_{i_0}|\|\vv_{i_0}\|
\]
by taking $i_0$ with $|x_{i_0}|\|\vv_{i_0}\|=\max_{i}|x_{i}|\|\vv_{i}\|$.
By \cref{lem:nice_basis:successive_minima} and \cref{lem:nice_basis:partial_determinant}, we have
\begin{equation}
\label{lem:nice_basis:quasi_orthogonal:to_det}
\sum_{i=1}^{r}|x_i|\|\vv_i\|
\ll
|x_{i_0}|\|\vv_{i_0}\|
\ll
|x_{i_0}|\frac{\det(\Lambda)}{\prod_{\substack{1\le i\le r\\i\neq i_0}}\lambda_i(\Lambda)}.
\end{equation}
Decompose $\vv_{i_0}$ as
\begin{equation}
\label{lem:nice_basis:quasi_orthogonal:decomp}
\vv_{i_0}=\widetilde{\vv}_{i_0}+\widetilde{\vv}_{i_0}^{\perp}
\quad\text{with}\quad
\widetilde{\vv}_{i_0}\in\sum_{\substack{1\le i\le r\\i\neq i_0}}\mathbb{R}\vv_{i}
\ \text{and}\ 
\widetilde{\vv}_{i_0}^{\perp}\in\biggl(\sum_{\substack{1\le i\le r\\i\neq i_0}}\mathbb{R}\vv_{i}\biggr)^{\perp}
\end{equation}
By \cref{lem:nice_basis:successive_minima}, we then have
\begin{align}
\det(\Lambda)
&=\det(\mathbb{Z}\vv_{1}+\cdots+\mathbb{Z}\widetilde{\vv}_{i_0}^{\perp}+\cdots+\mathbb{Z}\vv_{r})\\
&\ll
\|\vv_{1}\|\cdots\|\widetilde{\vv}_{i_0}^{\perp}\|\cdots\|\vv_{r}\|
\ll
\prod_{\substack{1\le i\le r\\i\neq i_0}}\lambda_i(\Lambda)\cdot\|\widetilde{\vv}_{i_0}^{\perp}\|.
\end{align}
By inserting this estimate into \cref{lem:nice_basis:quasi_orthogonal:to_det}
and recalling the decomposition \cref{lem:nice_basis:quasi_orthogonal:decomp},
\begin{align}
\sum_{i=1}^{r}|x_i|\|\vv_i\|
&\ll
|x_{i_0}|\|\widetilde{\vv}_{i_0}^{\perp}\|\\
&\ll
\|x_1\vv_1+\cdots+x_{i_0}\widetilde{\vv}_{i_0}+\cdots+x_{r}\vv_{r}+x_{i_0}\widetilde{\vv_{i_0}}^{\perp}\|
=
\biggl\|\sum_{i=1}^{r}x_i\vv_i\biggr\|.
\end{align}
This completes the proof.
\end{proof}

For a non-zero integral vector $\cc\in\Z^{N}\setminus\{0\}$, we define
\[
\Lambda_{\cc}\coloneqq\{\xx\in\Z^N\mid\langle\cc,\xx\rangle=0\}.
\]
By solving the equation $\langle\cc,\xx\rangle=0$,
it is easy to see that $\Lambda_{\cc}$
is a primitive integral lattice of $\R^N$ of rank $N-1$.
For integral vectors $\cc_1,\ldots,\cc_k\in\Z^{N}$,
let $\mathcal{G}(\cc_1,\ldots,\cc_k)$ be the greatest common divisor of determinants of $k$-minors of the $N$ by $k$ matrix $(\cc_1 \cdots \cc_k)$.
We have the following:
%%%%%%%%%%%%%%%%%%%%%%%%%%%%%%%%%%%%%%%%
\begin{lemma}
\label{lem:det_Lambda_c}
For vectors $\cc_1,\ldots,\cc_k\in\Z^{N}$ linearly independent over $\R$,
we have
\[
\det(\Lambda_{\cc_1}\cap\cdots\cap\Lambda_{\cc_k})
=
\frac{\det(\Z\cc_1+\cdots+\Z\cc_k)}
{\mathcal{G}(\cc_1,\ldots,\cc_k)}.
\]
\end{lemma}
%%%%%%%%%%%%%%%%%%%%%%%%%%%%%%%%%%%%%%%%
\begin{proof}
See \cite[Lemma~4]{leBoudec}.
\end{proof}
%%%%%%%%%%%%%%%%%%%%%%%%%%%%%%%%%%%%%%%%
For  $\xx\in\Zprim^{n+1}$,
by \cref{lem:det_Lambda_c}, we easily see that
\begin{equation}
\label{det_Lambda_nu_dn_x}
\det(\Lambda_{\nu_{d,n}(\xx)})=\|\nu_{d,n}(\xx)\|
\end{equation}
since $\nu_{d,n}(\xx)\in\Zprim^{\mathcal{M}_{d,n}}$.

%%%%%%%%%%%%%%%%%%%%%%%%%%%%%%%%%%%%%%%%
For an integral vector $\cc\in\Z^{N}$ and $q\in\N$,
we also use the $\mod{q}$ analog
\[
\Lambda_{\cc}^{(q)}
\coloneqq
\{\xx\in\Z^N\mid\langle\cc,\xx\rangle\equiv0\ \mod{q}\},
\]
which is an integral lattice of rank $N$.
%%%%%%%%%%%%%%%%%%%%%%%%%%%%%%%%%%%%%%%%
\begin{lemma}
\label{lem:det_Lambda_cq}
For an integral vector $\cc\in\Z^{N}$ and $q\in\N$, 
\[
\det(\Lambda_{\cc}^{(q)})=\frac{q}{\gcd(\cc,q)}
\and
\lambda_{1}(\Lambda_{\cc}^{(q)})\le\cdots\le\lambda_{N}(\Lambda_{\cc}^{(q)})\le q.
\]
\end{lemma}
%%%%%%%%%%%%%%%%%%%%%%%%%%%%%%%%%%%%%%%%
\begin{proof}
We may assume $\cc \neq 0$.
By considering the Smith normal form,
we can find $A\in\GL_{N}(\Z)$ with
\[
\cc^{T}A\equiv
\begin{pmatrix*}
d\cdot u&0&\cdots&0
\end{pmatrix*}
\ \mod{q}
\quad\text{with}\quad
d\coloneqq\gcd(\cc,q)\ \text{and}\ 
u\in(\Z/q\Z)^{\times}
\]
We then have
\begin{align}
\Lambda_{\cc}^{(q)}
&=
\{A\xx\in\Z^N\mid
\begin{pmatrix*}
d\cdot u&0&\cdots&0
\end{pmatrix*}\xx\equiv0\ \mod{q}\}\\
&=
A\cdot\{\xx\in\Z^N\mid
\begin{pmatrix*}
d\cdot u&0&\cdots&0
\end{pmatrix*}\xx\equiv0\ \mod{q}\}\\
&=
A\cdot(\tfrac{q}{d}\Z\ee_1+\Z\ee_2+\cdots+\Z\ee_N),
\end{align}
where $\ee_1,\ldots,\ee_N$ are standard vectors. Therefore, we have
\[
\det(\Lambda_{\cc}^{(q)})^2
=\det(A)^2(\tfrac{q}{d})^2
=(\tfrac{q}{d})^2
\quad\text{and so}\quad
\det(\Lambda_{\cc}^{(q)})=\tfrac{q}{d}.
\]
For successive minima, it suffices to note that
$q\ee_1,\ldots,q\ee_N
\in q\Z^{N}\subset\Lambda_{\cc}^{(q)}\subset\Z^{N}$.
\end{proof}

%%%%%%%%%%%%%%%%%%%%%%%%%%%%%%%%%%%%%%%%
Note that
\[
x\in V(\Q)
\iff
\aa_{V}\in\Lambda_{\nu_{d,n}(\xx)}
\]
when $\xx\in\Z^{n+1}$ is a homogeneous coordinate of $x$.
Thus, when taking average over $V$, the main difficulty lies in
counting points of $\Lambda_{\nu_{d,n}(\xx)}$ and its dependence on $\xx$.
As a key tool for understanding the successive minima 
of $\Lambda_{\nu_{d,n}(\xx)}$,
we use a result of Browning, le Boudec and Sawin~\cite{BBS}.
Following Browning, le Boudec and Sawin~\cite{BBS}, we introduce the following quantity:
%%%%%%%%%%%%%%%%%%%%%%%%%%%%%%%%%%%%%%%%
\begin{definition}
For $\xx\in\Z^{n+1}$ and $1\le r\le n+1$, let
\[
\dd_r(\xx)
\coloneqq
\min\{\det(\Lambda)
\mid\text{$\Lambda\subset\Z^{n+1}$ is a rank $r$ integral lattice with $\xx\in\Lambda$}\}.
\]
\end{definition}
%%%%%%%%%%%%%%%%%%%%%%%%%%%%%%%%%%%%%%%%
\noindent
It is clear that $1\le\dd_r(\xx)\le\|\xx\|$
since we can use some $r-1$ vectors in the standard basis together with $\xx$
to generate a rank $r$ integral lattice
of determinant $\le\|\xx\|$
and the determinant of integral lattice is a positive integer.
By using the quantity $\dd_2(\xx)$,
we can bound the largest successive minima $\lambda_{N_{d,n}-1}(\Lambda_{\nu_{d,n}(\xx)})$
as follows:
%%%%%%%%%%%%%%%%%%%%%%%%%%%%%%%%%%%%%%%%
\begin{lemma}[Browning--le Boudec--Sawin]
\label{lem:BBS_rank_d2_bound}
For $\xx\in\Zprim^{n+1}$, we have
\[
\lambda_{N_{d,n}-1}(\Lambda_{\nu_{d,n}(\xx)})
\le
\min\biggl(n\frac{\|\xx\|}{\dd_2(\xx)},\|\xx\|\biggr).
\]
\end{lemma}
%%%%%%%%%%%%%%%%%%%%%%%%%%%%%%%%%%%%%%%%
\begin{proof}
See Lemma~5 of \cite{leBoudec} and Lemma~3.15 of \cite{BBS}.
\end{proof}

%%%%%%%%%%%%%%%%%%%%%%%%%%%%%%%%%%%%%%%%
Following Schmidt~\cite[Section~2]{Schmidt}, we introduce quotient lattices.
Let $\Gamma$ be a primitive submodule of a lattice $\Lambda\subset\R^N$.
Consider the orthogonal projection
$\pi\colon\R^{N}\to\Gamma_{\mathbb{R}}^{\perp}$.
By the primitivity of $\Gamma$,
this projection induces an isomorphism $\Lambda/\Gamma\cong\pi(\Lambda)$.
By this isomorphism, we identify the quotient module $\Lambda/\Gamma$
with a lattice $\pi(\Lambda)\subset\R^N$.
Such defined lattice $\Lambda/\Gamma=\pi(\Lambda)$ is called \textbf{the quotient lattice}.
Clearly, $\rank(\Lambda/\Gamma)=\rank(\Lambda)-\rank(\Gamma)$.
The determinant of the quotient lattice $\Lambda/\Gamma$
is
\begin{equation}
\label{det_quotient}
\det(\Lambda/\Gamma)=\frac{\det(\Lambda)}{\det(\Gamma)}.
\end{equation}
Note that there is no canonical way to define integrality for the sublattices
of $\G_{\R}^{\perp}$.

% Unnecessary anymore:
% %%%%%%%%%%%%%%%%%%%%%%%%%%%%%%%%%%%%%%%%
% We need the following trivial bound for the quotient lattice:
% %%%%%%%%%%%%%%%%%%%%%%%%%%%%%%%%%%%%%%%%
% \begin{lemma}
% \label{lem:quotient_smallest_successive_minima}
% For an integral primitive lattice $\Gamma\subset\R^N$ of rank $<N$, we have
% \[
% \frac{1}{\det(\Gamma)}\ll\lambda_1(\Z^N/\Gamma)\le\cdots\le\lambda_r(\Z^N/\Gamma)\le1
% \]
% where the implicit constant depends only on $N$.
% \end{lemma}
% %%%%%%%%%%%%%%%%%%%%%%%%%%%%%%%%%%%%%%%%
% \begin{proof}
% Recall that $\Z^N/\Gamma$ is the image of $\Z^N$ under the orthogonal projection
% \[
% \pi\colon\R^{N}\to\spanR(\Gamma)^{\perp}.
% \]
% Since $\|\pi(x)\|\le\|x\|$, the upper bound $\lambda_i(\Z^N/\Gamma)\le1$ is clear
% for all $i=1,\ldots,r$, where $r\coloneqq N-\rank(\Gamma)$.
% For the lower bound, note that
% \[
% \frac{1}{\det(\Gamma)}
% =\det(\Z^N/\Gamma)\ll\lambda_1(\Z^N/\Gamma)\cdots\lambda_r(\Z^N/\Gamma)
% \]
% by \cref{det_quotient} and Minkowski's second theorem.
% By using the bound $\lambda_i(\Z^N/\Gamma)\le1$ for $i=2,\ldots,r$ proven above,
% we obtain the lower bound.
% \end{proof}

%%%%%%%%%%%%%%%%%%%%%%%%%%%%%%%%%%%%%%%%
Let us conclude this section with recalling a recent result by 
Barroero and Widmer on lattice point counting ~\cite{BarroeroWidmer}.
It provides an asymptotic formula of the number of lattice points
contained in semialgebraic sets which is applicable uniformly over
lattices and semialgebraic conditions varying semialgebraically.
Note that Barroero and Widmer
work with general $o$-minimal structures containing semialgebraic sets,
but for our purpose the semialgebraic sets are sufficient.

\begin{definition}
\label{def:Vnu}
For a linear subspace $W$ of a real metric vector space $V$,
consider the orthogonal projection
$\pi_{W}\colon V\to W$.
For a semialgebraic set $A\subset V$,
let
\[
V_{\nu}(A)
=
V_{\nu,V}(A)
\coloneqq
\sup_{\substack{%
W:\text{subspace of $V$}\\
\dim W=\nu
}}
\vol_{W}(\pi_{W}(A))
\]
(This is $V'_{\nu}(A)$ of Barroero--Widmer~\cite{BarroeroWidmer}.)
\end{definition}

%%%%%%%%%%%%%%%%%%%%%%%%%%%%%%%%%%%%%%%%
\begin{lemma}[Barroero--Widmer]
\label{lem:BarroeroWidmer_original}
Consider a semialgebraic set $Z\subset\R^{M+N}$ with $M,N\in\N$
such that for any $\TT\in\R^M$, the fiber
\[
Z_{\TT}\coloneqq\{\xx\in\R^N\mid(\TT,\xx)\in Z\}
\]
is bounded.
For any $\TT\in\R^M$ and any lattice $\Lambda \subset \R^N$ of rank $N$,
we have
\[
\#(\Lambda\cap Z_{\TT})
=
\frac{\vol_{N}(Z_{\TT})}{\det(\Lambda)}
+
O\Biggl(
\sum_{0\le\nu<N}\frac{V_{\nu,\R^N}(Z_{\TT})}{\lambda_1(\Lambda)\cdots\lambda_{\nu}(\Lambda)}
\Biggr),
\]
where the implicit constant depends only on $Z$.
\end{lemma}
%%%%%%%%%%%%%%%%%%%%%%%%%%%%%%%%%%%%%%%%
\begin{proof}
This is Theorem~1.3 of \cite[p.~4936]{BarroeroWidmer}
with $\mathsf{SemiAlg}\coloneqq(\mathsf{SemiAlg}(\R^{n}))$ as the $o$-minimal structure
except we bound $V_{\nu}(Z_{\TT})$
of \cite{BarroeroWidmer} by $V_{\nu}(Z_{\TT})$ defined by \cref{def:Vnu}.
Note that $\mathsf{SemiAlg}$ is an $o$-minimal structure
by the Tarski--Seidenberg theorem (see e.g.\ \cite[Corollary~2.11, p.~37]{vandenDries}).
\end{proof}

%%%%%%%%%%%%%%%%%%%%%%%%%%%%%%%%%%%%%%%%
\begin{lemma}
\label{lem:BarroeroWidmer}
Consider a semialgebraic set $Z\subset\R^{M+N}$ with $M,N\in\N$
such that for any $\TT\in\R^M$, the fiber
$Z_{\TT}$ defined in \cref{lem:BarroeroWidmer_original} is bounded.
For any $\TT\in\R^M$ and any lattice $\Lambda \subset \R^N$ of rank $r$,
we have
\[
\#(\Lambda\cap Z_{\TT})
=
\frac{\vol_{\L_{\R}}(\L_{\R}\cap Z_{\TT})}{\det(\Lambda)}
+
O\Biggl(
\sum_{0\le\nu<r}\frac{V_{\nu,\L_{\R}}(\L_{\R}\cap Z_{\TT})}
{\lambda_1(\Lambda)\cdots\lambda_{\nu}(\Lambda)}\Biggr),
\]
where the implicit constant depends only on $Z$.
\end{lemma}
%%%%%%%%%%%%%%%%%%%%%%%%%%%%%%%%%%%%%%%%
\begin{proof}
Consider the semialgebraic set
\[
\widetilde{Z}
\coloneqq
\{(\Phi,\TT,\xx)\in\R^{Nr}\times\R^{M}\times\R^r
\mid
\Phi\in\mathrm{O}_{N,r}
\ \text{and}\ 
(\TT,\Phi(\xx))\in Z
\},
\]
where we identify $\R^{Nr}$ with the set of $N\times r$ matrices
and $\mathrm{O}_{N,r}$ is the set of $N\times r$ orthogonal matrices,
i.e.\ $N\times r$ matrices $\Phi$ satisfying $\Phi^{T}\Phi=I_r$.
Note that $\tilde{Z}$ is a semialgebraic set.
For any $(\Phi,\TT)\in\mathrm{O}_{N,r}\times\R^M$,
we have
\begin{equation}
\label{lem:BarroeroWidmer:new_fiber}
\widetilde{Z}_{(\Phi,\TT)}
=
\Phi^{-1}(Z_{\TT}),
\end{equation}
where we identify $\Phi\in\mathrm{O}_{N,r}$ with the isometry $\R^r\to\R^N\semicolon\xx\to\Phi\cdot\xx$.
Since $\Phi$ is an isometry and $Z_{\TT}$ is bounded,
$\widetilde{Z}_{(\Phi,\TT)}$ is bounded.
Thus, by taking an isometry
$\Phi_{\Lambda}\colon\R^{r}\to \L_{\R} \subset\R^N$
and noting that $\Phi_{\Lambda}^{-1}(\Lambda)$ is a full-rank lattice of $\R^{r}$
with the same determinant and successive minima as $\Lambda$,
we can use \cref{lem:BarroeroWidmer_original} to obtain
\begin{align}
\#(\Lambda\cap Z_{\TT})
&=\#(\Phi_{\Lambda}^{-1}(\Lambda)\cap\widetilde{Z}_{(\Phi_{\Lambda},\TT)})\\
&=
\frac{\vol_r(\Phi_{\Lambda}^{-1}(Z_{\TT}))}{\det(\Lambda)}
+
O\Biggl(
\sum_{0\le\nu<r}
\frac{V_{\nu,\R^r}(\Phi_{\Lambda}^{-1}(Z_{\TT}))}
{\lambda_1(\Lambda)\cdots\lambda_{\nu}(\Lambda)}\Biggr).
\end{align}
Since $\Phi_{\Lambda}$ is an isometry, we have
\begin{align}
\vol_{r}(\Phi_{\Lambda}^{-1}(Z_{\TT}))
&=\vol_{\L_{\R}}(\L_{\R}\cap Z_{\TT})
\and
V_{\nu,\R^r}(\Phi_{\Lambda}^{-1}(Z_{\TT}))
=V_{\nu,\L_{\R}}(\L_{\R}\cap Z_{\TT})
\end{align}
and so the assertion follows.
\end{proof}

%%%%%%%%%%%%%%%%%%%%%%%%%%%%%%%%%%%%%%%%
\section{Preliminary lemmas for the archimedean place}
%We next prepare an asymptotic formula
%for the number of lattice points with local conditions.
%Note that the local conditions over finite places
%can be combined into one congruence by using the Chinese remainder theorem.
%%%%%%%%%%%%%%%%%%%%%%%%%%%%%%%%%%%%%%%%
%We separate the condition for the archimedean place by using \cref{lem:BarroeroWidmer}.
%Therefore, we need a bound for the volume appearing in the error term of \cref{lem:BarroeroWidmer}.
%%%%%%%%%%%%%%%%%%%%%%%%%%%%%%%%%%%%%%%%
For various counting problem with archimedean restriction,
we use \cref{lem:BarroeroWidmer} to separate the effect of archimedean restriction.
In this section, we prepare some lemmas to estimate the volumes
which appear in the application of \cref{lem:BarroeroWidmer}.

%%%%%%%%%%%%%%%%%%%%%%%%%%%%%%%%%%%%%%%%
Let us fix
$N\in\mathbb{Z}_{\ge2}$,
$\xii\in\mathbb{R}^{N}\setminus\{0\}$ and
$\s\in(0,1]$
throughout this section.

%%%%%%%%%%%%%%%%%%%%%%%%%%%%%%%%%%%%%%%%
\begin{lemma}
\label{lem:equater_band_vol}
We have
\begin{align}
\vol_{N}(\mathcal{C}_{N}^{\perp}(\xii,\realeps)\cap \ball{N}{1})
&=
2V_{N-1}
\int_{0}^{\realeps}
\Bigl((1-h^2)^{\frac{N-1}{2}}-((\tfrac{1}{\realeps})^2-1)^{\frac{N-1}{2}}h^{N-1}\Bigr)dh\\
&=
2\biggl(\frac{N-1}{N}\biggr)V_{N-1}\realeps
\exp(O(\realeps^2))
\end{align}
and so
\[
\vol_{N}(\mathcal{C}_{N}^{\perp}(\xii,\realeps)\cap \ball{N}{1})\asymp\realeps,
\]
where the implicit constant depends only on $N$.
\end{lemma}
%%%%%%%%%%%%%%%%%%%%%%%%%%%%%%%%%%%%%%%%
\begin{proof}
Since $\xii\neq0$, we can take an orthonormal basis
\[
\ee_0,\ee_1\ldots,\ee_{N-1}
\quad\text{with}\quad
\ee_0
\coloneqq
\frac{\xii}{\|\xii\|}.
\]
Let us parametrize $\xx\in\R^{N}$ as
\begin{equation}
\label{lem:equator_band_vol:hx_parametrization}
\xx=h\ee_0+y_1\ee_1+\cdots+y_{N-1}\ee_{N-1}
\end{equation}
and write $\yy\coloneqq(y_1,\ldots,y_{N-1})$.
We then have
\begin{equation}
\label{lem:equator_band_vol:cond}
\begin{aligned}
\xx\in\mathcal{C}_{N}^{\perp}(\xii,\realeps)\cap \ball{N}{1}
&\iff
|\langle\xii,\xx\rangle|\le\realeps\|\xx\|\|\xii\|
\ \text{and}\ 
\|\xx\|\le1\\
&\iff
\biggl(\frac{|\langle\ee_0,\xx\rangle|}{\realeps}\biggr)^2\le\|\xx\|^2\le 1\\
&\iff
\biggl(\frac{h}{\realeps}\biggr)^2\le h^2+\|\yy\|^2\le 1\\
&\iff
\biggl(\biggl(\frac{1}{\realeps}\biggr)^2-1\biggr)h^2\le\|\yy\|^2\le 1-h^2.
\end{aligned}
\end{equation}
By \cref{lem:equator_band_vol:cond},
using the parametrization \cref{lem:equator_band_vol:hx_parametrization} and noticing
\[
\biggl(\biggl(\frac{1}{\realeps}\biggr)^2-1\biggr)h^2\le\|\yy\|^2\le 1-h^2
\implies
|h|\le\realeps,
\]
we have the first equality
\begin{align}
&\vol_{N}(\mathcal{C}_{N}^{\perp}(\xii,\realeps)\cap \ball{N}{1})\\
&=
2V_{N-1}
\int_{0}^{\realeps}
\Bigl((1-h^2)^{\frac{N-1}{2}}-((\tfrac{1}{\realeps})^2-1)^{\frac{N-1}{2}}h^{N-1}\Bigr)dh
\end{align}
We next prove the second equality.
% When $\realeps=1$, we trivially have
% \[
% \vol_{N}(\mathcal{C}_{N}^{\perp}(\xii,\realeps)\cap \ball{N}{1})
% =\vol_{N}(\ball{N}{1})
% =V_N
% =2\realeps\biggl(\frac{N-1}{N}\biggr)V_{N-1}
% \exp(O(\realeps^2)).
% \]
% We may thus assume $0\le\realeps<1$.
%By the first equality, we then have
By the first equality, we have
\begin{align}
&\vol_{N}(\mathcal{C}_{N}^{\perp}(\xii,\realeps)\cap \ball{N}{1})\\
&=
2V_{N-1}
\int_{0}^{\realeps}
\Bigl((1-h^2)^{\frac{N-1}{2}}
-(\tfrac{1}{\realeps})^{N-1}(1+O(\realeps^2))^{\frac{N-1}{2}}h^{N-1}\Bigr)dh\\
&=
2V_{N-1}\biggl(\realeps+O(\realeps^3)
-\frac{1}{N}\realeps(1+O(\realeps^2))\biggr)\\
&=
2\biggl(\frac{N-1}{N}\biggr)V_{N-1}\realeps(1+O(\realeps^{2})).
\end{align}
Since the assertion is trivial for $\realeps\gg1$,
this completes the proof.
\end{proof}

%%%%%%%%%%%%%%%%%%%%%%%%%%%%%%%%%%%%%%%%
\begin{lemma}
\label{lem:cone_ball_cap_vol}
We have
\[
\vol_{N}(\mathcal{C}_{N}(\xii,\realeps)\cap \ball{N}{1})
=
\frac{2}{N}V_{N-1}\realeps^{N-1}\exp(O(\sigma^2))
\]
and so
\[
\vol_{N}(\mathcal{C}_{N}(\xii,\realeps)\cap \ball{N}{1})\asymp\realeps^{N-1}
\]
where the implicit constant depends only on $N$.
\end{lemma}
%%%%%%%%%%%%%%%%%%%%%%%%%%%%%%%%%%%%%%%%
\begin{proof}
By \cref{wedge_to_innerproduct}, we have
\begin{align}
\vol_{N}\bigl(
\mathcal{C}_{N}(\xii,\realeps)\cap \ball{N}{1}
\bigr)
&=
\vol_{N}\bigl(
(
\mathbb{R}^{N}
\setminus
\mathcal{C}_{N}^{\perp}(\xii,(1-\realeps^2)^{\frac{1}{2}})
)\cap \ball{N}{1}
\bigr)\\
&=
V_N
-
\vol_{N}\bigl(\mathcal{C}_{N}^{\perp}(\xii,(1-\realeps^2)^{\frac{1}{2}})\cap\ball{N}{1}\bigr).
\end{align}
For sufficiently small $\sigma$, by \cref{lem:equater_band_vol}, we have
\begin{align}
&\vol_{N}\bigl(\mathcal{C}_{N}(\xii,\realeps)\cap\ball{N}{1}\bigr)\\
&=
V_N-
2V_{N-1}
\int_{0}^{(1-\realeps^2)^{\frac{1}{2}}}
\Bigl((1-h^2)^{\frac{N-1}{2}}-(\tfrac{1}{1-\realeps^2}-1)^{\frac{N-1}{2}}h^{N-1}\Bigr)dh\\
&=
2V_{N-1}
\int_{(1-\realeps^2)^{\frac{1}{2}}}^{1}(1-h^2)^{\frac{N-1}{2}}dh
+
2V_{N-1}
(\tfrac{1}{1-\realeps^2}-1)^{\frac{N-1}{2}}
\int_{0}^{(1-\realeps^2)^{\frac{1}{2}}}h^{N-1}dh\\
&=
\frac{2}{N}V_{N-1}\realeps^{N-1}(1-\sigma^2)^{\frac{1}{2}}+O(\sigma^{N+1})\\
&=
\frac{2}{N}V_{N-1}\realeps^{N-1}\bigl(1+O(\sigma^{2})\bigr).
\end{align}
Thus the assertion holds if $\sigma$ is small.
When $\sigma\gg1$, the assertion is then trivial.
\end{proof}

%%%%%%%%%%%%%%%%%%%%%%%%%%%%%%%%%%%%%%%%

\begin{lemma}
\label{lem:cone_ball_projection_vol}
Let $W\subset\R^{N}$ be an $\R$-subspace of $\dim W = \nu\ge1$ and
$\pi\colon\R^{N}\to W$
be the orthogonal projection. Let
\[
\realeta\coloneqq\frac{\|\pi(\xii)\|}{\|\xii\|}\in[0,1].
\]
Then, for $X \geq 0$, we have
\[
\vol_{W}\bigl(\pi(\mathcal{C}_{N}(\xii,\realeps)\cap\mathcal{B}_{N}(X))\bigr)
\ll
\biggl(\frac{\realeta}{\realeps}+1\biggr)\cdot(\realeps X)^{\nu}.
\]
In particular, 
\[
\vol_{W}\bigl(\pi(\mathcal{C}_{N}(\xii,\realeps)\cap\mathcal{B}_{N}(X))\bigr)
\ll
\frac{1}{\realeps}\cdot(\realeps X)^{\nu}.
\]
Here the implicit constants depend only on $\nu$.
\end{lemma}

\begin{proof}
%We may assume $\realeps\ge\frac{1}{2}$ since otherwise
%\[
%\biggl(\frac{\realeta}{\realeps}+1\biggr)\cdot(\realeps X)^{\nu}
%\gg
%X^{\nu}
%\gg
%\vol_{\nu}\bigl(\pi(\mathcal{B}_{N}(X))\bigr)
%\gg
%\vol_{\nu}\bigl(\pi(\mathcal{C}_{N}(\xii,\realeps)\cap\mathcal{B}_{N}(X))\bigr)
%\]
%and the assertion trivially holds. Then, we have $\realeta>0$ and so $\pi(\xii)\neq0$.
%%%%%%%%%%%%%%%%%%%%%%%%%%%%%%%%%%%%%%%%%
Take $\xx\in\mathcal{C}_{N}(\xii,\realeps)\cap\mathcal{B}_{N}(X)\setminus\{\bm{0}\}$ arbitrarily.
Decompose orthogonally as
\begin{equation}
\label{lem:cone_ball_projection_vol:decomp}
\xx=\tilde{\xx}+\tilde{\xx}^{\perp}
\quad\text{with}\quad
\tilde{\xx}\in\R\xii\and
\tilde{\xx}^{\perp}\in(\R\xii)^{\perp}
\end{equation}
so that
\[
\tilde{\xx}\coloneqq\frac{\langle\xx,\xii\rangle}{\|\xii\|^2}\xii.
\]
We then have
\[
\|\tilde{\xx}^{\perp}\|^2
=
\biggl\|\xx-\frac{\langle\xx,\xii\rangle}{\|\xii\|^2}\xii\biggr\|^2
=
\frac{\|\xx\|^2\cdot\|\xii\|^2-\langle\xx,\xii\rangle^2}{\|\xii\|^2}
=
\frac{\|\xx\wedge\xii\|^2}{\|\xx\|^2\cdot\|\xii\|^2}\cdot\|\xx\|^2.
\]
Since $\xx\in\mathcal{C}_{N}(\xii,\realeps)\cap\mathcal{B}_{N}(X)$, this gives
\begin{equation}
\label{lem:cone_ball_projection_vol:distance}
\|\tilde{\xx}^{\perp}\|\le\realeps X.
\end{equation}
We now consider two cases according to whether $\realeta=0$ or not.

%%%%%%%%%%%%%%%%%%%%%%%%%%%%%%%%%%%%%%%%
We first consider the case $\realeta=0$ and so $\pi(\xii)=0$. In this case,
by \cref{lem:cone_ball_projection_vol:decomp} and \cref{lem:cone_ball_projection_vol:distance},
\[
\|\pi(\xx)\|
\le\|\pi(\tilde{\xx})\|+\|\pi(\tilde{\xx}^{\perp})\|
=\|\pi(\tilde{\xx}^{\perp})\|
\le\|\tilde{\xx}^{\perp}\|
\le\realeps X.
\]
This shows the inclusion
$\pi(\mathcal{C}_{N}(\xii,\realeps)\cap\mathcal{B}_{N}(X))
\subset\mathcal{B}_{W}(\realeps X)$
and so
\begin{align}
\vol_{\nu}\bigl(\pi(\mathcal{C}_{N}(\xii,\realeps)\cap\mathcal{B}_{N}(X))\bigr)
&\le
\vol_{W}\bigl(\mathcal{B}_{W}(\realeps X)\bigr)
=
\vol_{\nu}\bigl(\mathcal{B}_{\nu}(\realeps X)\bigr)\\
&\ll
(\realeps X)^{\nu}
=
\biggl(\frac{\realeta}{\realeps}+1\biggr)\cdot(\realeps X)^{\nu}.
\end{align}
This proves the assertion when $\realeta=0$.

%%%%%%%%%%%%%%%%%%%%%%%%%%%%%%%%%%%%%%%%
We next consider the case $\realeta>0$.
Decompose $\pi(\xx)$ orthogonally as
\begin{equation}
\label{lem:cone_ball_projection_vol:decomp_in_W}
\pi(\xx)=\yy+\yy^{\perp}
\quad\text{with}\quad
\yy\in\R\pi(\xii)\and
\yy^{\perp}\in(\R\pi(\xii))^{\perp}.
\end{equation}
By \cref{lem:cone_ball_projection_vol:decomp} and \cref{lem:cone_ball_projection_vol:decomp_in_W}, we have
\[
\|\yy\|\le\|\pi(\xx)\|
\le\|\pi(\tilde{\xx})\|+\|\pi(\tilde{\xx}^{\perp})\|
\le\|\pi(\tilde{\xx})\|+\|\tilde{\xx}^{\perp}\|
\]
By using the definition of $\realeta$ and \cref{lem:cone_ball_projection_vol:distance}, we have
\begin{equation}
\label{lem:cone_ball_projection_vol:y_bound}
\|\yy\|
\le\realeta\|\tilde{\xx}\|+\|\tilde{\xx}^{\perp}\|
\le\realeta\|\xx\|+\|\tilde{\xx}^{\perp}\|
\le(\realeta+\realeps)X.
\end{equation}
Since \cref{lem:cone_ball_projection_vol:decomp} implies
\[
\pi(\xx)=\pi(\tilde{\xx})+\pi(\tilde{\xx}^{\perp})
\and
\pi(\tilde{\xx})\in\R\pi(\xii),
\]
by using \cref{lem:cone_ball_projection_vol:distance} and \cref{lem:cone_ball_projection_vol:decomp_in_W}, we have
\[
\|\yy^{\perp}\|^2
\le\|\pi(\tilde{\xx})-\yy\|^2+\|\yy^{\perp}\|^2
=\|\pi(\tilde{\xx}^{\perp})\|^2
\le\|\tilde{\xx}^{\perp}\|^2
\le(\realeps X)^2
\]
and so
\begin{equation}
\label{lem:cone_ball_projection_vol:y_perp_bound}
\|\yy^{\perp}\|\le\realeps X.
\end{equation}
Since the decomposition \cref{lem:cone_ball_projection_vol:decomp_in_W} is orthogonal one
and the assumption $\realeta>0$ implies $\pi(\xii)\neq 0$,
by using \cref{lem:cone_ball_projection_vol:y_perp_bound} and \cref{lem:cone_ball_projection_vol:y_bound}
we have
\begin{align}
\vol_{W}\bigl(\pi(\mathcal{C}_{N}(\xii,\realeps)\cap\mathcal{B}_{N}(X))\bigr)
&\le\int_{\|\yy\|\le(\realeta+\realeps)X}
\biggl(\int_{\|\yy^{\perp}\|\le\realeps X}d\yy^{\perp}\biggr)d\yy\\
&\ll
(\realeps X)^{\nu-1}
\int_{\|\yy\|\le(\realeta+\realeps)X}d\yy
\ll
\biggl(\frac{\realeta}{\realeps}+1\biggr)\cdot(\realeps X)^{\nu}.
\end{align}
This proves the assertion when $\realeta\neq0$.
\end{proof}

%%%%%%%%%%%%%%%%%%%%%%%%%%%%%%%%%%%%%%%%
\begin{lemma}
\label{lem:cone_intersection}
Let $W\subset\R^{N}$ be a linear subspace of dimension $\nu$. Write
\begin{equation}
\label{lem:cone_intersection:decomp}
\xii=\tilde{\xii}+\tilde{\xii}^{\perp}
\quad\textup{with}\quad
\tilde{\xii}\in W\and
\tilde{\xii}^{\perp}\in W^{\perp}.
\end{equation}
Let
\[
\realeta\coloneqq\frac{\|\tilde{\xii}\|}{\|\xii\|}\in[0,1].
\]
We then have
\begin{enumerate}[label=\textup{(\roman*)},topsep=0mm]
%%%%%
\item\label{lem:cone_intersection:degenerate}
When $0\le\realeta^2<1-\realeps^2$, we have
\[
\mathcal{C}_{N}(\xii,\realeps)\cap W
=
\{0\}.
\]
%%%%%
\item\label{lem:cone_intersection:nondegenerate}
When $1-\realeps^2\le\realeta^2\le1$, we have
\[
\mathcal{C}_{N}(\xii,\realeps)\cap W
\subset
\mathcal{C}_{W}(\tilde{\xii},\tilde{\realeps}),
\]
where $\tilde{\realeps}\in[0,1]$ is defined by
\[
\tilde{\realeps}^{\,2}
\coloneqq
\left\{
\begin{array}{>{\displaystyle}cl}
\frac{\realeps^2+\realeta^2-1}{\realeta^2}&\textup{if $\realeta>0$},\\[4mm]
1&\textup{if $\realeta=0$ (and so $\realeps=1$)},
\end{array}
\right.
\]
\end{enumerate}
In particular, we have
\[
\tilde{\xii}=0
\implies
\mathcal{C}_{N}(\xii,\realeps)\cap W
=
\left\{
\begin{array}{>{\displaystyle}cl}
W&\textup{if $\realeps=1$},\\
\{0\}&\textup{if $0\le\realeps<1$},\\
\end{array}
\right.
\]
and
\[
\tilde{\xii}\neq0
\implies
\mathcal{C}_{N}(\xii,\realeps)\cap W
\subset
\mathcal{C}_{W}(\tilde{\xii},\realeps).
\]
\end{lemma}
%%%%%%%%%%%%%%%%%%%%%%%%%%%%%%%%%%%%%%%%
\begin{proof}
Take $\xx\in\mathcal{C}_{N}(\xii,\realeps)\cap W$ arbitrarily.
By \cref{lem:cone_intersection:decomp}, we then have
\[
\langle\xx,\tilde{\xii}\rangle
=
\langle\xx,\xii\rangle.
\]
Therefore, by recalling $\xx\in\mathcal{C}_{N}(\xii,\realeps)$, we have
\begin{align}
\|\xx\wedge\tilde{\xii}\|^2
&=
\|\xx\|^2\|\tilde{\xii}\|^2-\langle\xx,\tilde{\xii}\rangle^2\\
&=
\|\xx\|^2\|\tilde{\xii}\|^2
-\langle\xx,\xii\rangle^2\\
&=
\|\xx\|^2\|\xii\|^2
-\langle\xx,\xii\rangle^2
+\|\xx\|^2\|\tilde{\xii}\|^2-\|\xx\|^2\|\xii\|^2\\
&=
\|\xx\wedge\xii\|^2
+\|\xx\|^2\|\tilde{\xii}\|^2-\|\xx\|^2\|\xii\|^2\\
&\le
\realeps^2\|\xx\|^2\|\xii\|^2
+\|\xx\|^2\|\tilde{\xii}\|^2-\|\xx\|^2\|\xii\|^2\\
&=
(\realeps^2+\realeta^2-1)\|\xx\|^2\|\xii\|^2.
\end{align}
and so
\begin{equation}
\label{lem:cone_intersection:prelim}
0\le
\|\xx\wedge\tilde{\xii}\|^2
\le
(\realeps^2+\realeta^2-1)\|\xx\|^2\|\xii\|^2.
\end{equation}
When $0\le\realeta^{2}<1-\realeps^2$,
by \cref{lem:cone_intersection:prelim} and $\|\xii\|\neq0$,
we should have $\|\xx\|=0$ and so \cref{lem:cone_intersection:degenerate} holds.
When $\realeta^2=1-\realeps^2$ and $\realeps<1$, by \cref{lem:cone_intersection:prelim},
we have $\xx\in\mathcal{C}_{W}(\tilde{\xii},\tilde{\realeps})$
since $\|\xx\wedge\tilde{\xii}\|=0$ in this case.
When $\realeta^2=1-\realeps^2$ and $\realeps=1$,
we have $\xx\in W=\mathcal{C}_{W}(\tilde{\xii},\tilde{\realeps})$
since $\tilde{\realeps}=1$ in this case.
When $1-\realeps^2<\realeta\le1$, since $\realeta>0$, by \cref{lem:cone_intersection:prelim}, we have
\[
\|\xx\wedge\tilde{\xii}\|^2
\le
\biggl(\frac{\realeps^2+\realeta^2-1}{\realeta^2}\biggr)\|\xx\|^2\|\tilde{\xii}\|^2
\]
so that $\xx\in\mathcal{C}_{W}(\tilde{\xii},\tilde{\realeps})$.
Therefore, \cref{lem:cone_intersection:nondegenerate} holds as well.
\end{proof}

\section{Lattice point counting with local conditions}

In this section, we prepare some asymptotic formulas of the number of 
lattice points satisfying given semialgebraic and congruence conditions.
The main semialgebraic condition in mind is "contained in cones such as $\cone{N}{\xii}{\sigma}$",
but for later purpose we consider arbitrary semialgebraic sets invariant under dilation.
We need formulas that are uniform over families of cones or semialgebraic sets.
Let us introduce:

\if0
We next prepare some asymptotic formulas for the lattice point counting problem with local conditions.
For later purpose, we use general ``semialgebraic families of homogeneous sets'' $Z$
in place of the archimedean cones $\cone{N}{\xii}{\sigma}$
and then give a more specific analysis for $\cone{N}{\xii}{\sigma}$.
The ``semialgebraic families of homogeneous sets'' $Z$
correspond to the semialgebraic subsets of $\P^{N-1}(\R)$
and so we define them as follows:
\fi
%%%%%%%%%%%%%%%%%%%%%%%%%%%%%%%%%%%%%%%%
\begin{definition}[Semialgebraic family of homogeneous sets]
A semialgebraic set $Z\subset\R^{M+N}$ is called
a \textdf{semialgebraic family of homogeneous sets}
if for any $\TT\in\R^{M}$,
\begin{align}
   Z_{\TT} = \{ \xx \in \R^N \mid (\TT, \xx) \in Z \} 
\end{align}
is invariant under dilation by a non-zero real number,
i.e.\ $cZ_{\TT}=Z_{\TT}$ for any $c\in\R^{\times}$.
\end{definition}

We use the following terminology.
\begin{definition}
    Let $\L$ be a free $\Z$-module of finite rank $r \geq 0$.
    For $\cc \in \L$ and $q \in \Z_{\geq 1}$, we say $\cc$ is $q$-primitive if the following holds:
    \begin{align}
        \text{$d \mid q$ and $\cc \in d \L$} \Longrightarrow d=1.
    \end{align}
\end{definition}

\begin{remark}
    Note that the following are equivalent.
    \begin{enumerate}
        \item $\cc$ is $q$-primitive.

        \item Either $\cc =0 $ and $q=1$, or 
        $\cc \neq 0$ and if $\cc = l \vv$ for some $l \in \Z$ and primitive $\vv \in \L$,
        then $(l,q) = 1$.
            
        \item Either $\cc =0 $ and $q=1$, or 
        $\cc \neq 0$ and $\left( (\L/\Z \cc)_{\rm tors}, q  \right)=1$.

        \item Either $\cc =0 $ and $q=1$, or 
        $\cc \neq 0$ and the image of $\cc$ in $\L/q\L$ can be extended to a $\Z/q\Z$-basis
        of $\L/q\L$.
    \end{enumerate}
\end{remark}

%%%%%%%%%%%%%%%%%%%%%%%%%%%%%%%%%%%%%%%%
\subsection{With general semialgebraic family of homogeneous sets}
%%%%%%%%%%%%%%%%%%%%%%%%%%%%%%%%%%%%%%%%

We use the following notation.
\begin{definition}
    Let $\L \subset \G \subset \R^N$ be lattices.
    For $\cc \in \G$, $q \in \Z_{\geq 1}$, and a semialgebraic set $Z \subset \R^N$, 
    we set
    \begin{align}
        &\mathfrak{V}(\cc,q;Z) = \mathfrak{V}(\L, \G; \cc,q;Z)
        =\vol_{\L_{\R}}\bigl(\spanR(\Lambda\cap(\cc+q\Gamma))\cap Z\cap\ball{N}{1}\bigr)\\
        &\mathfrak{V}_{\nu}(\cc,q;Z) = \mathfrak{V}_{\nu}(\L, \G;\cc,q;Z) =
        V_{\nu,\L_{\R}}
        (\spanR(\Lambda\cap(\cc+q\Gamma))\cap Z \cap\ball{N}{1}).
    \end{align}
    (see \cref{def:Vnu} for the definition of $V_{\nu,\L_{\R}}$.)
    Note that $\spanR(\Lambda\cap(\cc+q\Gamma))$ is either $\L_{\R}$ or $0$
    according to $\Lambda\cap(\cc+q\Gamma)$ is empty or not.
    Thus if $\Lambda\cap(\cc+q\Gamma) \neq \emptyset$, 
    \begin{align}
        &\mathfrak{V}(\cc,q;Z) = \vol_{\L_{\R}} (Z \cap \ball{\L_{\R}}{1})\\
        &\mathfrak{V}_{\nu}(\cc,q;Z) = V_{\nu, \L_{\R}}(Z \cap \ball{\L_{\R}}{1}).
    \end{align}
\end{definition}

\begin{lemma}
\label{lem:lattice_count_local_prelim_general}
Consider a semialgebraic family of homogeneous sets $Z\subset\R^{M+N}$ with $M,N\in\N$.
Consider a lattice $\Gamma\subset\R^N$
and its primitive sublattice $\Lambda$ of rank $r\ge0$.
For $X\ge0$, $\TT\in\R^{M}$, $\cc\in\Gamma$ and $q\in\N$, we have
\begin{align}
&\#\bigl(\Lambda
\cap(\cc+q\Gamma)\cap Z_{\TT}\cap\ball{N}{X}\bigr)\\
&=
\frac{\mathfrak{V}(\cc,q;Z_{\TT})}{q^{r}}
\frac{X^{r}}{\det(\Lambda)}
+O\biggl(
\sum_{1\le\nu<r}
\frac{\mathfrak{V}_{\nu}(\cc,q;Z_{\TT})}{q^{\nu}}
\frac{X^{\nu}}{\lambda_1(\Lambda)\cdots\lambda_{\nu}(\Lambda)}
+\mathfrak{Q}+\mathfrak{R}\biggr),
\end{align}
where
\begin{align}
\mathfrak{Q}
&\coloneqq
\mathbbm{1}_{X\ge q\lambda_1(\Lambda)},\\
\mathfrak{R}
&\coloneqq
\mathbbm{1}_{\Lambda\cap(\cc+q\Gamma)\cap Z_{\TT}\cap\ball{N}{X}\neq\emptyset}
+
\mathbbm{1}_{X<q\lambda_1(\Lambda)}\times
\frac{\mathfrak{V}(\cc,q;Z_{\TT})}{q^{r}}
\frac{X^{r}}{\det(\Lambda)}
% \sup_{\substack{\UU\in\R^{M}\\\diam(Z_{\UU})\le q\lambda_1(\Lambda)}}
% \frac{\mathfrak{V}(\cc,q;Z_{\UU})}{q^{r}\det(\Lambda)}
\end{align}
%so $\mathfrak{R}\ll1$
and the implicit constant depends only on $Z$ and $r$.
\end{lemma}
%%%%%%%%%%%%%%%%%%%%%%%%%%%%%%%%%%%%%%%%
\begin{proof}
We may assume $\Lambda\cap(\cc+q\Gamma)\neq\emptyset$
since otherwise the assertion is trivial.
Since we can shift $\cc$ by the elements of $q\Gamma$
without changing $\cc+q\Gamma$,
we may assume $\cc\in\Lambda$ without loss of generality.
For any $\xx\in\Lambda\cap(\cc+q\Gamma)$, we then have
\[
\tfrac{1}{q}(\xx-\cc)\in(\tfrac{1}{q}\Lambda)\cap\Gamma=\Lambda
\]
by the primitivity of $\Lambda$. Thus, we have
\[
\Lambda\cap(\cc+q\Gamma)
=\cc+q\Lambda.
\]
For simplicity, write
\[
\mathcal{N}\coloneqq
\#\bigl(\Lambda
\cap(\cc+q\Gamma)\cap Z_{\TT}\cap\ball{N}{X}\bigr)
=
\#\bigl((\cc+q\Lambda)\cap Z_{\TT}\cap\ball{N}{X}\bigr).
\]

%%%%%%%%%%%%%%%%%%%%%%%%%%%%%%%%%%%%%%%%
We first consider the case $X\ge q\lambda_1(\Lambda)$.
We have
\begin{equation}
\label{lem:lattice_count_local_prelim_general:rephrase1}
\begin{aligned}
\mathcal{N}
=
\#\bigl((\cc+q\Lambda)\cap Z_{\TT}\cap\ball{N}{X}\bigr)
=
\#\bigl(q\Lambda\cap(Z_{\TT}\cap\ball{N}{X}-\cc\bigr)\bigr).
\end{aligned}
\end{equation}
Therefore, by considering the set
\[
\widetilde{Z}\coloneqq\left\{(\cc,\TT,X,\xx)
\in\R^{N}\times\R^{M}\times\R\times\R^{N}
\midmid
\xx\in Z_{\TT}\cap\ball{N}{X}-\cc
\right\},
\]
which is semialgebraic by the Tarski--Seidenberg theorem, we have
\begin{equation}
\label{lem:lattice_count_local_prelim_general:rephrase2}
\widetilde{Z}_{(\cc,\TT,X)}
=
Z_{\TT}\cap\ball{N}{X}-\cc
\quad\text{and so}\quad
\mathcal{N}
=
\#(q\Lambda\cap\widetilde{Z}_{(\cc,\TT,X)}),
\end{equation}
where $\widetilde{Z}_{(\cc,\TT,X)}$ is as in \cref{lem:BarroeroWidmer}.
Also, for any $(\cc,\TT,X)\in\R^{N+M+1}$,
$\widetilde{Z}_{(\cc,\TT,X)}$ is bounded by \cref{lem:lattice_count_local_prelim_general:rephrase2}.
Thus,
by applying \cref{lem:BarroeroWidmer} to \cref{lem:lattice_count_local_prelim_general:rephrase2}
with noting that
\[
\det(q\Lambda)=q^{r}\det(\Lambda)
\and
\lambda_i(q\Lambda)=q\lambda_i(\Lambda)\quad\text{for $i=1,\ldots,r$},
\]
we obtain
\begin{equation}
\label{lem:lattice_count_local_prelim_general:BarroeroWidmer}
\begin{aligned}
\mathcal{N}
&=
\frac{\vol_{r}(\L_{\R}\cap \widetilde{Z}_{(\cc,\TT,X)})}{q^r\det(\Lambda)}\\
&\hspace{20mm}+
O\Biggl(\sum_{1\le\nu<r}
\frac{V_{\nu,\L_{\R}}(\L_{\R}\cap \widetilde{Z}_{(\cc,\TT,X)})}
{q^{\nu}\lambda_1(\Lambda)\cdots\lambda_\nu(\Lambda)}+1\Biggr),
\end{aligned}
\end{equation}
where $+1$ corresponds to the term for $\nu=0$ of the sum in the error term.
Recall that $Z$ is a semialgebraic family of homogeneous sets.
Therefore, we have
\begin{align}
\vol_{r}(\L_{\R}\cap \widetilde{Z}_{(\cc,\TT,X)})
&=\vol_{r}(\L_{\R}\cap (Z_{\TT}\cap\ball{N}{X}-\cc))\\
&=\vol_{r}(\L_{\R}\cap (Z_{\TT}\cap\ball{N}{X}))\\
&=\vol_{r}(\L_{\R}\cap (Z_{\TT}\cap\ball{N}{1}))X^{r}\\
&=\mathfrak{V}(\cc,q;Z_{\TT})X^{r}
\end{align}
and
\begin{align}
V_{\nu,\L_{\R}}(\L_{\R}\cap \widetilde{Z}_{(\cc,\TT,X)})
&=V_{\nu,\L_{\R}}(\L_{\R}\cap (Z_{\TT}\cap\ball{N}{X}-\cc))\\
&=V_{\nu,\L_{\R}}(\L_{\R}\cap (Z_{\TT}\cap\ball{N}{X}))\\
&=V_{\nu,\L_{\R}}(\L_{\R}\cap (Z_{\TT}\cap\ball{N}{1}))X^{\nu}\\
&=\mathfrak{V}_{\nu}(\cc,q;Z_{\TT})X^{\nu}
\end{align}
and $\mathfrak{Q}=1$ in the current case, 
we obtain the assertion if $X\ge q\lambda_1(\Lambda)$.

%%%%%%%%%%%%%%%%%%%%%%%%%%%%%%%%%%%%%%%%
We consider the remaining case $X<q\lambda_1(\Lambda)$. In this case, we have
\[
\mathcal{N}
=
\#\bigl((\cc+q\Lambda)\cap Z_{\TT}\cap\ball{N}{X}\bigr)
\ll
\mathbbm{1}_{\Lambda
\cap(\cc+q\Gamma)\cap Z_{\TT}\cap\ball{N}{X}\neq\emptyset}
\]
since two distinct elements of $\cc+q\Lambda$ are $\ge q\lambda_1(\Lambda)$ apart.
Thus, we have
\[
\mathcal{N}
+\frac{\mathfrak{V}(\cc,q;Z_{\TT})}{q^{r}}
\frac{X^{r}}{\det(\Lambda)}\ll\mathfrak{R}.
\]
and so the assertion trivially holds if $X<q\lambda_1(\Lambda)$.
\end{proof}

%%%%%%%%%%%%%%%%%%%%%%%%%%%%%%%%%%%%%%%%
\begin{lemma}
\label{lem:uc_distribution}
Consider
\begin{itemize}
%%%%%%%%%%%%%%%%%%%%%%%%%%%%%%%%%%%%%%%%
\item
A lattice $\Lambda\subset\R^{N}$ of rank $r\ge1$.
%%%%%%%%%%%%%%%%%%%%%%%%%%%%%%%%%%%%%%%%
\item
A vector $\cc\in\Lambda$ and $q\in\N$ such that
$\cc$ is $q$-primitive.
% i.e.\ for the projection $\pi\colon\Lambda\to\Lambda/q\Lambda$,
% the order of $\pi(\cc)$ in $\Lambda/q\Lambda$ is $q$.
%%%%%%%%%%%%%%%%%%%%%%%%%%%%%%%%%%%%%%%%
\item
A ball $\aa+\mathcal{B}_{N}(T)$ centered at $\aa\in\R^{N}$ with radius $T\ge0$.
\end{itemize}
If $T\le Cq\lambda_1(\Lambda)$ with some $C\ge1$, then we have
\begin{align}
% \#\{u\ \mod{q}\mid (u\cc+q\Lambda)\cap(\aa+\mathcal{B}(T))\neq\emptyset\}
% &\le
\sum_{u\ \mod{q}}
\sum_{\substack{%
\xx\in\aa+\ball{N}{T}\\
\xx\in u\cc+q\Lambda
}}1
\ll
\frac{T}{\lambda_1(\Lambda)}+1,
\end{align}
where the implicit constant depends only on $r$ and $C$.
\end{lemma}
%%%%%%%%%%%%%%%%%%%%%%%%%%%%%%%%%%%%%%%%
\begin{proof}
%The first inequality is obvious.
For simplicity, write
\[
\mathcal{N}
\coloneqq
\sum_{u\ \mod{q}}
\sum_{\substack{%
\xx\in\aa+\ball{N}{T}\\
\xx\in u\cc+q\Lambda
}}1.
\]
% It suffices to prove the second inequality
% \begin{equation}
% \label{lem:uc_distribution:goal}
% \mathcal{N}\ll\frac{T}{\lambda_1(\Lambda)}+1.
% \end{equation}
We prove the assertion by induction on the rank $r$.

%%%%%%%%%%%%%%%%%%%%%%%%%%%%%%%%%%%%%%%%
We first consider the initial case $r=1$.
Since $\cc$ is $q$-primitive,
for a given $\xx\in\mathbb{R}^N$,
there are at most one $u\ \mod{q}$ with $\xx\in u\cc+q\Lambda$.
Thus, we have
\begin{align}
\mathcal{N}
&=
\sum_{\xx\in\aa+\ball{N}{T}}
\sum_{\substack{u\ \mod{q}\\\xx\in u\cc+q\Lambda}}1
\ll
\sum_{\substack{\xx\in\aa+\mathcal{B}(T)\\\xx\in\Lambda}}1.
\end{align}
By using \cref{lem:BarroeroWidmer} here, we have the claimed bound for the case $r=1$.

%%%%%%%%%%%%%%%%%%%%%%%%%%%%%%%%%%%%%%%%
Assume that $r\ge2$ and that the assertion holds for the rank $r-1$ case.
Take a $\mathbb{Z}$-basis $(\vv_1,\ldots,\vv_r)$ of $\Lambda$ as given in \cref{lem:nice_basis}.
Write
\[
\cc=c_1\vv_1+\cdots+c_r\vv_r
\quad\text{with}\quad
c_1,\ldots,c_r\in\mathbb{Z}
\]
and
\[
\tilde{\cc}\coloneqq c_2\vv_2+\cdots+c_r\vv_r.
\]
We then have $\gcd(c_1,\ldots,c_r,q)=1$ because $\cc$ is $q$-primitive.
Also, write
\[
\xx=x_1\vv_1+\tilde{\xx}\in\Lambda,\quad
\aa=a_1\vv_1+\tilde{\aa}+\aa^{\perp}
\]
with
\[
x_1\in\mathbb{Z},\quad
a_1\in\mathbb{R},\quad
\tilde{\xx}\in\tilde{\Lambda},\quad
\tilde{\aa}\in\tilde{\Lambda}_{\mathbb{R}},\quad
\tilde{\Lambda}\coloneqq\mathbb{Z}\vv_2+\cdots+\mathbb{Z}\vv_r
\and
\aa^{\perp}\in\Lambda_{\mathbb{R}}^{\perp}.
\]
By \cref{lem:nice_basis:successive_minima} and \cref{lem:nice_basis:quasi_orthogonal}of \cref{lem:nice_basis},
if $\xx\in\aa+\mathcal{B}_{N}(T)$, we have
\[
T
\ge\|\xx-\aa\|
\ge\|(x_1-a_1)\vv_1+(\tilde{\xx}-\tilde{\aa})\|
\gg|x_1-a_1|\|\vv_1\|
\gg|x_1-a_1|\lambda_{1}(\Lambda)
\]
and so
\[
|x_1-a_1|\le C_1U
\quad\text{with}\quad
U\coloneqq\frac{T}{\lambda_1(\Lambda)}
\]
for some $C_1\in\mathbb{R}$ with $1\le C_1\ll 1$.
Note that $a_1$ depends only on $\aa$
and the choice of the $\mathbb{Z}$-basis $(\vv_1,\ldots,\vv_r)$
and so independent of $\xx$ and $T$.
We then have
\begin{align}
\mathcal{N}
&\le
\sum_{u\ \mod{q}}
\sum_{\substack{|x_1-a_1|\le C_1U\\x_1\equiv uc_1\ \mod{q}}}
\sum_{\substack{%
\yy\in-x_1\vv_1+\aa+\ball{N}{T}\\
\yy\in u\tilde{\cc}+q\tilde{\Lambda}
}}1\\
&=
\sum_{u\ \mod{q}}
\sum_{\substack{|\xi_1-\frac{a_1}{(c_1,q)}|\le\frac{C_1U}{(c_1,q)}\\\xi_1\equiv u\frac{c_1}{(c_1,q)}\ \mod{\frac{q}{(c_1,q)}}}}
\sum_{\substack{%
\yy\in-(c_1,q)\xi_1\vv_1+\aa+\ball{N}{T}\\
\yy\in u\tilde{\cc}+q\tilde{\Lambda}
}}1\\
&=
\sum_{|\xi_1-\frac{a_1}{(c_1,q)}|\le\frac{C_1U}{(c_1,q)}}
\sum_{\substack{u\ \mod{q}\\\xi_1\equiv u\frac{c_1}{(c_1,q)}\ \mod{\frac{q}{(c_1,q)}}}}
\sum_{\substack{%
\yy\in-(c_1,q)\xi_1\vv_1+\aa+\ball{N}{T}\\
\yy\in u\tilde{\cc}+q\tilde{\Lambda}
}}1.
\end{align}
We now write
\[
u=\xi_1\overline{\tfrac{c_1}{(c_1,q)}}+\tfrac{q}{(c_1,q)}v\quad\text{with}\quad v\in\mathbb{Z}/(c_1,q)\mathbb{Z}
\]
and
\[
\yy=\xi_1\overline{\tfrac{c_1}{(c_1,q)}}\tilde{\cc}+\zz\quad\text{with}\quad\zz\in\tfrac{q}{(c_1,q)}v\tilde{\cc}+q\tilde{\Lambda},
\]
where $\overline{\frac{c_1}{(c_1,q)}}$ is
the multiplicative inverse of  $\frac{c_1}{(c_1,q)}\ \mod{\frac{q}{(c_1,q)}}$. We then have
\[
\mathcal{N}
\le
\sum_{|\xi_1-\frac{a_1}{(c_1,q)}|\le\frac{C_1U}{(c_1,q)}}
\sum_{v\ \mod{(c_1,q)}}
\sum_{\substack{%
\zz\in-\xi_1\overline{\frac{c_1}{(c_1,q)}}\tilde{\cc}-(c_1,q)\xi_1\vv_1+\aa+\ball{N}{T}\\
\zz\in\frac{q}{(c_1,q)}v\tilde{\cc}+q\tilde{\Lambda}
}}1.
\]
Thus, we can further write
\[
\zz=\tfrac{q}{(c_1,q)}\widetilde{\xx}\quad\text{with}\quad\widetilde{\xx}\in v\tilde{\cc}+(c_1,q)\tilde{\Lambda}
\]
and get
\begin{equation}
\label{lem:uc_distribution:prefinal}
\mathcal{N}
\le
\sum_{|\xi_1-\frac{a_1}{(c_1,q)}|\le\frac{C_1U}{(c_1,q)}}
\sum_{v\ \mod{(c_1,q)}}
\sum_{\substack{%
\tilde{\xx}\in \widetilde{a}(\xi_1)+\ball{N}{\widetilde{T}}\\
\tilde{\xx}\in v\tilde{\cc}+(c_1,q)\tilde{\Lambda}
}}1,
\end{equation}
where
\[
\widetilde{a}(\xi_1)\coloneqq\tfrac{1}{q/(c_1,q)}(-\xi_1\overline{\tfrac{c_1}{(c_1,q)}}\tilde{\cc}-(c_1,q)\xi_1\vv_1+\aa)
\and
\widetilde{T}\coloneqq\tfrac{T}{q/(c_1,q)}.
\]
By $\rank\widetilde{\Lambda}=r-1$,
$\gcd(c_2,\ldots,c_r,(c_1,q))=\gcd(c_1,\ldots,c_r,q)=1$ and $\widetilde{T}\le C\cdot(c_1,q)\lambda_1(\Lambda)\le C\cdot(c_1,q)\lambda_1(\widetilde{\Lambda})$,
we find that the two inner sum of the right-hand side of \cref{lem:uc_distribution:prefinal}
can be bounded by the induction hypothesis. Since $\lambda_1(\widetilde{\Lambda})\ge\lambda_1(\Lambda)$,
we thus have
\begin{align}
\mathcal{N}
&\ll
\biggl(\frac{T}{(c_1,q)\lambda_1(\Lambda)}+1\biggr)\biggl(\frac{T}{q/(c_1,q)\cdot\lambda_1(\Lambda)}+1\biggr)\\
&\ll
\frac{T^2}{q\lambda_1(\Lambda)^2}+\frac{T}{(c_1,q)\lambda_1(\Lambda)}+\frac{T}{q/(c_1,q)\cdot\lambda_1(\Lambda)}+1
\ll
\frac{T}{\lambda_1(\Lambda)}+1
\end{align}
since $T\ll q\lambda_1(\Lambda)$. This completes the proof.
\end{proof}

%%%%%%%%%%%%%%%%%%%%%%%%%%%%%%%%%%%%%%%%
\begin{lemma}
\label{prop:lattice_count_local_general}
Consider a semialgebraic family of homogeneous sets $Z\subset\R^{M+N}$ with $M,N\in\N$.
Consider a lattice $\Gamma\subset\R^N$
and its primitive sublattice $\Lambda$ of rank $r\ge2$.
Let $\cc\in\Gamma$ and $q\in\N$ such that $\cc$ is $q$-primitive in $\G$.
For $X\ge0$ and $\TT\in\R$, we have
\begin{align}
&\astsum_{u\ \mod{q}}
\#\bigl(\Lambda\cap(u\cc+q\Gamma)\cap Z_{\TT}\cap\ball{N}{X}\bigr)\\
&=
\frac{\varphi(q)}{q^{r}}
\frac{\mathfrak{V}(\cc,q;Z_{\TT})X^{r}}{\det(\Lambda)}
+O\biggl(
\sum_{2\le\nu<r}\frac{\varphi(q)}{q^{\nu}}
\frac{\mathfrak{V}_{\nu}(\cc,q;Z_{\TT})X^{\nu}}{\lambda_1(\Lambda)\cdots\lambda_{\nu}(\Lambda)}
+\widehat{\mathfrak{R}}+\frac{X}{\lambda_1(\Lambda)}\biggr),
\end{align}
where 
$\widehat{\mathfrak{R}}$ is defined by
\begin{align}
\widehat{\mathfrak{R}}
&\coloneqq
\mathbbm{1}_{X<q\lambda_1(\Lambda)}\times
\frac{\varphi(q)}{q^{r}}
\frac{\mathfrak{V}(\cc,q;Z_{\TT})X^{r}}{\det(\Lambda)}
\end{align}
and the implicit constant depends only on $Z$ and $r$.
\end{lemma}
%%%%%%%%%%%%%%%%%%%%%%%%%%%%%%%%%%%%%%%%
\begin{proof}
If $\L \cap (\cc + q \G) = \emptyset$, then
$\L \cap (u\cc + q \G) = \emptyset$ for all $u \in (\Z/q\Z)^{\times}$ and
in this case, there is nothing to prove.
Suppose $\L \cap (\cc + q \G) \neq \emptyset$.
Then we first replace $\cc$ so that $\cc \in \L$ as in \cref{lem:lattice_count_local_prelim_general}.
After this replacement, $\cc$ is $q$-primitive in $\L$.
We use \cref{lem:lattice_count_local_prelim_general} with $\cc$ replaced by $u\cc$
and then sum up the resulting formula over $u\in(\Z/q\Z)^{\times}$.
Note that $\mathfrak{V}(u\cc,q;Z_{\TT})=\mathfrak{V}(\cc,q;Z_{\TT})$ for any $u\in(\mathbb{Z}/q\mathbb{Z})^{\times}$.
The main term is just multiplied by $\varphi(q)$
and all but the last two error terms $\mathfrak{Q},\mathfrak{R}$ are multiplied by $\varphi(q)$.
It thus suffices to show
\begin{equation}
\label{prop:lattice_count_local_general:Q_bound}
\varphi(q)\mathfrak{Q}\ll\frac{X}{\lambda_1(\Lambda)}
\end{equation}
and
\begin{equation}
\label{prop:lattice_count_local_general:R_bound}
E\coloneqq
\astsum_{u\ \mod{q}}
\mathbbm{1}_{\Lambda\cap(u\cc+q\Gamma)\cap Z_{\TT}\cap\ball{N}{X}\neq\emptyset}
\ll
\frac{X}{\lambda_1(\Lambda)}+1
\end{equation}
The estimate \cref{prop:lattice_count_local_general:Q_bound}
is clear by checking two cases $X\ge q\lambda_1(\Lambda)$
and $X<q\lambda_1(\Lambda)$ separately.
We thus prove \cref{prop:lattice_count_local_general:R_bound}.
If $X\ge q\lambda_1(\Lambda)$, we have
\[
E
\le q
\le
\frac{X}{\lambda_1(\Lambda)}.
\]
If $\lambda_1(\Lambda)\le X<q\lambda_1(\Lambda)$,
%tby taking a point $\aa\in Z_{\TT}$,
we have
\[
E
\le
\astsum_{u\ \mod{q}}
\sum_{\substack{\xx\in\ball{N}{X}\\\xx\in u\cc+q\Lambda}}1.
\]
Thus, by \cref{lem:uc_distribution}, we have
\[
E
\ll
\frac{X}{\lambda_1(\Lambda)}+1.
\]
If $X<\lambda_1(\Lambda)$, we have
\[
E
\le
\#\bigl((\Lambda\cap\mathcal{B}_N(X))\setminus\{0\}\bigr)
=0
\]
by the definition of $\lambda_1(\Lambda)$. This completes the proof.
\end{proof}

%%%%%%%%%%%%%%%%%%%%%%%%%%%%%%%%%%%%%%%%
We also need an asymptotic formula
for the number of the primitive vectors in $\Z^{n+1}$.
The next lemma generalize Lemma~3 of le Boudec~\cite[p.~663]{leBoudec}
with local conditions
and a minor relaxing on the condition of the rank of lattice
and a minor change of the error term estimate.

%%%%%%%%%%%%%%%%%%%%%%%%%%%%%%%%%%%%%%%%
\begin{proposition}
\label{prop:lattice_count_primitive_local_general}
Consider a semialgebraic family of homogeneous sets $Z\subset\R^{M+N}$ with $M,N\in\N$.
Consider a lattice $\Gamma\subset\R^N$
and its primitive sublattice $\Lambda$ of rank $r\ge2$.
% Consider a primitive integral lattice $\Lambda\subset\R^N$ of rank $r\ge2$.
Let $\cc\in\Gamma$ and $q\in\N$ such that $\cc$ is $q$-primitive in $\G$.
% {\color{red}
% The condition \cref{prop:lattice_count_primitive_local_general:modq_primitive}
% is equivalent to
% \[
% (\Gamma/\mathbb{Z}\cc)_{\mathrm{tors}}\to(\Gamma/\mathbb{Z}\cc)_{\mathrm{tors}}\semicolon
% x\mapsto qx
% \quad\text{is isom.}
% \]
% or, still equivalent to
% \[
% (q,\#((\Gamma/\mathbb{Z}\cc)_{\mathrm{tors}}))=1.
% \]
% }
% \Yuta{Is it true? Consider $q=2$ and $\cc=0$.
% We need to consider somehow the $\cc=0$ case separately.}
For $X\ge0$ and $\TT\in\R$, we have
\begin{align}
&\astsum_{u\ \mod{q}}
\#\bigl(\Lambda\cap\Gamma_{\prim}\cap(u\cc+q\Gamma)\cap Z_{\TT}\cap\mathcal{B}_N(X)\bigr)\\
&=
\frac{\varphi(q)\mathfrak{V}(\cc,q;Z_{\TT})}{J_r(q)}
\frac{X^r}{\zeta(r)\det(\Lambda)}\\
&\hspace{20mm}
+O\biggl(
\sum_{1\le\nu<r}\frac{\varphi(q)}{q^{\nu}}
\frac{\mathfrak{V}_{\nu}(\cc,q;Z_{\TT})X^{\nu}}{\lambda_1(\Lambda)\cdots\lambda_{\nu}(\Lambda)}
+\frac{X}{\lambda_1(\Lambda)}\log\biggl(\frac{X}{\lambda_1(\Lambda)}+2\biggr)\biggr),
\end{align}
where the implicit constant depends only on $Z$ and $r$.
\end{proposition}
%%%%%%%%%%%%%%%%%%%%%%%%%%%%%%%%%%%%%%%%
\begin{proof}
As in the proof of \cref{prop:lattice_count_local_general},
we may assume $\cc\in\Lambda$ and so $\Lambda\cap(u\cc+q\Gamma)=u\cc+q\Lambda$.
For brevity, let us write
\[
\mathcal{N}\coloneqq
\astsum_{u\ \mod{q}}
\#\bigl((u\cc+q\Lambda)\cap\Gamma_{\prim}\cap Z_{\TT}\cap\mathcal{B}_N(X)\bigr).
\]
We have
\begin{equation}
\label{prop:lattice_count_primitive_local_general:pre_first_step}
\mathcal{N}
=
\astsum_{u\ \mod{q}}
\sum_{\substack{%
\xx\in u\cc+q\Lambda\\
\xx\in\Gamma_{\prim}\\
\xx\in Z_{\TT}\\
\|\xx\|\le X
}}1
=
\astsum_{u\ \mod{q}}
\sum_{\substack{%
\xx\in u\cc+q\Lambda\\
\xx\in Z_{\TT}\\
0<\|\xx\|\le X}}
\sum_{\substack{\ell\in\mathbb{N}\\\xx\in\ell\Gamma}}\mu(\ell).
\end{equation}
By the primitivity of $\Lambda$ with respect to $\Gamma$, we have
\[
\xx\in u\cc+q\Lambda\subset\Lambda\ \text{and}\ \xx\in\ell\Gamma
\implies
\tfrac{1}{\ell}\xx\in(\tfrac{1}{\ell}\Lambda)\cap\Gamma=\Lambda
\]
and so
\begin{align}
\xx\in u\cc+q\Lambda\subset\Lambda,\ 0<\|\xx\|\le X\ \text{and}\ \xx\in\ell\Gamma
&\implies
0<\|\xx\|\le X\ \text{and}\ \xx\in\ell\Lambda\\
&\implies
\ell\lambda_1(\Lambda)
=\lambda_1(\ell\Lambda)\le\|\xx\|\le X.
\end{align}
Thus, by writing $\xx=\ell\yy$ with $\yy\in\Lambda$
in \cref{prop:lattice_count_primitive_local_general:pre_first_step},
we further have
\begin{equation}
\label{prop:lattice_count_primitive_local_general:first_step}
\mathcal{N}
=\sum_{\ell\le\frac{X}{\lambda_1(\Lambda)}}\mu(\ell)
\astsum_{u\ \mod{q}}
\sum_{\substack{%
\yy\in\Lambda\\
\ell\yy\in u\cc+q\Lambda\\
\ell\yy\in Z_{\TT}\\
0<\|\yy\|\le\frac{X}{\ell}}}1.
\end{equation}
Since $Z$ is a semialgebraic family of homogeneous sets,
we have $\ell\yy\in Z_{\TT}$ iff $\yy\in Z_{\TT}$.
Since $\cc$ is $q$-primitive in $\G$, we have
\[
(u,q)=1\ \text{and}\ \ell\yy\in u\cc+q\Lambda
\implies
(\ell,q)=1.
\]
Thus, we can rewrite the inner sum of \cref{prop:lattice_count_primitive_local_general:first_step} as
\begin{equation}
\label{prop:lattice_count_primitive_local_general:third_step}
\begin{aligned}
\mathcal{N}
&=
\sum_{\substack{\ell\le\frac{X}{\lambda_1(\Lambda)}\\(\ell,q)=1}}\mu(\ell)
\astsum_{u\ \mod{q}}
\sum_{\substack{%
\yy\in \overline{\ell}u\cc+q\Lambda\\
\yy\in Z_{\TT}\\
0<\|\yy\|\le\frac{X}{\ell}}}1.\\
&=
\sum_{\substack{\ell\le\frac{X}{\lambda_1(\Lambda)}\\(\ell,q)=1}}\mu(\ell)
\astsum_{u\ \mod{q}}
\sum_{\substack{%
\yy\in u\cc+q\Lambda\\
\yy\in Z_{\TT}\\
0<\|\yy\|\le\frac{X}{\ell}}}1\\
&=
\sum_{\substack{\ell\le\frac{X}{\lambda_1(\Lambda)}\\(\ell,q)=1}}\mu(\ell)
\astsum_{u\ \mod{q}}
\#(\Lambda\cap(u\cc+q\Gamma)\cap Z_{\TT}\cap\ball{N}{\tfrac{X}{\ell}})
+O\biggl(\frac{X}{\lambda_1(\Lambda)}\biggr)
\end{aligned}
\end{equation}
since
\[
0\in u\cc+q\Gamma
\implies
q=1
\]
by the $q$-primitivity of $\cc$.
%By \cref{prop:lattice_count_primitive_local_general:modq_primitive},
%the assumption \cref{prop:lattice_count_local_general:modq_primitive}
%of \cref{prop:lattice_count_local_general} is satisfied.
By \cref{prop:lattice_count_local_general} and
\cref{prop:lattice_count_primitive_local_general:third_step},
we have
\begin{equation}
\label{prop:lattice_count_primitive_local_general:after_lattice_point_counting}
\begin{aligned}
\mathcal{N}
&=
\frac{\varphi(q)}{q^r}
\frac{\mathfrak{V}(\cc,q;Z_{\TT})X^{r}}{\det(\Lambda)}
S_{\textup{main}}\\
&\hspace{20mm}
+O\biggl(\sum_{1\le\nu<r}E_{\nu}
+E_{\mathfrak{R}}
+\frac{X}{\lambda_1(\Lambda)}\log\biggl(\frac{X}{\lambda_1(\Lambda)}+2\biggr)\biggr),
\end{aligned}
\end{equation}
where
\begin{align}
S_{\textup{main}}
&\coloneqq
\sum_{\substack{\ell\le\frac{X}{\lambda_1(\Lambda)}\\(\ell,q)=1}}\frac{\mu(\ell)}{\ell^{r}},\\
E_{\nu}
&\coloneqq
\sum_{\ell\le\frac{X}{\lambda_1(\Lambda)}}\frac{\mu^2(\ell)}{\ell^{\nu}}
\frac{\varphi(q)}{q^{\nu}}
\frac{\mathfrak{V}_{\nu}(\cc,q;Z_{\TT})X^{\nu}}{\lambda_1(\Lambda)\cdots\lambda_{\nu}(\Lambda)},\\
E_{\mathfrak{R}}
&\coloneqq
\frac{\varphi(q)}{q^r}
\frac{\mathfrak{V}(\cc,q;Z_{\TT})X^{r}}{\det(\Lambda)}
\sum_{\frac{X}{q\lambda_1(\Lambda)}<\ell\le\frac{X}{\lambda_1(\Lambda)}}
\frac{\mu^2(\ell)}{\ell^{r}}.
\end{align}
%%%%%%%%%%%%%%%%%%%%%%%%%%%%%%%%%%%%%%%%
For $S_{\textup{main}}$, since $r\ge2$, we have
\begin{equation}
\label{prop:lattice_count_primitive_local_general:S_main_pre}
S_{\textup{main}}
=
\prod_{p\nmid q}\biggl(1-\frac{1}{p^{r}}\biggr)
+O\biggl(\biggl(\frac{\lambda_1(\Lambda)}{X}\biggr)^{r-1}\biggr)
=
\frac{q^{r}}{J_r(q)\zeta(r)}
+O\biggl(\biggl(\frac{\lambda_1(\Lambda)}{X}\biggr)^{r-1}\biggr).
\end{equation}
By using Minkowski's second theorem (\cref{lem:Minkowski_second}), we have
\begin{equation}
\label{prop:lattice_count_primitive_local_general:S_main}
\begin{aligned}
\frac{\varphi(q)}{q^r}
\frac{\mathfrak{V}(\cc,q;Z_{\TT})X^{r}}{\det(\Lambda)}
S_{\textup{main}}
=
\frac{\varphi(q)\mathfrak{V}(\cc,q;Z_{\TT})}{J_r(q)}
\frac{X^r}{\zeta(r)\det(\Lambda)}
+O\biggl(\frac{X}{\lambda_1(\Lambda)}\biggr).
\end{aligned}
\end{equation}
For $E_\nu$ with $\nu=1$, we have
\begin{equation}
\label{prop:lattice_count_primitive_local_general:E_1}
E_1
\ll
\sum_{\ell\le\frac{X}{\lambda_1(\Lambda)}}\frac{\mu^2(\ell)}{\ell}
\frac{X}{\lambda_1(\Lambda)}
\ll
\frac{X}{\lambda_1(\Lambda)}\log\biggl(\frac{X}{\lambda_1(\Lambda)}+2\biggr)
\end{equation}
For $E_\nu$ with $1<\nu<r$, we have
\begin{equation}
\label{prop:lattice_count_primitive_local_general:E_nu}
E_\nu
\ll
\frac{\varphi(q)}{q^{\nu}}
\frac{\mathfrak{V}_{\nu}(\cc,q;Z_{\TT})X^{\nu}}{\lambda_1(\Lambda)\cdots\lambda_{\nu}(\Lambda)}.
\end{equation}
For $E_{\mathfrak{R}}$, since $r\ge2$,
by Minkowski's second theorem (\cref{lem:Minkowski_second}),
we have
\begin{equation}
\label{prop:lattice_count_primitive_local_general:E_R}
E_{\mathfrak{R}}
\ll
\frac{\varphi(q)}{q^r}
\frac{\mathfrak{V}(\cc,q;Z_{\TT})X^{r}}{\det(\Lambda)}
\biggl(\frac{q\lambda_1(\Lambda)}{X}\biggr)^{r-1}
\ll
\frac{X}{\lambda_1(\Lambda)}.
\end{equation}
On inserting
\cref{prop:lattice_count_primitive_local_general:S_main},
\cref{prop:lattice_count_primitive_local_general:E_1},
\cref{prop:lattice_count_primitive_local_general:E_nu}
and \cref{prop:lattice_count_primitive_local_general:E_R}
into \cref{prop:lattice_count_primitive_local_general:after_lattice_point_counting}, we obtain the lemma.
\end{proof}

%%%%%%%%%%%%%%%%%%%%%%%%%%%%%%%%%%%%%%%%
\subsection{With the intersection \texorpdfstring{$\cone{N}{\xii}{\sigma}\cap\ball{N}{X}$}{cone cap ball}}
We now specialize the setting of the previous subsection
to the semialgebraic family of homogeneous sets
\[
Z
=
\{
((\xii,\realeps),\xx)\in\mathbb{R}^{N}\times(0,1]\times\mathbb{R}^{N}
\mid
\xx\in\mathcal{C}_{N}(\xii,\realeps)
\}.
\]
Accordingly, for lattices $\L \subset \G \subset \R^N$,
$\cc \in \G$, $q \in \Z_{\geq 1}$, 
$\xii\in\mathbb{R}^{N}\setminus\{0\}$ and $\sigma\in(0,1]$, we write
\begin{align}
\mathfrak{V}(\cc,q;\xii,\realeps)
={}&
\mathfrak{V}(\Lambda,\Gamma;\cc,q;\xii,\realeps)\\
\coloneqq{}&
\vol_{\Lambda_{\mathbb{R}}}
(\spanR(\Lambda\cap(\cc+q\G))
\cap
\mathcal{C}_{N}(\xii,\realeps)
\cap
\mathcal{B}_{N}(1)
),\\
\mathfrak{V}_{\nu}(\cc,q;\xii,\realeps)
={}&
\mathfrak{V}_{\nu}(\Lambda,\Gamma;\cc,q;\xii,\realeps)\\
\coloneqq{}&
V_{\nu,\Lambda_{\mathbb{R}}}
(\spanR(\Lambda\cap(\cc+q\G))
\cap
\mathcal{C}_{N}(\xii,\realeps)
\cap
\mathcal{B}_{N}(1)
).
\end{align}

%%%%%%%%%%%%%%%%%%%%%%%%%%%%%%%%%%%%%%%%
\begin{lemma}
\label{lem:V_bound}
Consider a lattice $\Gamma\subset\R^N$
and its primitive sublattice $\Lambda$ of rank $r\ge0$.
For $\cc\in\Gamma$, $q\in\N$,
$\xii\in\R^{N}\setminus\{0\}$, $\realeps\in(0,1]$
and $\nu=1,\ldots,N$, we have
\[
\mathfrak{V}_{\nu}(\cc,q;\xii,\realeps)
\ll
\sigma^{\nu-1},
\]
where the implicit constant depends only on $r$.
In particular, we have
\[
\mathfrak{V}(\cc,q;\xii,\realeps)
\ll
\sigma^{r-1}
\]
where the implicit constant depends only on $r$.
\end{lemma}
%%%%%%%%%%%%%%%%%%%%%%%%%%%%%%%%%%%%%%%%
\begin{proof}
We have
\[
\mathfrak{V}_{\nu}(\cc,q;\xii,\realeps)
\le
V_{\nu,\Lambda_{\mathbb{R}}}
\bigl(\Lambda_{\mathbb{R}}
\cap\mathcal{C}_{N}(\xii,\realeps)\cap\mathcal{B}_{N}(1)\bigr).
\]
Take a $\nu$-dimensional subspace $W$ of $\Lambda_{\mathbb{R}}$ arbitrarily and
consider the orthogonal projection $\pi_{W}\colon\Lambda_{\mathbb{R}}\to W$.
It suffices to bound
\begin{equation}
\vol_{W}(\pi_{W}(\Lambda_{\mathbb{R}}\cap\cone{N}{\xii}{\realeps}\cap\ball{N}{1})).
\end{equation}
In order to apply \cref{lem:cone_intersection}, consider the orthogonal decomposition
\[
\xii=\tilde{\xii}+\tilde{\xii}^{\perp}
\quad\text{with}\quad
\tilde{\xii}\in\Lambda_{\mathbb{R}}
\and
\tilde{\xii}^{\perp}\in\Lambda_{\mathbb{R}}^{\perp}.
\]
When $\tilde{\xii}=0$ and $\realeps=1$,
by \cref{lem:cone_intersection}, we have
\[
\vol_{W}(\pi_{W}(\Lambda_{\mathbb{R}}\cap\cone{N}{\xii}{\realeps}\cap\ball{N}{1}))
=\vol_{W}\bigl(\pi_{W}(\ball{W}{1})\bigr)
\ll
1
=\realeps^{\nu-1}.
\]
When $\tilde{\xii}=0$ and $0\le\realeps<1$,
by \cref{lem:cone_intersection}, we have
\[
\vol_{W}(\pi_{W}(\Lambda_{\mathbb{R}}\cap\cone{N}{\xii}{\realeps}\cap\ball{N}{1}))
=
\vol_{W}(\{0\})
=0
\ll
\realeps^{\nu-1}.
\]
When $\tilde{\xii}\neq0$,
by \cref{lem:cone_intersection}, we have
\begin{equation}
\label{prop:lattice_count_local:V_nu_pre}
\begin{aligned}
&\vol_{W}(\pi_{W}(\Lambda_{\mathbb{R}}\cap\cone{N}{\xii}{\realeps}\cap\ball{N}{1}))\\
&\le
\vol_{W}
\Bigl(
\pi_{W}\bigl(
\mathcal{C}_{\Lambda_{\mathbb{R}}}(\tilde{\xii},\realeps)\cap\mathcal{B}_{W}(1)
\bigr)
\Bigr).
\end{aligned}
\end{equation}
By \cref{lem:cone_ball_projection_vol},
we can further bound as
\[
\vol_{W}(\pi_{W}(\Lambda_{\mathbb{R}}\cap\cone{N}{\xii}{\realeps}\cap\ball{N}{1}))
\ll
\realeps^{\nu-1}.
\]
By these bounds, we obtain the lemma for all cases.
\end{proof}

%%%%%%%%%%%%%%%%%%%%%%%%%%%%%%%%%%%%%%%%
\begin{lemma}
\label{lem:P_main_bound}
Consider a lattice $\Gamma\subset\R^N$
and its primitive sublattice $\Lambda$ of rank $r\ge0$.
For $\cc\in\Z^{N}$, $q\in\N$,
$\xii\in\R^{N}\setminus\{0\}$,
$\realeps\in(0,1]$ and $X\ge0$, we have
\[
\mathbbm{1}_{X<q\lambda_1(\Lambda)}\times
\frac{\mathfrak{V}(\cc,q;\xii,\realeps)}{q^{r}}
\frac{X^{r}}{\det(\Lambda)}
\ll
\frac{1}{\realeps}\sum_{1\le\nu<r}\frac{(\frac{\realeps}{q}X)^{\nu}}{\lambda_1(\Lambda)\cdots\lambda_{\nu}(\Lambda)}
+\mathbbm{1}_{r\in\{0,1\}},
\]
where the implicit constant depends only on $r$.
\end{lemma}
%%%%%%%%%%%%%%%%%%%%%%%%%%%%%%%%%%%%%%%%
\begin{proof}
If $r=0$ or $1$, then
\[
\mathbbm{1}_{X<q\lambda_1(\Lambda)}\times
\frac{\mathfrak{V}(\cc,q;\xii,\realeps)}{q^{r}}
\frac{X^{r}}{\det(\Lambda)}
\ll
1
=\mathbbm{1}_{r\in\{0,1\}}.
\]
If $r\ge2$, then by \cref{lem:Minkowski_second} and \cref{lem:V_bound}, we have
\begin{align}
\mathbbm{1}_{X<q\lambda_1(\Lambda)}\times
\frac{\mathfrak{V}(\cc,q;\xii,\realeps)}{q^{r}}
\frac{X^{r}}{\det(\Lambda)}
&\ll
\mathbbm{1}_{X<q\lambda_1(\Lambda)}\times
\frac{1}{\realeps}\frac{(\frac{\realeps}{q}X)^{r-1}}{\lambda_1(\Lambda)\cdots\lambda_{r-1}(\Lambda)}\cdot\frac{X}{q\lambda_{r}(\Lambda)}\\
&\le
\frac{1}{\realeps}\frac{(\frac{\realeps}{q}X)^{r-1}}{\lambda_1(\Lambda)\cdots\lambda_{r-1}(\Lambda)}\\
&\le
\frac{1}{\realeps}\sum_{1\le\nu<r}\frac{(\frac{\realeps}{q}X)^{\nu}}{\lambda_1(\Lambda)\cdots\lambda_{\nu}(\Lambda)}.
\end{align}
This completes the proof.
\end{proof}

\begin{proposition}
\label{prop:lattice_count_local}
Consider a lattice $\Gamma\subset\R^N$
and its primitive sublattice $\Lambda$ of rank $r\ge2$.
Let $\cc\in\Gamma$ and $q\in\N$ be such that $\cc$ is $q$-primitive in $\G$.
For $\xii\in\R^{N}\setminus\{0\}$, $\realeps\in(0,1]$ and $X\ge0$,
\begin{align}
&\astsum_{u\ \mod{q}}
\#\bigl(\Lambda
\cap(u\cc+q\Gamma)\cap\mathcal{C}_{N}(\xii,\realeps)
\cap\mathcal{B}_N(X)\bigr)\\
&=
\mathfrak{V}(\cc,q;\xii,\realeps)
\frac{\varphi(q)}{q^{r}}
\frac{X^r}{\det(\Lambda)}
+O\biggl(\frac{q}{\realeps}\sum_{1\le\nu<r}\frac{(\frac{\realeps}{q}X)^{\nu}}{\lambda_1(\Lambda)\cdots\lambda_{\nu}(\Lambda)}\biggr),
\end{align}
where the implicit constant depends only on $r$.
\end{proposition}
%%%%%%%%%%%%%%%%%%%%%%%%%%%%%%%%%%%%%%%%
\begin{proof}
We use \cref{prop:lattice_count_local_general}
with a semialgebraic family of homogeneous sets
\[
Z\coloneqq
\{((\xii,\realeps),\xx)\in\mathbb{R}^{N}\times(0,1]\times\times\mathbb{R}^{N}
\mid \xx\in\cone{N}{\xii}{\realeps}\}
\]
and $\TT=(\xii,\realeps)$.
For the main term, we then have
\[
\mathfrak{V}(\cc,q;Z_{\TT})
=\mathfrak{V}(\cc,q;\xii,\realeps).
\]
For the sum in the error term, \cref{lem:V_bound} gives
\[
\mathfrak{V}_{\nu}(\cc,q;Z_{\TT})
=\mathfrak{V}_{\nu}(\cc,q;\xii,\realeps)
\ll\sigma^{\nu-1}.
\]
For the error term $X/\lambda_{1}(\Lambda)$, we have
\[
\frac{X}{\lambda_{1}(\Lambda)}
=
\frac{q}{\realeps}
\frac{\frac{\sigma}{q}X}{\lambda_{1}(\Lambda)}
\le
\frac{q}{\realeps}\sum_{1\le\nu<r}\frac{(\frac{\realeps}{q}X)^{\nu}}{\lambda_1(\Lambda)\cdots\lambda_{\nu}(\Lambda)}.
\]
For the error term$\widehat{\mathfrak{R}}$,
since $r\ge2$,
by \cref{lem:P_main_bound,lem:V_bound}, we have
\begin{align}
\mathbbm{1}_{X<q\lambda_1(\Lambda)}\times
\frac{\mathfrak{V}(\cc,q;Z_{\TT})}{q^{r}}
\frac{X^{r}}{\det(\Lambda)}
&=
\mathbbm{1}_{X<q\lambda_1(\Lambda)}\times
\frac{\mathfrak{V}(\cc,q;\xii,\realeps)}{q^{r}}
\frac{X^{r}}{\det(\Lambda)}\\
&\ll
\frac{1}{\realeps}\sum_{1\le\nu<r}\frac{(\frac{\realeps}{q}X)^{\nu}}{\lambda_1(\Lambda)\cdots\lambda_{\nu}(\Lambda)}.
\end{align}
On inserting all the above observation into \cref{prop:lattice_count_local_general},
we obtain the lemma.
\end{proof}

%%%%%%%%%%%%%%%%%%%%%%%%%%%%%%%%%%%%%%%%
\begin{proposition}
\label{prop:lattice_count_primitive_local}
Consider a lattice $\Gamma\subset\R^N$
and its primitive sublattice $\Lambda$ of rank $r\ge2$.
Let $\cc\in\Gamma$ and $q\in\N$ be such that $\cc$ is $q$-primitive in $\G$.
For $\xii\in\R^{N}\setminus\{0\}$, $\realeps\in(0,1]$
and $X\ge0$,
\begin{align}
&\astsum_{u\ \mod{q}}
\#\bigl(\Lambda\cap\Gamma_{\prim}
\cap(u\cc+q\Gamma)\cap\mathcal{C}_{N}(\xii,\realeps)
\cap\mathcal{B}_N(X)\bigr)\\
&=
\frac{\varphi(q)\mathfrak{V}(\cc,q;\xii,\realeps)}{J_r(q)}
\frac{X^r}{\zeta(r)\det(\Lambda)}\\
&\hspace{20mm}
+O\biggl(
\frac{q}{\realeps}\sum_{1\le\nu<r}\frac{(\frac{\realeps}{q}X)^{\nu}}{\lambda_1(\Lambda)\cdots\lambda_{\nu}(\Lambda)}
+\frac{X}{\lambda_1(\Lambda)}\log\biggl(\frac{X}{\lambda_1(\Lambda)}+2\biggr)\biggr),
\end{align}
where the implicit constant depends only on $r$.
% In particular, when $\Gamma=\mathbb{Z}^{N}$,
% we can use this lemma with
% any primitive integral lattice $\Lambda$
% of rank $r\ge2$ and any integral vector $\cc$ with $\gcd(\cc,q)=1$.
\end{proposition}
%%%%%%%%%%%%%%%%%%%%%%%%%%%%%%%%%%%%%%%%
\begin{proof}
We can obtain this lemma by using \cref{prop:lattice_count_primitive_local_general}
similarly to the proof of \cref{prop:lattice_count_local}
and bound the error term by using \cref{lem:V_bound}.
\end{proof}

%%%%%%%%%%%%%%%%%%%%%%%%%%%%%%%%%%%%%%%%
\section{Sums of reciprocals including the Veronese embedding}

%%%%%%%%%%%%%%%%%%%%%%%%%%%%%%%%%%%%%%%%
% \Yuta{Probably, when $\s=0$, then the left-hand side is essentially
% a one-dimensional sum
% \[
% \ll
% \frac{1}{\min(\ast)^{d}}
% +
% \sum_{n\ge 1}\frac{1}{(qn)^{d}}
% \ll
% \mathfrak{R}
% +
% \frac{1}{q},
% \]
% where
% \begin{align}
% \min(\ast)
% \coloneqq{}&
% \min
% \{\|x\|
% \mid
% \xx\in\cc+q\Z^{n+1},\ 0<\|\xx\|\le X,\ 
% [\xx]=[\xii]
% \}\\
% \ge{}&
% \min\{\|x\|
% \mid
% \xx\in
% (\cc+q\Z^{n+1})
% \cap(\mathcal{B}_{n+1}(X))\setminus\{0\})\}
% \end{align}
% and so the assertion is WORSE than this rough estimate.
% }
\begin{lemma}
\label{lem:det_reciprocal_sum}
Let $n,d$ be positive integers with $n\ge d\ge 2$.
For $\cc\in\Z^{n+1}$ and $q\in\N$ such that $\gcd(\cc,q)=1$
and for $\xii\in\R^{n+1}\setminus\{0\}$, $0\le\realeps\le1$ and $X\ge1$, we have
\begin{align}
%T_{d,n}(X;\cc,q;\xii,\realeps)
%\coloneqq{}
&\sum_{\substack{%
\xx\in\cc+q\Z^{n+1}\\
0<\|\xx\|\le X\\
\xx\in\mathcal{C}_{n+1}(\xii,\realeps)}}
\frac{1}{\|\nu_{d,n}(\xx)\|}\\
&=
\mathfrak{W}_{d,n}(\xii,\realeps)
\frac{X^{n+1-d}}{q^{n+1}}
+O\biggl(\frac{1}{\realeps}\biggl(\frac{\realeps}{q}\biggr)^{n}(X^{n-d}+\log X)
+\frac{1}{q}
+\mathfrak{R}\biggr),
\end{align}
where the coefficient $\mathfrak{W}_{d,n}(\xii,\realeps)$ is defined by
\[
\mathfrak{W}_{d,n}(\xii,\realeps)
\coloneqq
\int_{\mathcal{C}_{n+1}(\xii,\realeps)\cap\mathcal{B}_{n+1}(1)}
\frac{d\xx}{\|\nu_{d,n}(\xx)\|},
\]
the error term $\mathfrak{R}$ is defined by
\begin{align}
\mathfrak{R}
&\coloneqq
\int_{1}^{q}\mathbbm{1}_{(\cc+q\Z^{n+1})
\cap(\mathcal{B}_{n+1}(\min(X,dt^{\frac{1}{d}}))\setminus\{0\})\neq\emptyset}\frac{dt}{t^2}
\end{align}
\textup{(}so $\mathfrak{R}\ll1$\textup{)}
and the implicit constant depends only on $d,n$.
\end{lemma}
%%%%%%%%%%%%%%%%%%%%%%%%%%%%%%%%%%%%%%%%
\begin{proof}
%For simplicity, write
%\[
%S\coloneqq
%\sum_{\substack{%
%\xx\in\cc+q\Z^{n+1}\\
%0<\|\xx\|\le X\\
%\xx\in\mathcal{C}_{n+1}(\xii,\realeps)}}
%\frac{1}{\|\nu_{d,n}(\xx)\|}.
%\]
%Since the cone $\mathcal{C}_{N}(\xii,\realeps)$ is invariant under the bijective dilation
%\[
%\R^N\to\R^N\semicolon
%\xx\mapsto\frac{\xx}{q},
%\]
%by changing the summation variable, we have
%\begin{equation}
%\label{lem:det_reciprocal_sum:rephrase1}
%T_{d,n}(X)
%=
%\frac{1}{q^{d}}
%\sum_{\substack{%
%\xx\in\Z^{n+1}\\
%0<\|\frac{\cc}{q}+\xx\|\le X/q\\
%\frac{\cc}{q}+\xx\in\mathcal{C}_{n+1}(\xii,\realeps)}}
%\frac{1}{\|\nu_{d,n}(\frac{\cc}{q}+\xx)\|}.
%\end{equation}
By partial summation, we have
\begin{equation}
\label{lem:det_reciprocal_sum:partial_sum}
\begin{aligned}
&\sum_{\substack{%
\xx\in\cc+q\Z^{n+1}\\
0<\|\xx\|\le X\\
\xx\in\mathcal{C}_{n+1}(\xii,\realeps)}}
\frac{1}{\|\nu_{d,n}(\xx)\|}\\
&=
\sum_{\substack{%
\xx\in\cc+q\Z^{n+1}\\
0<\|\xx\|\le X\\
\xx\in\mathcal{C}_{n+1}(\xii,\realeps)}}
\int_{\|\nu_{d,n}(\xx)\|}^{\infty}\frac{dt}{t^2}
=
\int_{1}^{\infty}
\biggl(
\sum_{\substack{%
\xx\in\cc+q\Z^{n+1}\\
0<\|\xx\|\le X\\
\|\nu_{d,n}(\xx)\|\le t\\
\xx\in\mathcal{C}_{n+1}(\xii,\realeps)}}1
\biggr)\frac{dt}{t^2}.
\end{aligned}
\end{equation}
We have
\begin{equation}
\label{lem:det_reciprocal_sum:rephrase1}
\sum_{\substack{%
\xx\in\cc+q\Z^{n+1}\\
0<\|\xx\|\le X\\
\|\nu_{d,n}(\xx)\|\le t\\
\xx\in\mathcal{C}_{n+1}(\xii,\realeps)}}1
=
\#\left\{\xx\in q\Z^{n+1}
\midmid
\begin{gathered}
\|\xx\|>0,\ 
0<\|\cc+\xx\|\le X,\ 
\|\nu_{d,n}(\cc+\xx)\|\le t\\
\|(\cc+\xx)\wedge\xii\|
\le\realeps\|\cc+\xx\|\|\xii\|
\end{gathered}
\right\}.
\end{equation}
Therefore, by considering the semialgebraic set
\[
Z\coloneqq\left\{(\cc,\xii,\realeps,X,t,\xx)
\in\R^{n+1}\times\R^{n+1}\times\R\times\R\times\R\times\R^{n+1}
\midmid
(\ast)
\right\},
\]
where $(\ast)$ is the condition
\[
\tag{$\ast$}
\begin{gathered}
\|\xx\|>0,\quad
0<\|\cc+\xx\|\le X,\quad
\|\nu_{d,n}(\cc+\xx)\|\le t\\
\|(\cc+\xx)\wedge\xii\|
\le\realeps\|\cc+\xx\|\|\xii\|,
\end{gathered}
\]
we have
\begin{equation}
\label{lem:det_reciprocal_sum:rephrase2}
\sum_{\substack{%
\xx\in\cc+q\Z^{n+1}\\
0<\|\xx\|\le X\\
\|\nu_{d,n}(\xx)\|\le t\\
\xx\in\mathcal{C}_{n+1}(\xii,\realeps)}}1
=
\#(q\Z^{n+1}\cap Z_{(\cc,\xii,\realeps,X,t)}),
\end{equation}
where $Z_{\TT}$ is as in \cref{lem:BarroeroWidmer_original}.
For any $(\cc,\xii,\realeps,X,t)\in\R^{(n+1)+(n+1)+1+1+1}$, we have
$Z_{(\cc,\xii,\realeps,X,t)}\subset\mathcal{B}_{r}(X+\|\cc\|)$
which is bounded. Therefore,
by applying \cref{lem:BarroeroWidmer_original} to \cref{lem:det_reciprocal_sum:rephrase2}, we obtain
\begin{equation}
\label{lem:det_reciprocal_sum:BarroeroWidmer}
\sum_{\substack{%
\xx\in\cc+q\Z^{n+1}\\
0<\|\xx\|\le X\\
\|\nu_{d,n}(\xx)\|\le t\\
\xx\in\mathcal{C}_{n+1}(\xii,\realeps)}}1
=
\frac{\vol_{n+1}(Z_{(\cc,\xii,\realeps,X,t)})}{q^{n+1}}
+
O\Biggl(\sum_{0\le\nu\le n}
\frac{V_{\nu}(Z_{(\cc,\xii,\realeps,X,t)})}{q^{\nu}}
\Biggr).
\end{equation}
We have
\begin{equation}
\label{lem:det_reciprocal_sum:ZT_rephrase}
\begin{aligned}
Z_{(\cc,\xii,\realeps,X,t)}
&=
\left\{
\xx\in\R^{n+1}
\midmid
\begin{gathered}
\|\xx\|>0,\ 
0<\|\cc+\xx\|\le X,\ 
\|\nu_{d,n}(\cc+\xx)\|\le t\\
\|(\cc+\xx)\wedge\xii\|
\le\realeps\|\cc+\xx\|\|\xii\|
\end{gathered}
\right\}\\
&=
\left\{
\xx\in\R^{n+1}
\midmid
\begin{gathered}
\xx\neq\cc,\ 
0<\|\xx\|\le X,\ 
\|\nu_{d,n}(\xx)\|\le t\\
\xx\in\mathcal{C}_{n+1}(\xii,\realeps)
\end{gathered}
\right\}-\cc.
\end{aligned}
\end{equation}
We next bound $V_{\nu}(Z_{(\cc,\xii,\realeps,X,t)})$.
Take a $\nu$-dimensional subspace $W$ of $\R^{n+1}$ arbitrarily and
consider the orthogonal projection $\pi_{W}\colon\R^{r}\to W$.
It suffices to bound
$\vol_{W}(\pi_{W}(Z_{(\cc,\xii,\realeps,X,t)}))$.
By \cref{Veronese_bound} and \cref{lem:det_reciprocal_sum:ZT_rephrase}, we have
\[
\vol_{W}\bigl(\pi_{W}(Z_{(\cc,\xii,\realeps,X,t)})\bigr)
\le
\vol_{W}
\Bigl(
\pi_{W}\bigl(
\mathcal{C}_{n+1}(\xii,\realeps)\cap\mathcal{B}_{n+1}(\min(X,dt^{\frac{1}{d}}))
\bigr)\Bigr).
\]
By \cref{lem:cone_ball_projection_vol}, we have
\begin{equation}
\label{lem:det_reciprocal_sum:V_nu}
\vol_{W}\bigl(\pi_{W}(Z_{(\cc,\xii,\realeps,X,t)})\bigr)
\ll
\frac{1}{\realeps}\bigl(\realeps\min(X,dt^{\frac{1}{d}})\bigr)^{\nu}
\end{equation}
for $\nu\ge1$.
Also, we trivially have
\begin{equation}
\label{lem:det_reciprocal_sum:V_0}
V_{0}(Z_{(\cc,\xii,\realeps,X,t)})
\le1.
\end{equation}
On inserting
\cref{lem:det_reciprocal_sum:ZT_rephrase},
\cref{lem:det_reciprocal_sum:V_nu} and \cref{lem:det_reciprocal_sum:V_0}
into \cref{lem:det_reciprocal_sum:BarroeroWidmer}, we obtain
\[
\sum_{\substack{%
\xx\in\cc+q\Z^{n+1}\\
0<\|\xx\|\le X\\
\|\nu_{d,n}(\xx)\|\le t\\
\xx\in\mathcal{C}_{n+1}(\xii,\realeps)}}1
=
\frac{1}{q^{n+1}}
\int\limits_{\substack{\|\xx\|\le X\\
\|\nu_{d,n}(\xx)\|\le t\\
\xx\in\mathcal{C}_{n+1}(\xii,\realeps)}}d\xx
+
O(E_1+E_2),
\]
where
\[
E_1\coloneqq\frac{1}{\realeps}\biggl(\frac{\realeps}{q}\min(X,dt^{\frac{1}{d}})\biggr)^{n}\and
E_2\coloneqq\frac{1}{q}\min(X,dt^{\frac{1}{d}})
\]
provided $E_2\ge\frac{1}{3}$ and so $E_2\gg1$.
When $E_2\le\frac{1}{3}$, by \cref{Veronese_bound}, we have
\[
0<\|\xx\|\le X\ \text{and}\ 
\|\nu_{d,n}(\xx)\|\le t
\implies
0<\|\xx\|\le\frac{q}{3}.
\]
By \cref{lem:cone_ball_projection_vol}
and $\frac{\realeps}{q}\min(X,dt^{\frac{1}{d}})\le1$ in this case, we have
\begin{align}
\frac{1}{q^{n+1}}
\int\limits_{\substack{\|\xx\|\le X\\
\|\nu_{d,n}(\xx)\|\le t\\
\xx\in\mathcal{C}_{n+1}(\xii,\realeps)}}d\xx
&\ll
\frac{1}{q^{n+1}}
\vol_{n+1}\Bigl(\mathcal{C}_{n+1}(\xii,\realeps)
\cap\mathcal{B}_{n+1}\bigl(\min(X,dt^{\frac{1}{d}})\bigr)
\Bigr)\\
&\ll
\frac{1}{\realeps}\biggl(\frac{\realeps}{q}\min(X,dt^{\frac{1}{d}})\biggr)^{n+1}
\ll
E_1.
\end{align}
Also, since there is at most one point in
\[
(\cc+q\Z^{n+1})\cap\mathcal{B}_{n+1}(\min(X,dt^{\frac{1}{d}}))
\subset
(\cc+q\Z^{n+1})\cap\mathcal{B}_{n+1}(\tfrac{q}{3}),
\]
we have
\[
\sum_{\substack{%
\xx\in\cc+q\Z^{n+1}\\
0<\|\xx\|\le X\\
\|\nu_{d,n}(\xx)\|\le t\\
\xx\in\mathcal{C}_{n+1}(\xii,\realeps)}}1
\le
\mathbbm{1}_{(\cc+q\Z^{n+1})
\cap(\mathcal{B}_{n+1}(\min(X,dt^{\frac{1}{d}}))\setminus\{0\})\neq\emptyset}
\eqqcolon\tilde{\mathfrak{R}}.
\]
Thus, in any case, we have
\[
\sum_{\substack{%
\xx\in\cc+q\Z^{n+1}\\
0<\|\xx\|\le X\\
\|\nu_{d,n}(\xx)\|\le t\\
\xx\in\mathcal{C}_{n+1}(\xii,\realeps)}}1
=
\frac{1}{q^{n+1}}
\int\limits_{\substack{\|\xx\|\le X\\
\|\nu_{d,n}(\xx)\|\le t\\
\xx\in\mathcal{C}_{n+1}(\xii,\realeps)}}d\xx
+
O(E_1+E_2+\tilde{\mathfrak{R}}).
\]
We substitute this formula into \cref{lem:det_reciprocal_sum:partial_sum} and obtain
\begin{equation}
\label{lem:det_reciprocal_sum:after_Barroero_Widmer}
\sum_{\substack{%
\xx\in\cc+q\Z^{n+1}\\
0<\|\xx\|\le X\\
\xx\in\mathcal{C}_{n+1}(\xii,\realeps)}}
\frac{1}{\|\nu_{d,n}(\xx)\|}
=
I+O(I_{0}+I_{1}+I_{2}+I_{\tilde{\mathfrak{R}}}),
\end{equation}
where
\begin{align}
I&\coloneqq
\frac{1}{q^{n+1}}
\int_{0}^{\infty}
\biggl(\int\limits_{\substack{\|\xx\|\le X\\
\|\nu_{d,n}(\xx)\|\le t\\
\xx\in\mathcal{C}_{n+1}(\xii,\realeps)}}d\xx\biggr)
\frac{dt}{t^2},\\
I_{0}&\coloneqq
\frac{1}{q^{n+1}}
\int_{0}^{1}
\biggl(\int\limits_{\substack{\|\xx\|\le X\\
\|\nu_{d,n}(\xx)\|\le t\\
\xx\in\mathcal{C}_{n+1}(\xii,\realeps)}}d\xx\biggr)
\frac{dt}{t^2},\\
I_{1}&\coloneqq
\frac{1}{\realeps}\biggl(\frac{\realeps}{q}\biggr)^{n}
\int_{1}^{\infty}\min(X,dt^{\frac{1}{d}})^{n}\frac{dt}{t^2},\\
I_{2}&\coloneqq
\frac{1}{q}
\int_{1}^{\infty}\min(X,dt^{\frac{1}{d}})\frac{dt}{t^2},\\
I_{\tilde{\mathfrak{R}}}
&\coloneqq
\int_{1}^{\infty}\mathbbm{1}_{(\cc+q\Z^{n+1})
\cap(\mathcal{B}_{n+1}(\min(X,dt^{\frac{1}{d}}))\setminus\{0\})\neq\emptyset}\frac{dt}{t^2}
\end{align}
For the integral $I$, we just swap the integral and obtain
\begin{align}
I
&=
\frac{1}{q^{n+1}}
\int\limits_{\substack{\|\xx\|\le X\\
\xx\in\mathcal{C}_{n+1}(\xii,\realeps)}}
\biggl(\int_{\|\nu_{d,n}(\xx)\|}^{\infty}\frac{dt}{t^2}\biggr)d\xx\\
&=
\frac{1}{q^{n+1}}
\int\limits_{\substack{\|\xx\|\le X\\
\xx\in\mathcal{C}_{n+1}(\xii,\realeps)}}\frac{d\xx}{\|\nu_{d,n}(\xx)\|}
=
\frac{X^{n+1-d}}{q^{n+1}}
\int_{\mathcal{C}_{n+1}(\xii,\realeps)\cap\mathcal{B}_{n+1}(1)}
\frac{d\xx}{\|\nu_{d,n}(\xx)\|}.
\end{align}
and so
\begin{equation}
\label{lem:det_reciprocal_sum:I}
I=
\mathfrak{W}_{d,n}(\xii,\realeps)
\frac{X^{n+1-d}}{q^{n+1}}.
\end{equation}
For the integral $I_0$, by using \cref{Veronese_bound} and \cref{lem:cone_ball_projection_vol}, we bound as
\begin{equation}
\label{lem:det_reciprocal_sum:I0}
\begin{aligned}
I_0
&\ll
\frac{1}{q^{n+1}}
\int_{0}^{1}
\vol_{n+1}(\mathcal{C}_{n+1}(\xii,\realeps)\cap\mathcal{B}_{n+1}(dt^{\frac{1}{d}}))\frac{dt}{t^2}\\
&\ll
\frac{1}{\realeps}\biggl(\frac{\realeps}{q}\biggr)^{n+1}
\int_{0}^{1}
t^{\frac{n+1}{d}-2}dt
\ll
\frac{1}{\realeps}\biggl(\frac{\realeps}{q}\biggr)^{n+1}
\ll
\frac{1}{\realeps}\biggl(\frac{\realeps}{q}\biggr)^{n}
X^{n-d}.
\end{aligned}
\end{equation}
For $I_1$, we dissect integral at $t=X^d$ and bound as
\begin{equation}
\label{lem:det_reciprocal_sum:I1}
\begin{aligned}
I_1
&=
\frac{1}{\realeps}\biggl(\frac{\realeps}{q}\biggr)^{n}
\biggl(
\int_{1}^{X^d}t^{\frac{n}{d}-2}dt
+
X^n\int_{X^d}^{\infty}\frac{dt}{t^2}
\biggr)\\
&\ll
\frac{1}{\realeps}\biggl(\frac{\realeps}{q}\biggr)^{n}
(X^{n-d}+\log X),
\end{aligned}
\end{equation}
where $\log X$ is introduced to deal with the case $n=d$.
Similarly, $I_2$ is
\begin{equation}
\label{lem:det_reciprocal_sum:I2}
\begin{aligned}
I_2
\ll
\frac{1}{q}\biggl(\int_{1}^{X^d}t^{\frac{1}{d}-2}du+X\int_{X^d}^{\infty}\frac{dt}{t^2}\biggr)
\ll
\frac{1}{q}
\end{aligned}
\end{equation}
Finally, we bound $I_{\tilde{\mathfrak{R}}}$ as
\begin{equation}
\label{lem:det_reciprocal_sum:IR}
I_{\tilde{\mathfrak{R}}}
\ll
\mathfrak{R}
+\int_{q}^{\infty}\frac{dt}{t^2}
\ll
\mathfrak{R}+\frac{1}{q}.
\end{equation}
On inserting
\cref{lem:det_reciprocal_sum:I},
\cref{lem:det_reciprocal_sum:I0},
\cref{lem:det_reciprocal_sum:I1},
\cref{lem:det_reciprocal_sum:I2}
and \cref{lem:det_reciprocal_sum:IR}
into \cref{lem:det_reciprocal_sum:after_Barroero_Widmer},
we get the lemma.
\end{proof}

%%%%%%%%%%%%%%%%%%%%%%%%%%%%%%%%%%%%%%%%
\begin{lemma}
\label{lem:det_reciprocal_sum_with_unit}
Let $n,d$ be positive integers with $n\ge d\ge 2$.
For $\cc\in\Z^{n+1}$ and $q\in\N$ such that $\gcd(\cc,q)=1$
and for $\xii\in\R^{n+1}\setminus\{0\}$, $0\le\realeps\le1$ and $X\ge1$, we have
\begin{align}
%T_{d,n}(X;\cc,q;\xii,\realeps)
%\coloneqq{}
&\astsum_{u\ \mod{q}}
\sum_{\substack{%
\xx\in u\cc+q\Z^{n+1}\\
0<\|\xx\|\le X\\
\xx\in\mathcal{C}_{n+1}(\xii,\realeps)}}
\frac{1}{\|\nu_{d,n}(\xx)\|}\\
&=
\mathfrak{W}_{d,n}(\xii,\realeps)
\frac{\varphi(q)}{q^{n+1}}X^{n+1-d}
+O\biggl(\biggl(\frac{\realeps}{q}\biggr)^{n-1}(X^{n-d}+\log X)+1\biggr),
\end{align}
where $\mathfrak{W}_{d,n}(\xii,\realeps)$ is as in \cref{lem:det_reciprocal_sum}
and the implicit constant depends only on $d,n$.
\end{lemma}
%%%%%%%%%%%%%%%%%%%%%%%%%%%%%%%%%%%%%%%%
\begin{proof}
We use \cref{lem:det_reciprocal_sum} with $\cc$ replaced by $u\cc$
and then sum up the resulting formula over $u\in(\Z/q\Z)^{\times}$.
The main term is just multiplied by $\varphi(q)$
and all but the last error term is multiplied by $\varphi(q)\le q$.
It thus suffices to bound the remaining error term
\[
E\coloneqq
\int_{1}^{q}
\biggl(\astsum_{u\ \mod{q}}
\mathbbm{1}_{(u\cc+q\Z^{n+1})
\cap\mathcal{B}_{n+1}(\min(X,dt^{\frac{1}{d}}))\neq\emptyset}\biggr)\frac{dt}{t^2}.
\]
By \cref{lem:uc_distribution}, we have
\[
E
\le
\int_{1}^{q}
\biggl(\astsum_{u\ \mod{q}}
\mathbbm{1}_{(u\cc+q\Z^{n+1})
\cap\mathcal{B}_{n+1}(dt^{\frac{1}{d}})\neq\emptyset}\biggr)\frac{dt}{t^2}
\ll
\int_{1}^{q}t^{\frac{1}{d}-2}dt
\ll
1.
\]
This completes the proof.
\end{proof}

%%%%%%%%%%%%%%%%%%%%%%%%%%%%%%%%%%%%%%%%
\begin{lemma}
\label{lem:W_bound}
For $n,d\in\mathbb{N}$ with $n\ge d\ge 2$, $\xii\in\R^{n+1}\setminus\{0\}$ and $0\le\realeps\le1$, we have
\[
\mathfrak{W}_{d,n}(\xii,\realeps)
\asymp
\realeps^{n},
\]
where $\mathfrak{W}_{d,n}(\xii,\realeps)$ is as in \cref{lem:det_reciprocal_sum}
and the implicit constant depends only on $d,n$.
\end{lemma}
%%%%%%%%%%%%%%%%%%%%%%%%%%%%%%%%%%%%%%%%
\begin{proof}
For the upper bound, we use \cref{Veronese_bound} and \cref{lem:cone_ball_projection_vol} to obtain
\begin{align}
\mathfrak{W}_{d,n}(\xii,\realeps)
&=
\int_{\mathcal{C}_{n+1}(\xii,\realeps)\cap\mathcal{B}_{n+1}(1)}
\int_{\|\nu_{d,n}(\xx)\|}^{\infty}\frac{dt}{t^2}d\xx\\
&\le
\int_{0}^{\infty}
\biggl(\int_{\mathcal{C}_{n+1}(\xii,\realeps)\cap\mathcal{B}_{n+1}(\min(1,dt^{\frac{1}{d}}))}
d\xx\biggr)\frac{dt}{t^2}\\
&\ll
\realeps^{n}
\int_{0}^{\infty}\min(1,dt^{\frac{1}{d}})^{n+1}\frac{dt}{t^2}\\
&\ll
\realeps^{n}
\biggl(\int_{0}^{1}t^{\frac{n+1}{d}-2}dt
+\int_{1}^{\infty}\frac{dt}{t^2}\biggr)
\ll
\realeps^{n}
\end{align}
For the lower bound, by \cref{Veronese_bound} and \cref{lem:cone_ball_cap_vol}, we have
\begin{align}
\mathfrak{W}_{d,n}(\xii,\realeps)
&\gg
\int_{\mathcal{C}_{n+1}(\xii,\realeps)\cap\mathcal{B}_{n+1}(1)}d\xx
\gg
\realeps^{n}.
\end{align}
This completes the proof.
\end{proof}

%%%%%%%%%%%%%%%%%%%%%%%%%%%%%%%%%%%%%%%%
\begin{lemma}
\label{lem:det_reciprocal_prim_sum}
Let $n,d$ be positive integers with $n\ge d\ge 2$.
For $\cc\in\Z^{n+1}$ and $q\in\N$ such that $\gcd(\cc,q)=1$
and for $\xii\in\R^{n+1}\setminus\{0\}$, $0\le\realeps\le1$ and $X\ge1$,
\begin{align}
%&E_{d,n}(X;\cc,q;\xii,\realeps)
%\coloneqq
&
\astsum_{u\ \mod{q}}
\sum_{\substack{%
\xx\in\Zprim^{n+1}\\
\|\xx\|\le X\\
\xx\equiv u\cc\ \mod{q}\\
\xx\in\mathcal{C}_{n+1}(\xii,\realeps)}}
\frac{1}{\|\nu_{d,n}(\xx)\|}\\
&=
\mathfrak{W}_{d,n}(\xii,\realeps)\frac{\varphi(q)}{J_{n+1}(q)}\cdot
\frac{X^{n+1-d}}{\zeta(n+1)}
+O\biggl(\biggl(\frac{\realeps}{q}\biggr)^{n-1}(X^{n-d}+\log X)+1\biggr),
\end{align}
where $\mathfrak{W}_{d,n}(\xii,\realeps)$ is as in \cref{lem:det_reciprocal_sum}
and the implicit constant depends only on $d,n$.
\end{lemma}
%%%%%%%%%%%%%%%%%%%%%%%%%%%%%%%%%%%%%%%%
\begin{proof}
For brevity, let us write
\[
S
\coloneqq
\astsum_{u\ \mod{q}}
S(u)
\and
S(u)\coloneqq
\sum_{\substack{%
\xx\in\Zprim^{n+1}\\
\|\xx\|\le X\\
\xx\equiv u\cc\ \mod{q}\\
\xx\in\mathcal{C}_{n+1}(\xii,\realeps)}}
\frac{1}{\|\nu_{d,n}(\xx)\|}.
\]
For $u\in(\Z/q\Z)^{\times}$, we have
\begin{align}
S(u)
=
\sum_{\substack{%
\xx\in\Z^{n+1}\\
0<\|\xx\|\le X\\
\xx\equiv u\cc\ \mod{q}\\
\xx\in\mathcal{C}_{n+1}(\xii,\realeps)}}
\frac{1}{\|\nu_{d,n}(\xx)\|}
\sum_{\ell\mid\xx}\mu(\ell)
&=
\sum_{\ell\le X}\mu(\ell)
\sum_{\substack{%
\xx\in\Z^{n+1}\\
0<\|\xx\|\le X\\
\xx\equiv u\cc\ \mod{q}\\
\xx\in\mathcal{C}_{n+1}(\xii,\realeps)\\
\ell\mid\xx}}
\frac{1}{\|\nu_{d,n}(\xx)\|}\\
&=
\sum_{\ell\le X}\frac{\mu(\ell)}{\ell^{d}}
\sum_{\substack{%
\xx\in\Z^{n+1}\\
0<\|\xx\|\le X/\ell\\
\ell\xx\equiv u\cc\ \mod{q}\\
\ell\xx\in\mathcal{C}_{n+1}(\xii,\realeps)}}
\frac{1}{\|\nu_{d,n}(\xx)\|}.
\end{align}
Since the cone $\mathcal{C}_{N}(\xii,\realeps)$ is invariant under dilation, we have
\begin{equation}
\label{lem:det_reciprocal_prim_sum:first_step}
S
=
\sum_{\ell\le X}\frac{\mu(\ell)}{\ell^{d}}
\astsum_{u\ \mod{q}}
\sum_{\substack{%
\xx\in\Z^{n+1}\\
0<\|\xx\|\le X/\ell\\
\ell\xx\equiv u\cc\ \mod{q}\\
\xx\in\mathcal{C}_{n+1}(\xii,\realeps)}}
\frac{1}{\|\nu_{d,n}(\xx)\|}.
\end{equation}
By the assumption $\gcd(\cc,q)=1$, we have
\[
(u,q)=1\ \text{and}\ 
\ell\xx\equiv u\cc\ \mod{q}
\implies
(\ell,q)=1.
\]
Thus,
by changing the summation variable $u$ via $u\rightsquigarrow \ell u$,
we can rewrite \cref{lem:det_reciprocal_prim_sum:first_step} as
\begin{equation}
\label{lem:det_reciprocal_prim_sum:second_step}
\begin{aligned}
S
&=
\sum_{\substack{\ell\le X\\(\ell,q)=1}}\frac{\mu(\ell)}{\ell^{d}}
\astsum_{u\ \mod{q}}
\sum_{\substack{%
\xx\in\Z^{n+1}\\
0<\|\xx\|\le X/\ell\\
\ell\xx\equiv\ell u\cc\ \mod{q}\\
\xx\in\mathcal{C}_{n+1}(\xii,\realeps)}}
\frac{1}{\|\nu_{d,n}(\xx)\|}\\
&=
\sum_{\substack{\ell\le X\\(\ell,q)=1}}\frac{\mu(\ell)}{\ell^{d}}
\astsum_{u\ \mod{q}}
\sum_{\substack{%
\xx\in\Z^{n+1}\\
0<\|\xx\|\le X/\ell\\
\xx\equiv u\cc\ \mod{q}\\
\xx\in\mathcal{C}_{n+1}(\xii,\realeps)}}
\frac{1}{\|\nu_{d,n}(\xx)\|}.
\end{aligned}
\end{equation}
We then apply \cref{lem:det_reciprocal_sum_with_unit} to the inner sum and get
\begin{equation}
\label{lem:det_reciprocal_prim_sum:prefinal}
S
=
\mathfrak{W}_{d,n}(\xii,\realeps)
\frac{\varphi(q)}{q^{n+1}}X^{n+1-d}
\sum_{\substack{\ell\le X\\(\ell,q)=1}}\frac{\mu(\ell)}{\ell^{n+1}}
+O\biggl(\biggl(\frac{\realeps}{q}\biggr)^{n-1}(X^{n-d}+\log X)+1\biggr)
\end{equation}
since $n\ge d\ge2$. For the main term, we use
\[
\sum_{\substack{\ell\le X\\(\ell,q)=1}}\frac{\mu(\ell)}{\ell^{n+1}}
=
\prod_{p\nmid q}\biggl(1-\frac{1}{p^{n+1}}\biggr)+O\biggl(\frac{1}{X}\biggr)
=
\frac{q^{n+1}}{J_{n+1}(q)\zeta(n+1)}+O\biggl(\frac{1}{X}\biggr).
\]
Therefore, by \cref{lem:W_bound}, we have
\begin{align}
&\mathfrak{W}_{d,n}(\xii,\realeps)
\frac{\varphi(q)}{q^{n+1}}X^{n+1-d}
\sum_{\substack{\ell\le X\\(\ell,q)=1}}\frac{\mu(\ell)}{\ell^{n+1}}\\
&=
\mathfrak{W}_{d,n}(\xii,\realeps)
\frac{\varphi(q)}{J_{n+1}(q)}\cdot\frac{X^{n+1-d}}{\zeta(n+1)}
+
O\biggl(\biggl(\frac{\realeps}{q}\biggr)^{n}X^{n-d}\biggr).
\end{align}
On inserting this formula into \cref{lem:det_reciprocal_prim_sum:prefinal},
we obtain the lemma.
\end{proof}

\section{Key lemmas}
%%%%%%%%%%%%%%%%%%%%%%%%%%%%%%%%%%%%%%%%%

\begin{definition}
Let $M$ be a free $\Z$-module of finite rank.
Let $v \in M$.
Define 
\begin{align}
g(v) := \#(M/\Z v)_{\rm tors}
\end{align}
for $v \neq 0$ and define $g(0)=0$.
\end{definition}
When $v \neq 0$, write $v = d w$ with $d \in \Z$ and a primitive element $w \in M$.
Then $g(v) = |d|$.

\begin{definition}\label{def:Srn}
Let us consider 
\begin{itemize}
    \item $n \in \Z_{\geq 2}$ and $r \in \Z$ with $1 \leq r \leq n+1$;
    \item $s_{1}, \dots ,s_{r} \geq 1$;
    \item $\s \in (0,1]$ and $\xii \in \R^{n+1} \setminus \{0\}$;
    \item $q \in \Z_{\geq 1}$ and $\cc \in \Z^{n+1}$ with $\gcd(g(\cc), q)=1$.
\end{itemize}
Note that if $\cc =0$, the last condition means $q=1$.
Let
\begin{align}
S_{r,n} = S_{r,n}(s_1,\dots, s_r; \cc, q ; \xii, \s) =
\# \mathcal{S}_{r,n}
\end{align}
where
\begin{align}
&\mathcal{S}_{r,n}= \mathcal{S}_{r,n}(s_1,\dots, s_r; \cc, q ; \xii, \s)\\
&=\left\{ L \subset \Z^{n+1} \ \txt{primitive\\sublattice} \ \middle|\ 
\txt{$\rank L = r$, $s_{i}/2 < \l_{i}(L) \leq s_{i} \ (1 \leq i \leq r)$\\ 
$\exists \xx \in L$ s.t. $\xx \equiv u\cc\ \mod {q\Z^{n+1}}$ \\
for some $u \in (\Z/q\Z)^{\times}$\\
and $d_{\infty}(\xx, \xii) \leq \s$}  \right\}.
\end{align}
\end{definition}

\begin{definition}
For a  non-zero sublattice $L \subset \Z^{m}$ and $\xii \in \R^{m} \setminus \{0\}$, define 
\begin{align}
d_{\infty}(L, \xii) := \inf \left\{ d_{\infty}(\xx, \xii) \ \middle|\ \xx \in  L_{\R} \setminus \{0\} \right\}.
\end{align}
\end{definition}

\begin{remark}
Let $\xii \in \R^m$ be such that $\|\xii\|=1$.
If $\xii = \xii_{1} + \xii_{2}$ where $\xii_{1} \in L_{\R}$ and $\xii_{2} \in L_{\R}^{\perp}$,
we have
\begin{align}
d_{\infty}(L, \xii) = \| \xii_{2} \|. 
\end{align}
Indeed, for $\xx \in L_{\R}$ with $ \| \xx \|=1$,  we have
\begin{align}
d_{\infty}(\xx, \xii)^{2} = \frac{\| \xx \|^{2} \| \xii \|^{2} - \langle \xx, \xii \rangle^{2}}{\| \xx \|^{2} \| \xii \|^{2} }  = 1 - \langle \xx, \xii_1 \rangle^{2}.
\end{align}
This takes minimum when $\xx = \xii_1/\|\xii_1\|$:
\begin{align}
    d_{\infty}(L, \xii)^2 = 1 - \frac{\langle \xii_1, \xii_1 \rangle^2}{\|\xii_1\|^2} = \|\xii_2\|^2.
\end{align}
\end{remark}

\begin{lemma}\label{lem:seq-of-sublattices-general}
Notation as in \cref{def:Srn}.
Let $L \in \mathcal{S}_{r,n}$.
Then there is a sequence of primitive sublattices 
\begin{align}
0=L_{0}\subset L_{1} \subset L_{2} \subset \cdots \subset L_{r}= L 
\end{align}
such that
\begin{enumerate}
\item $\rank L_{i}=i $
\item $\det L_{i} \asymp s_{1} \cdots s_{i}$
\item $ \l_{1}(L_{i}/L_{i-1}) \asymp s_{i}$
\end{enumerate}
for $i=1,\dots, r$. Here implicit constants depend only on $r$.
(We equip metrics on the quotient lattices using orthogonal projections as usual.)

Let $\pi_{i} \colon \Z^{n+1} \longrightarrow \Z^{n+1}/L_{i}$ be the quotient map.
Set $d_{i} = (g(\pi_{i}(\cc)), q)$ for $i=0,\dots,r$.
Then we have
\begin{align}
1 = d_{0}|d_{1} | \cdots | d_{r-1} | d_{r}=q.
\end{align}
We also have
\begin{align}
&d_{\infty}(L_{1}, \xii) \geq \cdots \geq d_{\infty}(L_{r}, \xii);\\
& d_{\infty}(L_{r}, \xii) \leq \s.
\end{align}

Moreover, if $v \in L_{i+1}/L_{i}$ is a $\Z$-basis, then
\begin{enumerate}
\item
there is $u \in (\Z/(d_{i+1}/d_{i})\Z)^{\times}$ such that
\begin{align}
v \equiv u \left(d_{i}^{-1}\pi_{i}(\cc) \right)\ \mod { (d_{i+1}/d_{i})\Z^{n+1}/L_{i}  }.
\end{align}
Here note that $d_{i}^{-1}\pi_{i}(\cc) \in \Z^{n+1}/L_{i}$ and 
\begin{align}
\left( g\left(d_{i}^{-1}\pi_{i}(\cc) \right), d_{i+1}/d_{i}\right) = 1
\end{align}
by the definition of $d_{i}$.
Note also that if $\pi_i(\cc) = 0$, then $d_i=d_{i+1} = q$ and hence this is vacuously true.
\item
we have
\begin{align}
d_{\infty, (L_{i})_{\R}^{\perp}}(v, \pi_{i}(\xii)) \leq \frac{d_{\infty}(L_{i+1}, \xii)}{d_{\infty}(L_{i}, \xii)} .
\end{align}
If $\xii \in (L_{i})_{\R}$, then the RHS is $\infty$ and the condition is vacuous.
\end{enumerate}
\end{lemma}

\begin{proof}
    The existence of the sequence $L_i$ follows from \cref{lem:nice_basis} or
    \cite[Lemma 12.3]{Davenport:DiophantineEquation}.
    We have $g(\pi_i(\cc)) \mid g(\pi_{i+1}(\cc))$ by the definition of $g$.
    Thus $d_i \mid d_{i+1}$.
    Since $L_r = L \in \mathcal{S}_{r,n}$, we have $d_r = q$.
    We have $d_{\infty}(L_i, \xii) \geq d_{\infty}(L_{i+1}, \xii)$
    since $L_i \subset L_{i+1}$.
    Since $L_r = L \in \mathcal{S}_{r,n}$, we have $d_{\infty}(L_r, \xii) \leq \s$.

    For (1), identify $\Z^{n+1}/L_i$ with $\Z^{n+1-i}$ via 
    a $\Z$-basis of $\Z^{n+1}/L_i$ extending $v$.
    Let us write $\pi_i(\cc) = (k_1,\dots, k_{n+1-i})$.
    Then
    \begin{align}
        &\gcd(k_1,\dots, k_{n+1-i},q) = d_i\\
        &\gcd(k_2,\dots, k_{n+1-i},q) = d_{i+1}.
    \end{align}
    Thus 
    \begin{align}
        d_i^{-1}\pi_i(\cc) = (k_1/d_i, \dots, k_{n+1-i}/d_i) \equiv
        (k_1/d_i, 0 , \dots, 0) \ \mod{(d_{i+1}/d_i)\Z^{n+1-i}}
    \end{align}
    and $\gcd(k_1/d_i, d_{i+1}/d_i)=1$. This proves what we wanted.

    For (2), first note that we may assume $\|\xii\|=1$.
    Let us write
    \begin{align}
        \xii = \etaa + \etaa' \quad \etaa \in (L_{i+1})_{\R},\ \etaa' \in (L_{i+1})_{\R}^\perp.
    \end{align}
    Since $\|\xii\| = 1$, we have $d_{\infty}(L_{i+1}, \xii) = \|\etaa'\|$.
    Write similarly
    \begin{align}
        \etaa = \etaa_1 + \etaa_2 \quad \etaa_1 \in (L_i)_{\R},\ \etaa_2 \in (L_i)_{\R}^{\perp}.
    \end{align}
    If $\etaa_2 + \etaa' = 0$, then $\xii \in \R L_i$ and we are done.
    Assume $\etaa_2 + \etaa' \neq 0$.
    Then we have $d_{\infty}(L_i, \xii) = \|\etaa_2 + \etaa'\|$.
    If $\etaa_2 = 0$, then $d_{\infty}(L_{i+1}, \xii) = d_{\infty}(L_i, \xii) = \|\etaa'\|$
    and we are done. Thus we may assume $\etaa_2 \neq 0$.
    We have
    \begin{align}
        d_{\infty}(L_{i+1}, \xii) = \|\etaa'\| =  \| (\etaa_2 + \etaa') - \etaa_2 \|.
    \end{align}
    Dividing both side by $d_{\infty}(L_i, \xii)$, we get
    \begin{align}
        \frac{d_{\infty}(L_{i+1}, \xii)}{d_{\infty}(L_i, \xii)} 
        = \left\| \frac{\etaa_2 + \etaa'}{\|\etaa_2 + \etaa'\|} - \frac{\etaa_2}{\|\etaa_2 + \etaa'\|} \right\|
        \geq |\sin(\theta)|
        = d_{\infty}(\etaa_2 + \etaa', \etaa_2)
    \end{align}
    where $\theta$ is the angle between 
    $\etaa_2 + \etaa'$ and $\etaa_2$.
    By the identification 
    $\pi_i|_{(L_i)_{\R}^{\perp}} \colon (L_i)_{\R}^{\perp} \to \R^{n+1}/(L_i)_{\R}$,
    $\pi_i(\xii)$ corresponds to $\etaa_2 + \etaa'$ and $v$ corresponds to
    some real multiple of $\etaa_2$.
    Thus $d_{\infty}(\etaa_2 + \etaa', \etaa_2) 
    = d_{\infty, \R L_i^{\perp}}(\pi_i(\xii), v)$ and we are done.
\end{proof}

\begin{lemma}\label{lem:point-count-general}
Let $L \subset \Z^{n+1}$ be a primitive sublattice of rank $i$.
Let $d \geq 1$.
Let $\cc \in \Z^{n+1}/L$ be such that $(g(\cc), d)=1$.
Let $\xii \in L_{\R}^{\perp} \setminus \{ 0\}$ and let $\s \in (0,1]$.
Then 
\begin{align}
&\# \left\{ v \in (\Z^{n+1}/L)_{\prim} \ \middle|\ \txt{$\| v \| \leq X, \exists u \in (\Z/d\Z)^{\times}, v \equiv u\cc\  \mod {d \Z^{n+1}/L}$ \\
$d_{\infty, \R L^{\perp}} (v, \xii) \leq \s$ } \right\}\\
&\ll \det L \left( \frac{\s^{n-i}}{d^{n-i}}X^{n+1-i} + X \right)
\end{align}
where the implicit constant depends only on $n,i$.
\end{lemma}
\begin{proof}
    This follows from \cref{lem:V_bound,prop:lattice_count_local}.
\end{proof}

\begin{proposition}\label{prop:Srn-bound}
Notation as in \cref{def:Srn}.
We have
\begin{align}
&S_{r,n}(s_1,\dots, s_r; \cc, q ; \xii, \s) \\
&\ll \Bigl(\tau(q) \log\frac{2}{\sigma}\Bigr)^{r-1} \left( \prod_{j=1}^{r} s_{j}^{r-j} \right) 
\left(  \left(\frac{\s }{q} \right)^{n+1-r}\prod_{j=1}^{r} s_{j}^{n+2-j} + \sum_{l =1}^{r}s_{l}\prod_{\substack{j \neq l \\ 1 \leq j \leq r}}s_{j}^{n+2-j}\right)
\end{align}
where the implicit constant depends only on $n$.
Here $\tau(q)$ is the number of divisors of $q$.
\end{proposition}

\begin{proof}
Consider the intervals 
\begin{align}
I_{k+1} = \left[0, \frac{1}{2^{k}} \right], I_{k} = \left( \frac{1}{2^{k}}, \frac{1}{2^{k-1}} \right], \dots, 
I_{2} = \left( \frac{1}{2^{2}}, \frac{1}{2} \right], I_{1}=\left( \frac{1}{2}, 1 \right]
\end{align}
where $k \in \Z_{\geq 0}$ is such that
\begin{align}
\frac{1}{2^{k+1}} < \s  \leq \frac{1}{2^{k}},
\end{align}
equivalently $k = \floor*{\log_{2}\s^{-1}}$.

By \cref{lem:seq-of-sublattices-general},
\begin{align}
S_{r,n} &\leq 
\sum_{1 \leq i_{1} \leq \cdots \leq i_{r-1} \leq k+1}
\sum_{d_{1}| \cdots |d_{r-1} |q}\\
&\sum_{\substack{L_{1} \subset \Z^{n+1} \text{prim.} \\ \rank L_{1} = 1 \\ \det L_{1} \asymp s_{1} \\ (g(\pi_{1}(\cc)), q)=d_{1} \\ d_{\infty}(L_{1}, \xii) \in I_{i_{1}} }}
\sum_{\substack{L_{1} \subset L_{2} \subset \Z^{n+1} \text{prim.} \\ \rank L_{2} = 2 \\ \det L_{2} \asymp s_{1}s_{2}\\
 \l_{1}(L_{2}/L_{1}) \asymp s_{2} \\ (g(\pi_{2}(\cc)), q)=d_{2} \\ d_{\infty}(L_{2}, \xii) \in I_{i_{2}} }}
\cdots 
\sum_{\substack{L_{r-1} \subset L_{r} \subset \Z^{n+1} \text{prim.} \\ \rank L_{r} = r \\ \det L_{r} \asymp s_{1}\cdots s_{r}\\ 
\l_{1}(L_{r}/L_{r-1}) \asymp s_{r} \\ (g(\pi_{r}(\cc)), q)=q \\ d_{\infty}(\xii, L_{r}) \in I_{k+1}}}
1.
\end{align}
Here $\pi_{i} \colon \Z^{n+1} \longrightarrow \Z^{n+1}/L_{i}$ are the quotient maps.

Let us set 
\begin{align}
&d_{0} = 1, d_{r} = q, i_{r} = k+1\\
&\s_{j} = \min \left\{ 1, \frac{ \max I_{i_{j}}}{ \min I_{i_{j-1}}}  \right\} \ \text{for $2 \leq j \leq r$}\\
& \s_{1} = \max I_{i_{1}}.
\end{align}
Note that if $\pi_{j}(\xii) = 0$, then $d_{\infty}(L_{j}, \xii) = 0 \in I_{k+1}$
and thus $\s_{j+1} = 1$.
In the following, we consider the condition of the form
\begin{align}
d_{\infty, (L_{j})_{\R}^{\perp}} (v, \pi_{j}(\xii)) \leq \s_{j+1}
\end{align}
for $v \in \Z^{n+1}/L_{j}$.
When $\pi_{j}(\xii) = 0$, this should be understood as a vacuous condition.

By \cref{lem:seq-of-sublattices-general} and \cref{lem:point-count-general},  we first see
\begin{align}
&\sum_{\substack{L_{r-1} \subset L_{r} \subset \Z^{n+1} \text{prim.} \\ \rank L_{r} = r \\ \det L_{r} \asymp s_{1}\cdots s_{r}\\ 
\l_{1}(L_{r}/L_{r-1}) \asymp s_{r} \\ (g(\pi_{r}(\cc)), q)=q \\ d_{\infty}(\xii, L_{r}) \in I_{k+1}}}
1\\
&\leq
\# \left\{ v \in (\Z^{n+1}/L_{r-1})_{\prim} \ \middle|\  \txt{$\| v \| \ll s_{r}$, 
$\exists u \in (\Z/(q/d_{r-1})\Z)^{\times}$,\\ $v \equiv u(d_{r-1}^{-1}\pi_{r-1}(\cc)) \ \mod {(q/d_{r-1})(\Z^{n+1}/L_{r-1})}$ \\
$d_{\infty, (L_{r-1})_{\R}^{\perp}} (v, \pi_{r-1}(\xii)) \leq \s_{r}$}\right\}\\
& \ll \det L_{r-1} \left(  \left(\frac{\s_{r}}{q/d_{r-1}}\right)^{n-(r-1)} s_{r}^{n+1-(r-1)} + s_{r}\right)\\
& \ll s_{1}\cdots s_{r-1} \left(  \left(\frac{\s_{r}}{q/d_{r-1}}\right)^{n-(r-1)} s_{r}^{n+1-(r-1)} + s_{r}\right).
\end{align}

In general, we have
\begin{align}
&\sum_{\substack{L_{j} \subset L_{j+1} \subset \Z^{n+1} \text{prim.} \\ \rank L_{j+1} = j+1 \\ \det L_{j+1} \asymp s_{1}\cdots s_{j+1}\\ 
\l_{1}(L_{j+1}/L_{j}) \asymp s_{j+1} \\ (g(\pi_{j+1}(\cc)), q)=d_{j+1} \\ d_{\infty}(\xii, L_{j+1}) \in I_{i_{j+1}}}}
1\\
&\leq
\# \left\{ v \in (\Z^{n+1}/L_{j})_{\prim} \ \middle|\  \txt{$\| v \| \ll s_{j+1}$, 
$\exists u \in (\Z/(d_{j+1}/d_{j})\Z)^{\times}$,\\ $v \equiv u(d_{j}^{-1}\pi_{j}(\cc)) \ \mod {(d_{j+1}/d_{j})(\Z^{n+1}/L_{j})}$ \\
$d_{\infty, (L_{j})_{\R}^{\perp}} (v, \pi_{j}(\xii)) \leq \s_{j+1}$}\right\}\\
& \ll s_{1}\cdots s_{j} \left(  \left(\frac{\s_{j+1}}{d_{j+1}/d_{j}}\right)^{n-j} s_{j+1}^{n+1-j} + s_{j+1}\right)
\end{align}
for $0 \leq j \leq r-1$.
We repeatedly use this bound and get
\begin{align}
&S_{r,n} \\
&\ll \sum_{\substack{1 \leq i_{1} \leq \cdots \\ \cdots \leq i_{r-1} \leq k+1}}
\sum_{d_{1}| \cdots |d_{r-1} |q}
\prod_{j=0}^{r-1} s_{1}\cdots s_{j} \left(  \left(\frac{\s_{j+1}}{d_{j+1}/d_{j}}\right)^{n-j} s_{j+1}^{n+1-j} + s_{j+1}\right)\\
&\ll \sum_{\substack{1 \leq i_{1} \leq \cdots \\ \cdots \leq i_{r-1} \leq k+1 \\ 
d_{1}| \cdots |d_{r-1} |q}}
\prod_{j=1}^{r} s_{j}^{r-j}
\left( \prod_{j=1}^{r} \left(\frac{\s_{j}}{d_{j}/d_{j-1}}\right)^{n-j+1} s_{j}^{n+2-j} 
+ \sum_{l=1}^{r} s_{l} \prod_{\substack{1 \leq j \leq r \\ j \neq l}} s_{j}^{n+2-j}   \right)\\
&\ll \sum_{\substack{1 \leq i_{1} \leq \cdots \leq i_{r-1} \leq k+1 \\
d_{1}| \cdots |d_{r-1} |q}}
\prod_{j=1}^{r} s_{j}^{r-j}
\left( \left(\frac{\s}{q}\right)^{n-r+1} \prod_{1 \leq j \leq r}s_{j}^{n+2-j} 
+ \sum_{l=1}^{r} s_{l} \prod_{\substack{1 \leq j \leq r \\ j \neq l}} s_{j}^{n+2-j}   \right)\\
&\ll \tau(q)^{r-1}\Bigl(\log\frac{2}{\s}\Bigr)^{r-1}
\prod_{j=1}^{r} s_{j}^{r-j}
\left( \left(\frac{\s}{q}\right)^{n-r+1} \prod_{1 \leq j \leq r}s_{j}^{n+2-j} 
+ \sum_{l=1}^{r} s_{l} \prod_{\substack{1 \leq j \leq r \\ j \neq l}} s_{j}^{n+2-j}   \right)
\end{align}

\end{proof}

%%%%%%%%%%%%%%%%%%%%%%%%%%%%%%%%%%%%%%%%
\begin{definition}
\label{def:Lrn}
Let us consider 
\begin{itemize}
    \item $n \in \Z_{\geq 2}$ and $r \in \Z$ with $1 \leq r \leq n+1$;
    \item $\s \in (0,1]$ and $\xii \in \R^{n+1} \setminus \{0\}$;
    \item $q \in \Z_{\geq 1}$ and $\cc \in \Z^{n+1}$ with $\gcd(g(\cc), q)=1$.
\end{itemize}
For real numbers $X,\Delta\ge0$, we let
\begin{align}
&\mathcal{L}_{r,n}(X,\Delta;\cc,q;\xii,\realeps)\\
&\coloneqq
\left\{\xx\in
\biggl(\astcup_{u\ \mod{q}}(u\cc+q\Z^{n+1})
\cap
\mathcal{C}_{n+1}(\xii,\realeps)
\cap\mathcal{B}_{n+1}(X)\biggr)
\midmid
\dd_{r}(\xx)\le\Delta\right\},
\end{align}
and
\[
L_{r,n}(X,\Delta;\cc,q;\xii,\realeps)
\coloneqq\#\mathcal{L}_{r,n}(X,\Delta;\cc,q;\xii,\realeps).
\]
\end{definition}

%%%%%%%%%%%%%%%%%%%%%%%%%%%%%%%%%%%%%%%%
\begin{lemma}
\label{lem:l2n_local}
Consider
\begin{itemize}
%%%%%
\item
An integer $n\ge3$.
%%%%%
\item
An integral vector $\cc\in\Z^{n+1}$ and $q \in \Z_{\geq 1}$ with $\gcd(g(\cc),q)=1$.
%%%%%
\item
A vector $\xii\in\R^{n+1}\setminus\{0\}$
and a real number $\s \in (0,1]$.
%%%%%
\item
Real numbers $X,\Delta\ge1$.
\end{itemize}
We then have
\begin{align}
L_{2,n}(X,\Delta;\cc,q;\xii,\realeps)
\ll
\tau(q)(\log\tfrac{2}{\realeps})(\log2\Delta)
\biggl(\biggl(\frac{\realeps}{q}\biggr)^{n}
%+\biggl(\frac{\realeps}{q}\biggr)^{n-1}\frac{\Delta^{\frac{1}{2}}}{X}
+\frac{\realeps}{q}\frac{1}{\Delta}
+\frac{1}{X}\biggr)
\Delta^{n}X^2
\end{align}
where the implicit constant depends only on $n$.
\end{lemma}
%%%%%%%%%%%%%%%%%%%%%%%%%%%%%%%%%%%%%%%%
\begin{proof}
%By the trivial bound $\dd_2(\xx)\ll\|\xx\|$, we may assume $\Delta\ll X$.
We may assume $\Delta\le X$
since othweise we have $X^{-1}\cdot \Delta^{n}X^{2}\gg X^{n+1}$
and so the assertion follows by the trivial bound
$L_{2,n}(X,\Delta;\cc,q;\xii,\realeps)\ll \#(\mathbb{Z}^{n+1}\cap\mathcal{B}_{n+1}(X))\ll X^{n+1}$.
We have
\begin{align}
&L_{2,n}(X,\Delta;\cc,q;\xii,\realeps)\\
&\ll
\astsum_{u\ \mod{q}}
\sum_{\substack{
\xx\in u\cc+q\Z^{n+1}\\
\xx\in\mathcal{C}_{n+1}(\xii,\realeps)\\
\|\xx\|\le X}}
\dsum{(d)}_{\substack{1\ll s_1\ll s_2\\s_1s_2\ll\Delta\\s_1\ll X}}
\sum_{\substack{%
\Lambda\in \mathcal{S}_{2,n}(s_1,s_2;\cc,q;\xii,\realeps)\\
\xx\in\Lambda}}1\\
&\ll
\dsum{(d)}_{\substack{1\ll s_1\ll s_2\\s_1s_2\ll\Delta\\s_1\ll X}}
\sum_{\Lambda\in \mathcal{S}_{2,n}(s_1,s_2;\cc,q;\xii,\realeps)}\\
&\hspace{20mm}\times
\astsum_{u\ \mod{q}}
\# \left(\L \cap (u\cc+q\Z^{n+1})
\cap\mathcal{C}_{n+1}(\xii,\realeps)
\cap\mathcal{B}_{n+1}(X) \right),
\end{align}
where the sum over $s_1,s_2$ is taken dyadically.
By \cref{lem:V_bound} and \cref{prop:lattice_count_local},
\begin{align}
\astsum_{u\ \mod{q}}
\# \left(\L \cap (u\cc+q\Z^{n+1})
\cap\mathcal{C}_{n+1}(\xii,\realeps)
\cap\mathcal{B}_{n+1}(X) \right)
\ll
\frac{\realeps}{q}\frac{X^2}{s_1s_2}+\frac{X}{s_1}.
\end{align}
Thus, combined with \cref{prop:Srn-bound}, we have
\begin{align}
&L_{2,n}(X,\Delta;\cc,q;\xii,\realeps)\\
&\ll
\tau(q)(\log\tfrac{2}{\realeps})
\sum_{\substack{1\ll s_1\ll s_2\\s_1s_2\ll\Delta\\s_1\ll X}}
\biggl(
\biggl(\frac{\realeps}{q}\biggr)^{n-1}s_1^{n+2}s_2^{n}
+s_1^2s_2^{n}+s_1^{n+2}s_2\biggr)
\biggl(\frac{\realeps}{q}\frac{X^2}{s_1s_2}+\frac{X}{s_1}\biggr).
\end{align}
Note that the summand in the above sum is increasing in $s_2$ and
\[
s_1\ll s_2\ \text{and}\ 
s_1s_2\ll\Delta\implies
s_1\ll\Delta^{\frac{1}{2}}\ \text{and}\ 
s_2\ll\frac{\Delta}{s_1}.
\]
Therefore, we have
\begin{align}
&L_{2,n}(X,\Delta;\cc,q;\xii,\realeps)\\
&\ll
\tau(q)(\log\tfrac{2}{\realeps})(\log2\Delta)\\
&\hspace{10mm}
\times\sum_{1\ll s_1\ll\Delta^{\frac{1}{2}}}
\biggl(
\biggl(\frac{\realeps}{q}\biggr)^{n-1}\Delta^{n}s_1^{2}
+s_1^{-(n-2)}\Delta^{n}+s_1^{n+1}\Delta\biggr)
\biggl(\frac{\realeps}{q}\frac{X^2}{\Delta}+\frac{X}{s_1}\biggr)\\
&\ll
\tau(q)(\log\tfrac{2}{\realeps})(\log2\Delta)\\
&\hspace{10mm}\times
\Biggl\{
\sum_{1\ll s_1\ll\Delta^{\frac{1}{2}}}
\biggl(
\biggl(\frac{\realeps}{q}\biggr)^{n-1}\Delta^{n}s_1^{2}
+s_1^{n+1}\Delta\biggr)
\biggl(\frac{\realeps}{q}\frac{X^2}{\Delta}+\frac{X}{s_1}\biggr)\\
&\hspace{50mm}
+
\sum_{1\ll s_1\ll\Delta^{\frac{1}{2}}}
s_1^{-(n-2)}\Delta^{n}
\biggl(\frac{\realeps}{q}\frac{X^2}{\Delta}+\frac{X}{s_1}\biggr)
\Biggr\}\\
&\ll
\tau(q)(\log\tfrac{2}{\realeps})(\log2\Delta)\\
&\hspace{10mm}\times
\Biggl\{
\biggl(
\biggl(\frac{\realeps}{q}\biggr)^{n-1}\Delta^{n+1}
+\Delta^{\frac{n+3}{2}}\biggr)
\biggl(\frac{\realeps}{q}\frac{X^2}{\Delta}+\frac{X}{\Delta^{\frac{1}{2}}}\biggr)
+
\Delta^{n}
\biggl(\frac{\realeps}{q}\frac{X^2}{\Delta}+X\biggr)
\Biggr\}.
\end{align}
By $n\ge3$ and $\Delta\ge1$, we have
\[
\Delta^{\frac{n+3}{2}}
\biggl(\frac{\realeps}{q}\frac{X^2}{\Delta}+\frac{X}{\Delta^{\frac{1}{2}}}\biggr)
\ll
\Delta^{n}
\biggl(\frac{\realeps}{q}\frac{X^2}{\Delta}+X\biggr).
\]
Therefore, we have
\begin{align}
&L_{2,n}(X,\Delta;\cc,q;\xii,\realeps)\\
&\ll
\tau(q)(\log\tfrac{2}{\realeps})(\log2\Delta)
\biggl(\biggl(\frac{\realeps}{q}\biggr)^{n}
+\biggl(\frac{\realeps}{q}\biggr)^{n-1}\frac{\Delta^{\frac{1}{2}}}{X}
+\frac{\realeps}{q}\frac{1}{\Delta}
+\frac{1}{X}\biggr)
\Delta^{n}X^2.
\end{align}
Furthermore, since $n\ge3$ and we are assuming $\Delta\le X$, we have
\[
\biggl(\frac{\realeps}{q}\biggr)^{n}
+
\frac{1}{X}
\ge
\biggl(\biggl(\frac{\realeps}{q}\biggr)^{n}\biggr)^{\frac{n-1}{n}}
\biggl(\frac{1}{X}\biggr)^{\frac{1}{n}}
=
\biggl(\frac{\realeps}{q}\biggr)^{n-1}
\frac{1}{X^{\frac{1}{n}}}
\ge
\biggl(\frac{\realeps}{q}\biggr)^{n-1}\frac{\Delta^{\frac{1}{2}}}{X}
\]
and so we obtain the result.
\end{proof}

%%%%%%%%%%%%%%%%%%%%%%%%%%%%%%%%%%%%%%%%
\section{Proof of the main result}

%%%%%%%%%%%%%%%%%%%%%%%%%%%%%%%%%%%%%%%%
\begin{lemma}
\label{lem:local_cond_translate}
Consider
\begin{itemize}
%%%%%
\item A finite set $S$ of finite places of $\Q$.
%%%%%
\item A tuple $\xi=(\xi_p)\in\prod_{p\in S\cup\{\infty\}}\P^{n}(\Q_p)$.
%%%%%
\item A tuple $\realeps=(\realeps_p)_{p\in S\cup\{\infty\}}$
such that
\begin{equation}
\textup{$\realeps_p=p^{-e_p}$ with $e_p\ge1$ for all $p\in S$}\and
\textup{$0<\realeps_{\infty}\le1$ for $p=\infty$.}
\end{equation}
\end{itemize}
Let
\[
q\coloneqq\prod_{p\in S}p^{e_p}.
\]
Consider
\[
x=(x_0:\cdots:x_n)\in\P^{n}(\Q)
\quad\text{with}\quad
\xx=(x_0,\ldots,x_n)\in\Z^{n+1}_{\prim}\setminus\{0\}.
\]
Then, there exists
\[
\cc\in\Z^{n+1}\setminus\{0\}
\and
\xii_{\infty}\in\R^{n+1}\setminus\{0\}
\]
such that $\gcd(\cc,q)=1$ and
\begin{equation}
\label{lem:local_cond_translate:iff}
\begin{aligned}
&d_p(x,\xi_p)\le\realeps_p
\ \text{for all $p\in S\cup\{\infty\}$}\\
&\iff
\xx\equiv u\cc\ \mod{q}\ \text{for some $u\in(\Z/q\Z)^{\times}$}
\and
\xx\in\mathcal{C}_{n+1}(\xii_{\infty},\realeps_{\infty}).
\end{aligned}
\end{equation}
\end{lemma}
%%%%%%%%%%%%%%%%%%%%%%%%%%%%%%%%%%%%%%%%
\begin{proof}
Let us write
\[
\xi_p=(c_{p,0}:\cdots:c_{p,n})\in\P^n(\Q_p)
\quad\text{with}\quad
\cc_{p}=(c_{p,0},\cdots,c_{p,n})\in\Q_p^{n+1}\setminus\{0\}
\]
for $p\in S$. By dilating $\cc_p$, we may assume
\begin{equation}
\label{lem:local_cond_translate:c_p_cond}
\cc_{p}\in\Z_p^{n+1}\setminus\{0\}
\and
\|\cc_{p}\|_{p}=1.
\end{equation}
%By using the weak approximation property of $\Q$,
%i.e.\ the Chinese remainder theorem componentwise,
By using the Chinese remainder theorem component-wise, we can get
\[
\cc=(c_0,\ldots,c_n)\in\Z^{n+1}
\]
such that
\begin{equation}
\label{lem:local_cond_translate:c_cond_raw}
\|\cc-\cc_{p}\|_p\le\realeps_p=p^{-e_p}
\quad\text{for all $p\in S$}.
\end{equation}
By \cref{lem:local_cond_translate:c_p_cond},
\cref{lem:local_cond_translate:c_cond_raw} and $e_p\ge1$, we have
\begin{equation}
\label{lem:local_cond_translate:c_cond_derived}
\|\cc\|_{p}=1\quad\text{for all $p\in S$}
\quad\text{and so}\quad
\gcd(q,\cc)=1.
\end{equation}
Furthermore, write
\[
\xi_{\infty}
=(\xi_{\infty,0}:\cdots:\xi_{\infty,n})\in\P^n(\R)
\quad\text{with}\quad
\xii_{\infty}=(\xi_{\infty,0},\cdots,\xi_{\infty,n})\in\R^{n+1}\setminus\{0\}.
\]
For such defined $\cc$ and $\xii_{\infty}$,
we prove \cref{lem:local_cond_translate:iff}.
It suffices to show the following:
\begin{enumerate}[label=\textup{(\roman*)}]
%%%%%
\item\label{lem:local_cond_translate:finite_p}
We have
\begin{equation}
\begin{aligned}
&d_p(x,\xi_p)\le\realeps_p
\ \text{for all $p\in S$}\\
&\iff
\xx\equiv u\cc\ \mod{q}\ \text{for some $u\in(\Z/q\Z)^{\times}$}.
\end{aligned}
\end{equation}
%%%%%
\item\label{lem:local_cond_translate:infinite_p}
We have
\[
d(x,\xi_{\infty})\le\realeps
\iff
\xx\in\mathcal{C}_{n+1}(\xii_{\infty},\realeps_{\infty}).
\]
\end{enumerate}
\medskip

%%%%%%%%%%%%%%%%%%%%%%%%%%%%%%%%%%%%%%%%
\noindent Proof of \cref{lem:local_cond_translate:finite_p}.
We first prove
\begin{equation}
\label{lem:local_cond_translate:finite_p:p}
d_p(x,\xi_p)\le\realeps_p
\iff
\xx\equiv u_p\cc\ \mod{p^{e_p}}
\quad\text{for some $u_p\in(\Z/p^{e_p}\Z)^{\times}$}
\end{equation}
for any $p\in S$. Since $\xx$ is primitive, by \cref{lem:local_cond_translate:c_p_cond},
we have $\|\xx\|_p=\|\cc\|_p=1$ and so
\begin{align}
d_p(x,\xi_p)\le\realeps_p
&\iff
\|\xx\wedge\cc_{p}\|_p\le\realeps_p\\
&\iff
\max_{0\le i<j\le n}|x_ic_{p,j}-x_jc_{p,i}|_p\le p^{-e_p}.
\end{align}
By \cref{lem:local_cond_translate:c_cond_raw},
we further have
\begin{equation}
\label{lem:local_cond_translate:finite_p:p_pre_equiv}
\begin{aligned}
d_p(x,\xi_p)\le\realeps_p
&\iff
\max_{0\le i<j\le n}|x_ic_{j}-x_jc_{i}|_p\le p^{-e_p}\\
&\iff
x_ic_j\equiv x_jc_i\ \mod{p^{e_p}}\quad\text{for all $0\le i<j\le n$}.
\end{aligned}
\end{equation}
By \cref{lem:local_cond_translate:c_cond_derived},
we can take $i_p\in\{0,\ldots,n\}$ with $(c_{i_p},p)=1$.
By taking $j=i_p$, we have
\begin{align}
d_p(x,\xi_p)\le\realeps_p
&\implies
x_i\equiv u_pc_i\ \mod{p^{e_p}}\quad\text{for all $0\le i\le n$}\\
&\implies
\xx\equiv u_p\cc\ \mod{p^{e_p}}.
\end{align}
This proves the implication $\Longrightarrow$ of \cref{lem:local_cond_translate:finite_p:p}.
The other implication is obvious by \cref{lem:local_cond_translate:finite_p:p_pre_equiv}.
For $u_p$ on the right hand side of \cref{lem:local_cond_translate:finite_p:p},
by the Chinese remainder theorem, we can take $u\in(\Z/q\Z)^{\times}$ with
\[
u\equiv u_p\ \mod{p^{e_p}}\quad\text{for all $p\in S$}.
\]
By \cref{lem:local_cond_translate:finite_p:p},
we then finally have
\begin{align}
&d_p(x,\xi_p)\le\realeps_p\ \text{for all $p\in S$}\\
&\iff
\xx\equiv u\cc\ \mod{p^{e_p}}
\ \text{for all $p\in S$ and some $u\in(\Z/q\Z)^{\times}$}\\
&\iff
\xx\equiv u\cc\ \mod{q}
\ \text{for some $u\in(\Z/q\Z)^{\times}$}.
\end{align}
This proves \cref{lem:local_cond_translate:finite_p}.
\medskip

%%%%%%%%%%%%%%%%%%%%%%%%%%%%%%%%%%%%%%%%
\noindent Proof of \cref{lem:local_cond_translate:infinite_p}.
We just have
\[
d(x,\xi_{\infty})\le\realeps_{\infty}
\iff
\frac{\|\xx\wedge\xii_{\infty}\|}{\|\xx\|\cdot\|\xii_{\infty}\|}\le\realeps_{\infty}
\iff
\xx\in\mathcal{C}_{n+1}(\xii_{\infty},\realeps_{\infty}).
\]
This completes the proof.
\end{proof}

%%%%%%%%%%%%%%%%%%%%%%%%%%%%%%%%%%%%%%%%
\begin{lemma}
\label{lem:leBoudec_bound}
Consider an integral lattice $\Lambda\subset\R^N$ of rank $r\ge2$.
For $A,M\ge1$ and $s\in\Z_{\ge0}$
such that $A,\lambda_r(\Lambda)\le M$, we have
\[
\sum_{1\le\nu<r}\frac{A^{\nu}}{\lambda_1(\Lambda)\cdots\lambda_{\nu}(\Lambda)}
\ll
A^{r}
\biggl(
\frac{1}{\det(\Lambda)}\biggl(\frac{M}{A}\biggr)^{s}
+\frac{1}{A}\biggl(\frac{1}{A}\biggr)^{s}
\biggr),
\]
where the implicit constant depends only on $r$.
\end{lemma}
%%%%%%%%%%%%%%%%%%%%%%%%%%%%%%%%%%%%%%%%
\begin{proof}
If $s\ge r$, then
\begin{align}
\sum_{1\le\nu<r}\frac{A^{\nu}}{\lambda_1(\Lambda)\cdots\lambda_{\nu}(\Lambda)}
&\le
\sum_{1\le\nu<r}\frac{M^{\nu}}{\lambda_1(\Lambda)\cdots\lambda_{\nu}(\Lambda)}\\
&\ll
\frac{M^{r}}{\lambda_1(\Lambda)\cdots\lambda_r(\Lambda)}
\asymp
A^{r}
\frac{1}{\det(\Lambda)}
\biggl(\frac{M}{A}\biggr)^{r}
\ll
A^{r}
\frac{1}{\det(\Lambda)}
\biggl(\frac{M}{A}\biggr)^{s}
\end{align}
by Minkowski's second theorem (\cref{lem:Minkowski_second}). Thus, we may assume $0\le s\le r-1$.

%%%%%%%%%%%%%%%%%%%%%%%%%%%%%%%%%%%%%%%%
We have
\[
\sum_{1\le\nu<r}\frac{A^{\nu}}{\lambda_1(\Lambda)\cdots\lambda_{\nu}(\Lambda)}
=
\sum_{1\le i\le s}\frac{A^{r-i}}{\lambda_1(\Lambda)\cdots\lambda_{r-i}(\Lambda)}
+\sum_{s+1\le i<r}\frac{A^{r-i}}{\lambda_1(\Lambda)\cdots\lambda_{r-i}(\Lambda)}.
\]
For the former sum, by Minkowski's second theorem (\cref{lem:Minkowski_second}) and $\lambda_r(\Lambda)\le M$,
\begin{align}
\sum_{1\le i\le s}\frac{A^{r-i}}{\lambda_1(\Lambda)\cdots\lambda_{r-i}(\Lambda)}
&=
\frac{A^{r}}{\lambda_1(\Lambda)\cdots\lambda_{r}(\Lambda)}
\sum_{i=1}^{s}\frac{\lambda_{r-i+1}(\Lambda)\cdots\lambda_{r}(\Lambda)}{A^{i}}\\
&\ll
\frac{A^{r}}{\det(\Lambda)}\biggl(\frac{M}{A}\biggr)^{s}.
\end{align}
For the latter sum, since $\Lambda$ is integral,
we have $\lambda_1(\Lambda)\ge1$ and so
\[
\sum_{s+1\le i<r}\frac{A^{r-i}}{\lambda_1(\Lambda)\cdots\lambda_{r-i}(\Lambda)}
\ll
\sum_{s+1\le i<r}A^{r-i}
\ll
A^{r-(s+1)}.
\]
By combining the above estimates, we obtain the lemma.
\end{proof}

\begin{theorem}
\label{thm:first_moment_local_cond}
Under the setting of \cref{lem:local_cond_translate}
with $n\ge d\ge2$ and $(n,d)\neq(2,2)$,
\begin{align}
\sum_{V\in\V_{d,n}(A)}N_V(B;\xi,\sigma)
=
\widetilde{C}_{d,n}(\xi,\sigma)
A^{N_{d,n}-1}B\bigl(1+O(R_{d,n}\cdot B^{\varepsilon})\bigr)
\end{align}
for $A,B\ge2$ provided
\[
%A\ge\max(\q,B^{\frac{1}{n(n+1-d)}}),\quad
A\ge\max(\q, B^{\frac{d+1}{(2n+1)(n+1-d)}},\q^{\frac{n-1}{2n-1}}B^{\frac{d}{(2n-1)(n+1-d)}}),\quad
B\ge\q^{n}\quad\text{with}\quad
\q\coloneqq\frac{q}{\sigma_{\infty}},
\]
where
\begin{align}
\widetilde{C}_{d,n}(\xi,\sigma)
&\coloneqq
\frac{V_{N_{d,n}-1}}{4\zeta(N_{d,n}-1)}
\mathfrak{W}_{d,n}(\xii_{\infty};\sigma_{\infty})
\frac{\varphi(q)}{J_{n+1}(q)\zeta(n+1)},\\
R_{d,n}
&\coloneqq
\q A^{-1}
+
\q^{n}B^{-1}
+
\q B^{-\frac{1}{n+1-d}}
+
A^{-1}B^{\frac{1}{n(n+1-d)}}
+
\widetilde{R}_{d,n}
\end{align}
with
\begin{align}
\widetilde{R}_{d,n}
\coloneqq
\left\{
\begin{array}{>{\displaystyle}cl}
0&
\text{if $A\ge B^{\frac{1}{n+1-d}}$},\\
A^{-(2n+1)}B^{\frac{d+1}{n+1-d}}
+
\q^{n-1}A^{-(2n-1)}B^{\frac{d}{n+1-d}}
&\text{if $A\le B^{\frac{1}{n+1-d}}$}
\end{array}
\right.
\end{align}
and the implicit constant depends only on $d,n,\varepsilon$.
\end{theorem}
%%%%%%%%%%%%%%%%%%%%%%%%%%%%%%%%%%%%%%%%
% \Yuta{It's better to introduce the condition
% \[
% A\ge\max(\q, B^{\frac{d+1}{(2n+1)(n+1-d)}},\q^{\frac{n-1}{2n-1}}B^{\frac{d}{(2n-1)(n+1-d)}})
% \]
% which covers the condition $A\ge B^{\frac{1}{n(n+1-d)}}$}
%%%%%%%%%%%%%%%%%%%%%%%%%%%%%%%%%%%%%%%%
\begin{proof}
Note that the assumption on the size of $A$ implies
\begin{equation}
A
\ge
B^{\frac{d+1}{(2n+1)(n+1-d)}}
\ge
B^{\frac{d+1}{3n(n+1-d)}}
\ge
B^{\frac{1}{n(n+1-d)}}.
\end{equation}
Thus, we have $A\ge\max(\q,B^{\frac{1}{n(n+1-d)}})$,
which we shall use below.
For brevity, write
\[
S
\coloneqq
\sum_{V\in\V_{d,n}(A)}N_V(B;\xi,\sigma).
\]
By \cref{lem:local_cond_translate}, we have
\begin{align}
S
&=
\frac{1}{4}
\sum_{\substack{
\aa\in\Zprim^{\mathcal{M}_{d,n}}\\
\|\aa\|\le A
}}
\astsum_{u\ \mod{q}}
\sum_{\substack{
\xx\in\Zprim^{n+1}\\
H(\xx)\le B\\
\xx\in u\cc+q\Z^{n+1}\\
\xx\in\cone{n+1}{\xii_{\infty}}{\realeps_{\infty}}\\
\aa\in\Lambda_{\nu_{d,n}(\xx)}
}}
1\\
&=
\frac{1}{4}
\astsum_{u\ \mod{q}}
\sum_{\substack{
\xx\in\Zprim^{n+1}\\
H(\xx)\le B\\
\xx\in u\cc+q\Z^{n+1}\\
\xx\in\cone{n+1}{\xii_{\infty}}{\realeps_{\infty}}
}}
\#(\Lambda_{\nu_{d,n}(\xx)}
\cap\Zprim^{\mathcal{M}_{d,n}}
\cap\ball{N_{d,n}}{A}).
\end{align}
By \cref{prop:lattice_count_primitive_local}
with $q=1$ and $\sigma=1$, we have
\[
S=T+O(E_{1}+E_{2}),
\]
where
\begin{align}
T
&\coloneqq
\frac{V_{N_{d,n}-1}}{4\zeta(N_{d,n}-1)}A^{N_{d,n}-1}
\astsum_{u\ \mod{q}}
\sum_{\substack{
\xx\in\Zprim^{n+1}\\
H(\xx)\le B\\
\xx\in u\cc+q\Z^{n+1}\\
\xx\in\cone{n+1}{\xii_{\infty}}{\realeps_{\infty}}
}}
\frac{1}{\det(\Lambda_{\nu_{d,n}(\xx)})},\\
E_{1}
&\coloneqq
\astsum_{u\ \mod{q}}
\sum_{\substack{
\xx\in\Zprim^{n+1}\\
H(\xx)\le B\\
\xx\in u\cc+q\Z^{n+1}\\
\xx\in\cone{n+1}{\xii_{\infty}}{\realeps_{\infty}}
}}
\sum_{1\le\nu<N_{d,n}-1}
\frac{A^{\nu}}{\lambda_{1}(\Lambda_{\nu_{d,n}(\xx)})\cdots\lambda_{\nu}(\Lambda_{\nu_{d,n}(\xx)})},\\
E_{2}
&\coloneqq
A\log A
\astsum_{u\ \mod{q}}
\sum_{\substack{
\xx\in\Zprim^{n+1}\\
H(\xx)\le B\\
\xx\in u\cc+q\Z^{n+1}\\
\xx\in\cone{n+1}{\xii_{\infty}}{\realeps_{\infty}}
}}
1
\end{align}

%%%%%%%%%%%%%%%%%%%%%%%%%%%%%%%%%%%%%%%%
For the main term $T$, by \cref{det_Lambda_nu_dn_x}
and \cref{lem:det_reciprocal_prim_sum}, we have
\begin{align}
T
&=
\widetilde{C}_{d,n}(\xi,\sigma)
A^{N_{d,n}-1}B
+O\bigl(\q^{-(n-1)}
A^{N_{d,n}-1}(B^{\frac{n-d}{n+1-d}}+\log B)
+A^{N_{d,n}-1}\bigr).
\end{align}
By \cref{lem:W_bound} and the bound $\varphi(q)\gg q^{1-\varepsilon}$,
we have $\widetilde{C}_{d,n}(\xi,\sigma)\gg\q^{-n-\varepsilon}$ and so
\begin{equation}
T
=
\widetilde{C}_{d,n}(\xi,\sigma)
A^{N_{d,n}-1}B
\Bigl(1+O\bigl((
\q B^{-\frac{1}{n+1-d}}
+\q^{n}B^{-1}
)B^{\varepsilon}\bigr)\Bigr),
\end{equation}
where we used $B\ge\q^{n}$.

%%%%%%%%%%%%%%%%%%%%%%%%%%%%%%%%%%%%%%%%
For the error term $E_{1}$, let us introduce
\[
\mu(\xx)\coloneqq n\frac{\|\xx\|}{\dd_2(\xx)}
\]
and we further dissect as
\[
E_{1}=E_{11}+E_{12},
\]
where
\begin{align}
E_{11}
&\coloneqq
\astsum_{u\ \mod{q}}
\sum_{\substack{
\xx\in\Zprim^{n+1}\\
H(\xx)\le B\\
\xx\in u\cc+q\Z^{n+1}\\
\xx\in\cone{n+1}{\xii_{\infty}}{\realeps_{\infty}}\\
\lambda_{N_{d,n}-1}(\Lambda_{\nu_{d,n}(\xx)})\le A
}}
\sum_{1\le\nu<N_{d,n}-1}
\frac{A^{\nu}}{\lambda_{1}(\Lambda_{\nu_{d,n}(\xx)})\cdots\lambda_{\nu}(\Lambda_{\nu_{d,n}(\xx)})},\\
E_{12}
&\coloneqq
\astsum_{u\ \mod{q}}
\sum_{\substack{
\xx\in\Zprim^{n+1}\\
H(\xx)\le B\\
\xx\in u\cc+q\Z^{n+1}\\
\xx\in\cone{n+1}{\xii_{\infty}}{\realeps_{\infty}}\\
\lambda_{N_{d,n}-1}(\Lambda_{\nu_{d,n}(\xx)})>A
}}
\sum_{1\le\nu<N_{d,n}-1}
\frac{A^{\nu}}{\lambda_{1}(\Lambda_{\nu_{d,n}(\xx)})\cdots\lambda_{\nu}(\Lambda_{\nu_{d,n}(\xx)})}.
\end{align}

%%%%%%%%%%%%%%%%%%%%%%%%%%%%%%%%%%%%%%%%
For the sum $E_{11}$, 
by using \cref{det_Lambda_nu_dn_x},
\cref{lem:Minkowski_second}
and \cref{lem:BBS_rank_d2_bound},
\begin{align}
E_{11}
\ll
A^{N_{d,n}-2}
\astsum_{u\ \mod{q}}
\sum_{\substack{
\xx\in\Zprim^{n+1}\\
H(\xx)\le B\\
\xx\in u\cc+q\Z^{n+1}\\
\xx\in\cone{n+1}{\xii_{\infty}}{\realeps_{\infty}}
}}
\frac{\mu(\xx)}{\|\xx\|^d}.
\end{align}
We then dissect the sum dyadically with writing
\begin{equation}
\label{thm:first_moment_local_cond:dyadic_variable}
\|\xx\|\asymp X
\and
\mu(\xx)\asymp M
\end{equation}
as
\[
E_{11}
\ll
A^{N_{d,n}-2}
\dsum{(d)}_{\substack{
1\ll X^{n+1-d}\ll B\\
1\ll M\ll X
}}
\frac{M}{X^d}
L_{2,n}\biggl(X,\frac{CX}{M};\cc,q;\xii_{\infty},\realeps_\infty\biggr),
\]
with some $1\le C\ll1$.
By \cref{prop:lattice_count_local}, \cref{lem:l2n_local}, we have
\[
L_{2,n}\biggl(X,\frac{CX}{M};\cc,q;\xii_{\infty},\realeps_\infty\biggr)
\ll
(F_{1}+F_{2}+F_{3}+F_{4})
B^{\varepsilon}
+X,
\]
where
\begin{equation}
\label{thm:first_moment_local_cond:F_def}
\begin{aligned}
F_{1}
=F_{1}(X,M)
&\coloneqq
\q^{-n}X^{n+1}\min(M^{-n}X,1),\\
F_{2}
=F_{2}(X,M)
&\coloneqq
\q^{-n}X^{n+1}\min(\q M^{-(n+\frac{1}{2})}X^{\frac{1}{2}},1),\\
F_{3}
=F_{3}(X,M)
&\coloneqq
\q^{-n}X^{n+1}\min(\q^{n-1}M^{-(n-1)},1),\\
F_{4}
=F_{4}(X,M)
&\coloneqq
\q^{-n}X^{n+1}\min(\q^{n}M^{-n},1).
\end{aligned}
\end{equation}
For $1\ll M\ll X$, we then have
\begin{align}
M\cdot F_{1}(X,M)
&\ll
\q^{-n}X^{n+1+\frac{1}{n}},\\
M\cdot F_{2}(X,M)
&\ll
\q^{-n+\frac{2}{2n+1}}X^{n+1+\frac{1}{2n+1}},\\
M\cdot F_{3}(X,M)
&\ll
\q^{-n+1}X^{n+1},\\
M\cdot F_{4}(X,M)
&\ll
\q^{-n+1}X^{n+1},\\
M\cdot X
&\ll
X^{2}.
\end{align}
Also, we have
\[
\q^{\frac{2}{2n+1}}
X^{\frac{1}{2n+1}}
\le
\q^{\frac{2}{2n+1}}
(X^{\frac{1}{n}})^{\frac{2n-1}{2n+1}}
\le
X^{\frac{1}{n}}
+
\q.
\]
We thus have
\begin{align}
E_{11}
&\ll
\q^{-n}
A^{N_{d,n}-2}
\dsum{(d)}_{1\ll X^{n+1-d}\ll B}
(X^{n+1-d}
(
X^{\frac{1}{n}}+\q
)
+\q^{n}X^{-(d-2)})
B^{\varepsilon}\\
&\ll
\q^{-n}A^{N_{d,n}-1}B
(
A^{-1}B^{\frac{1}{n(n+1-d)}}
+
\q A^{-1}
+
\q^{n}(AB)^{-1})
B^{\varepsilon}.
\end{align}
This completes the estimate for $E_{11}$.

%%%%%%%%%%%%%%%%%%%%%%%%%%%%%%%%%%%%%%%%
We next bound $E_{12}$.
For $\xx$ counted in $E_{12}$, by \cref{lem:BBS_rank_d2_bound}, we have
\[
A
<\lambda_{N_{d,n}-1}(\Lambda_{\nu_{d,n}(\xx)})
\le\|\xx\|
\le B^{\frac{1}{n+1-d}}.
\]
Thus, we may assume $A\le B^{\frac{1}{n+1-d}}$ since otherwise the sum $E_{12}$ is empty.
Also, for $\xx$ counted in $E_{12}$, by \cref{lem:BBS_rank_d2_bound},
we have $\mu(\xx)>A$.
We dissect dyadically with variables \cref{thm:first_moment_local_cond:dyadic_variable}
and use \cref{lem:leBoudec_bound} to get
\[
E_{12}
\ll
A^{N_{d,n}-1}
\dsum{(d)}_{\substack{
1\ll X^{n+1-d}\ll B\\
A\le M\ll X
}}
G\cdot
L_{2,n}\biggl(X,\frac{CX}{M};\cc,q;\xii_{\infty},\realeps_\infty\biggr)
\]
with some $1\le C\ll1$, where
\[
G
\coloneqq
\min_{s\in\mathbb{Z}_{\ge0}}
G_{s}
\quad\text{with}\quad
G_{s}
\coloneqq
\biggl(
\frac{1}{X^d}\biggl(\frac{M}{A}\biggr)^{s}
+
\biggl(\frac{1}{A}\biggr)^{s+1}
\biggr).
\]
By using \cref{lem:l2n_local}, we have
\[
E_{12}
\ll
A^{N_{d,n}-1}
\dsum{(d)}_{\substack{
1\ll X^{n+1-d}\ll B\\
A\le M\ll X
}}
G(\widetilde{F}_1+\widetilde{F}_2+\widetilde{F}_3+\widetilde{F}_4)
B^{\varepsilon},
\]
where
\begin{equation}
\label{thm:first_moment_local_cond:F_def_nomin}
\begin{aligned}
\widetilde{F}_{1}
&\coloneqq
\q^{-n}X^{n+1}\cdot M^{-n}X,\\
\widetilde{F}_{2}
&\coloneqq
\q^{-n}X^{n+1}\cdot \q M^{-(n+\frac{1}{2})}X^{\frac{1}{2}},\\
\widetilde{F}_{3}
&\coloneqq
\q^{-n}X^{n+1}\cdot\q^{n-1}M^{-(n-1)},\\
\widetilde{F}_{4}
&\coloneqq
\q^{-n}X^{n+1}\cdot\q^{n}M^{-n}.
\end{aligned}
\end{equation}
In the sum $E_{12}$, by the assumption $A\ge\q$, we have $\q\le A\le M\ll X$ and so
\[
\widetilde{F}_{2}
=
\q(MX)^{-\frac{1}{2}}\widetilde{F}_{1}
\ll
\widetilde{F}_{1}
\and
\widetilde{F}_{4}
=
\q M^{-1}\widetilde{F}_{3}
\ll\widetilde{F}_{3}.
\]
We thus simply have
\[
E_{12}
\ll
A^{N_{d,n}-1}\q^{\varepsilon}
\dsum{(d)}_{\substack{
1\ll X^{n+1-d}\ll B\\
A\le M\ll X
}}
G(\widetilde{F}_1+\widetilde{F}_3)B^{\varepsilon},
\]
We use the bound $G\le G_{n}$ for $G\widetilde{F}_{1}$
and $G\le G_{n-1}$ for $G\widetilde{F}_{3}$.
In the expressions $G_{n}\widetilde{F}_{1}$ and $G_{n-1}\widetilde{F}_{3}$
the exponents of $X$ is $\ge n+1-d$ and the exponents of $M$ is $\le 0$.
Thus, we have
\begin{align}
E_{12}
&\ll
A^{N_{d,n}-1}B^{\varepsilon}
\bigl(G_{n}\widetilde{F}_1+G_{n-1}\widetilde{F}_3\bigr)\big\vert_{X=B^{\frac{1}{n+1-d}}, M=A},\\
&=
\q^{-n}
A^{N_{d,n}-1}B\\
&\hspace{10mm}\times
\Bigl(
(1+A^{-(n+1)}B^{\frac{d}{n+1-d}})
A^{-n}B^{\frac{1}{n+1-d}}
\\
&\hspace{40mm}
+
(1+A^{-n}B^{\frac{d}{n+1-d}})
\q^{n-1}A^{-(n-1)}
\Bigr)
B^{\varepsilon}.
\end{align}
The assumption $A\ge\max(\q,B^{\frac{1}{n(n+1-d)}})$ implies
\[
A^{-n}B^{\frac{1}{n+1-d}}
\le
A^{-1}B^{\frac{1}{n(n+1-d)}}
\and
\q^{n-1}A^{-(n-1)}
\le
\q A^{-1}
\]
and so the contribution of terms
\[
A^{-n}B^{\frac{1}{n+1-d}},\quad
\q^{n-1}A^{-(n-1)}
\]
are already covered by $R_{d,n}$. The remaining terms contribute
\[
\q^{-n}
A^{N_{d,n}-1}B
(
A^{-(2n+1)}B^{\frac{d+1}{n+1-d}}
+
\q^{n-1}A^{-(2n-1)}B^{\frac{d}{n+1-d}})
B^{\varepsilon}
\]
which is $\widetilde{R}_{d,n}$ when $A\le B^{\frac{1}{n+1-d}}$.

%%%%%%%%%%%%%%%%%%%%%%%%%%%%%%%%%%%%%%%%
Finally, for $E_{2}$, we just use \cref{prop:lattice_count_local} with $B\ge\q^{n}$ to obtain
\begin{align}
E_{2}
&\ll
A\log A
(\q^{-n}B^{1+\frac{d}{n+1-d}}+B^{\frac{1}{n+1-d}})\\
&\ll
\q^{-n}A^{N_{d,n}-1}B
\cdot
A^{-(N_{d,n}-3)}B^{\frac{d}{n+1-d}}.
\end{align}
Note that we have $N_{d,n}\ge nd+3$.
Indeed,
\[
N_{d,n}
=
\biggl(\frac{n}{d}+1\biggr)\cdots\biggl(\frac{n}{2}+1\biggr)
\biggl(\frac{n}{1}+1\biggr)
\ge
2^{d-2}\cdot\frac{5}{2}(n+1)
\ge
\frac{5}{4}d(n+1)
\ge
dn+4.
\]
Therefore, by using the assumption $A\ge B^{\frac{1}{n(n+1-d)}}$, we have
\[
A^{-(N_{d,n}-3)}B^{\frac{d}{n+1-d}}
\le
A^{-nd}B^{\frac{d}{n+1-d}}
\le
A^{-1}B^{\frac{1}{n(n+1-d)}}
\]
and so the contribution of $E_{2}$ is covered by $R_{d,n}$.
By combining the above estimates,
we obtain the theorem.
\end{proof}

%%%%%%%%%%%%%%%%%%%%%%%%%%%%%%%%%%%%%%%%
\begin{remark}
Under the assumptions $n\ge d\ge 2$ and $B\ge\q^{n}$, we have
\begin{gather}
B^{\frac{d+1}{(2n+1)(n+1-d)}}
\le
B^{\frac{1}{n+1-d}},\\
\q^{\frac{n-1}{2n-1}}B^{\frac{d}{(2n-1)(n+1-d)}}
\le
B^{\frac{n-1}{n(2n-1)}+\frac{d}{(2n-1)(n+1-d)}}
\le
B^{\frac{n-1+d}{(2n-1)(n+1-d)}}
\le
B^{\frac{1}{n+1-d}}.
\end{gather}
Therefore, in \cref{thm:first_moment_local_cond},
we can replace the assumptions on the size of $A,B,\q$
to $A\ge B^{\frac{1}{n+1-d}}$ and $B\ge\q^{n}$
and then remove the error term $\widetilde{R}_{n,d}$. 
\end{remark}

%%%%%%%%%%%%%%%%%%%%%%%%%%%%%%%%%%%%%%%%
\begin{remark}
By choosing the parameter $s$ in \cref{lem:leBoudec_bound} more efficiently,
we can improve the error term in \cref{thm:first_moment_local_cond}
to get a wider admissible range for $A,B,\q$.
However, we took $s=n$ and $s=n-1$ for simplicity.
% and the possible improvement is not so significant.
% \Yuta{To be checked.}
\end{remark}

%%%%%%%%%%%%%%%%%%%%%%%%%%%%%%%%%%%%%%%%
We now simplify \cref{thm:first_moment_local_cond} to obtain \cref{mainthm:numerator}:
\begin{proof}[Proof of \cref{mainthm:numerator}]
It suffices to bound the error terms of \cref{thm:first_moment_local_cond}
by those of \cref{mainthm:numerator} under the assumptions on the size of $A,B,\q$
in \cref{thm:first_moment_local_cond}.
By $B\ge\q^{n}$, we have
\[
\q^{n}B=(\q B^{-\frac{1}{n}})^{n}\le \q B^{-\frac{1}{n}}.
\]
Also, since $d\ge2$, we have
\[
\q B^{-\frac{1}{n+1-d}}\le\q B^{-\frac{1}{n}}
\quad\text{and}\quad
A^{-1}B^{\frac{1}{n(n+1-d)}}
=
A^{-1}B^{\frac{3}{3n(n+1-d)}}
\le
A^{-1}B^{\frac{d+1}{(2n+1)(n+1-d)}}.
\]
Finally, since $A\ge B^{\frac{d+1}{(2n+1)(n+1-d)}}$, we have
\[
A^{-(2n+1)}B^{\frac{d+1}{n+1-d}}
=
(A^{-1}B^{\frac{d+1}{(2n+1)(n+1-d)}})^{2n+1}
\le
A^{-1}B^{\frac{d+1}{(2n+1)(n+1-d)}}.
\]
Thus, any error term of \cref{thm:first_moment_local_cond}
is covered by those of \cref{mainthm:numerator}.
\end{proof}

\section{Lang-Weil estimate and reducible locus}

In this section we collect some results on
the existence of $\Q_p$-points of hypersurfaces 
that can be easily deduced from the Lang-Weil estimate.
The contents of this section are well-known to experts
(cf.\ proof of \cite[Theorem 3.6]{PoonenVoloch}), 
but we give the proof for the completeness. 

\begin{theorem}[{Lang--Weil~\cite[Theorem~1]{LangWeil:VarietyFiniteField}}]\label{Lang-Weil1}
Let $n,d, r \in \Z_{\geq 0}$ and suppose $d > 0$.
Then there is a constant $A(n,d,r) > 0$ depending only on $n,d,r$ such that
for any prime power $q$ and geometrically irreducible closed subvariety $X \subset \P^{n}_{\F_{q}}$ 
with $\dim X = r$ and $\deg X = d$, we have
\begin{align}
|\#X(\F_{q}) - q^{r}| \leq (d-1)(d-2)q^{r-1/2} + A(n,d,r)q^{r-1}.
\end{align}
\end{theorem}

\begin{theorem}[{Lang--Weil~\cite[Lemma~1]{LangWeil:VarietyFiniteField}}]\label{Lang-Weil2}
Let $n,d, r \in \Z_{\geq 0}$ and suppose $d > 0$.
Then there is a constant $A_{1}(n,d,r) > 0$ depending only on $n,d,r$ such that
for any prime power $q$ and irreducible closed subvariety $X \subset \P^{n}_{\F_{q}}$ 
with $\dim X = r$ and $\deg X \leq d$, we have
\begin{align}
\#X(\F_{q})  \leq A_{1}(n,d,r)q^{r}.
\end{align}
\end{theorem}

\begin{lemma}\label{lem:existence-smooth-points}
    Let $n\in \Z_{\geq 2}$ and $d \in \Z_{\geq 1}$.
    Then there is $P_0 \geq 1$ such that for every prime number $p > P_0$
    and every geometrically integral hypersurface $X \subset \P^n_{\F_p}$ of degree $d$,
    $X$ has a smooth $\F_p$-point.
\end{lemma}
\begin{proof}
    Write $X = V_+(f)$, where $f \in \F_p[x_0,\dots,x_n]$ is a homogeneous of degree $d$.
    Let $f_i = \partial f / \partial x_i$.
    Since $X$ is geometrically integral, the regular locus of $X_{\overline{\F}_p}$
    is non-empty open. Thus if $p > d$, we have $\dim V_+(f_0,\dots f_n) \leq n-2$.
    Also there are $C, D \geq 1$ depending only on $d,n$ such that
    \begin{align}
        &\text{the number of irreducible components of $V_+(f_0,\dots, f_n)$} \leq C\\
        &\text{degree of irreducible components of $V_+(f_0,\dots, f_n)$} \leq D.
    \end{align}
    By \cref{Lang-Weil2}, we have
    \begin{align}
        \# V_+(f_0,\dots,f_n)(\F_p) \leq C \max\{A_1(n,D,r) \mid r \leq n-2\} p^{n-2}.
    \end{align}
    By \cref{Lang-Weil1}, we have 
    \begin{align}
        \#X(\F_p) \geq p^{n-1} - (d-1)(d-2)p^{n-1-1/2} - A(n,d,n-1)p^{n-2}.
    \end{align}
    Therefore, if $p$ is large enough we have 
    \begin{align}
        \#X(\F_p) > \# V_+(f_0,\dots,f_n)(\F_p)
    \end{align}
    and we are done.
\end{proof}

Now let $n\in \Z_{\geq 0}$ and $d \in \Z_{\geq 1}$.
Recall we identify (the ring valued points of) $\A^{ N_{d,n}}_{\Z}$ with the set of homogeneous polynomials 
of degree $d$ in $n+1$ variables.
Then we have the following morphism of schemes corresponding to the product of polynomials:
\begin{align}
p_{d',d''} \colon \A^{ N_{d',n}}_{\Z} \times_{\Z} \A^{ N_{d'',n}}_{\Z} \longrightarrow \A^{ N_{d,n}}_{\Z}
\end{align}
where $d' + d'' = d$.
\begin{definition}
We set
\begin{align}
\NIP_{d,n} = \bigcup_{\substack{d'+d'' = d \\ d',d''\geq 1}} \text{ $\Bigl($scheme theoretic image of $p_{d',d''}$ $\Bigr)$} \subset \A^{ N_{d,n}}_{\Z}.
\end{align}
We equip $\NIP_{d,n}$ with the reduced scheme structure.
Here $\NIP$ stands for Non-Irreducible Polynomials.
\end{definition}

\begin{lemma}\label{dim-of-NIP}
Let $n, d \in \Z_{\geq 2}$ and suppose $(n,d) \neq (2,2), (2,3)$.
Then $\NIP_{d,n}$ is flat over $\Z$ and 
\begin{align}
&\dim \NIP_{d,n} \leq  N_{d,n} -1\\
&\dim (\NIP_{d,n})_\Q \leq  N_{d,n} -2
\end{align}
where $(\NIP_{d,n})_\Q$ is the generic fiber of $\NIP_{d,n}$
over $\Spec \Z$.
\end{lemma}
\begin{proof}
Let $X_{d',d''} = \Bigl($scheme theoretic image of $p_{d',d''}$ $\Bigr)$.
Since 
\begin{align}
    &\dim \NIP_{d,n} = \max\{ \dim X_{d',d''} \mid d'+d'' = d, d',d''\geq 1 \}\\
    &\dim (\NIP_{d,n})_\Q = \max\{ \dim (X_{d',d''})_\Q \mid d'+d'' = d, d',d''\geq 1 \},
\end{align}
to bound the dimension it is enough to show
\begin{align}
    \dim X_{d',d''} \leq N_{d,n}-1 \quad \text{and} \quad
    \dim (X_{d',d''})_\Q \leq N_{d,n}-2
\end{align}
for all $d',d' \geq 1$ with $d'+d'' = d$.
Note that $X_{d',d''}$ is an integral scheme because so is
$\A^{ N_{d',n}}_{\Z} \times_{\Z} \A^{ N_{d'',n}}_{\Z}$.
Therefore the generic point of $X_{d',d''}$ is mapped to the generic point of $\Spec \Z$.
By the definition of $\NIP_{d,n}$, its irreducible components can be found among $X_{d',d''}$ and hence $\NIP_{d,n}$ is flat over $\Spec \Z$.

Now, since the structure morphism 
$\pi \colon X_{d',d''} \longrightarrow \Spec \Z$ is flat,
for any scheme point $x \in X_{d',d''}$, we have
\begin{align}
    \dim \mathcal{O}_{X_{d',d''}, x} = \dim \mathcal{O}_{\Spec \Z, \pi(x)} 
    + \dim \mathcal{O}_{\pi^{-1}(\pi(x)), x}.
\end{align}
Thus 
\begin{align}
    \dim X_{d',d''} 
    \leq 1 
    + \max\{\dim \pi^{-1}(\pi(x)) \mid x \in X_{d',d''}\}.
\end{align}
Since $\pi$ is flat, all the non-empty fibers of $\pi$ have the same dimension
(cf.\ \cite[Theorem 14.116]{GortzWedhorn}).
Thus it is enough to show that the generic fiber of $\pi$ has dimension at most $N_{d,n}-2$.
Since the generic fiber $(X_{d',d''})_{\Q}$ is the scheme theoretic image of
the base change $(p_{d',d''})_{\Q}$ of $p_{d',d''}$, it is enough to show that
\begin{align}
    \dim \A_{\Q}^{N_{d',n}} \times_{\Q} \A_{\Q}^{N_{d'',n}} = N_{d',n} + N_{d'',n} \leq N_{d,n} -2,
\end{align}
or equivalently,
\begin{align}
    \binom{n + d'}{n} + \binom{n+ d''}{n} \leq \binom{n+d}{n} -2.
\end{align}
Noticing
\begin{align}
    &\binom{n+d'}{n} = \binom{n+d'-1}{n-1} + \binom{n+d'-2}{n-1} + \cdots +\binom{n}{n-1} + 1\\
    &\binom{n+d}{n} = \binom{n+d-1}{n-1} + \cdots + \binom{n+d''}{n-1} + \binom{n+d''}{n},
\end{align}
we get
\begin{align}
    \binom{n+d}{n} - \binom{n+d'}{n} - \binom{n+d''}{n} \geq 
    d'\binom{n}{n-2}d'' -1
\end{align}
and the right hand side is at least $2$ under our assumption.
\end{proof}

\begin{proposition}\label{Lang-Weil-range}
Let $n \in \Z_{\geq 2}$ and $d \in \Z_{\geq 1}$.
There is $P_{0} \geq 1$ with the following property. 
For all $p > P_{0}$ and $\aa \in \Z_{p}^{ N_{d,n}}$,
if $\aa\ \mod p \in \F_{p}^{ N_{d,n}} \setminus \NIP_{d,n}(\F_{p})$, then we have
\begin{align}
V_{+}(f_{\aa})(\Q_{p}) \neq  \emptyset.
\end{align}
\textup{(}cf.\ \cref{def:monomial_homog_poly_hypsurf} for the definition of the notation $f_{\aa}$.\textup{)}
\end{proposition}

\begin{proof}
Let $P_0$ be as in \cref{lem:existence-smooth-points}.
If $\aa\ \mod p \in \F_{p}^{ N_{d,n}} \setminus \NIP_{d,n}(\F_{p})$,
then $V_+(f_{\aa\ \mod p})$ is geometrically integral.
Then by \cref{lem:existence-smooth-points}, $V_+(f_{\aa\ \mod p})$
has a smooth $\F_p$-point and thus $V_+(f_{\aa})$ has a $\Q_p$-point by the lifting property.
\end{proof}

%%%%%%%%%%%%%%%%%%%%%%%%%%%%%%%%%%%%%%%
%%%%%%%%%%%%%%%%%%%%%%%%%%%%%%%%%%%%%%%
% \if0
% \section{Lattice point counting}
% \Yohsuke{cite previous section}
% \begin{proposition}\label{lattice-point-counting-for-Vdn}
% Let $N \in \Z_{\geq 2}$, $q \in \Z_{\geq 2}$.
% Let $K \subset \R^{N}\times \R^{m}$ be a semi-algebraic set such that $\R K_{t} \subset K_{t}$ for all $t \in \R^{m}$.
% Let $\cc \in \Z^{N}$ be such that $\gcd(\cc, q)=1$.
% Let $A \in \R_{\geq 1}$.
% Then we have
% \begin{align}
% &\sum_{u \in (\Z/q\Z)^{\times}} \# \left(  \Z^{N}_{\prim} \cap K_{t} \cap B_{N}(A) \cap (u\cc + q \Z^{N}) \right)\\
% &= \frac{\f(q) \mu_{S^{N-1}}(K_{t}\cap S^{N-1})V_{N}}{\zeta(N)J_{N}(q)}A^{N} + O\left(\frac{\f(q)A^{N-1}}{q^{N-1}} + A \log 2A\right)
% \end{align}
% where the implicit constant depends only on $N, m, K$.
% Here
% \begin{align}
% &V_{N} = \vol(B_{N}(1))\\
% &J_{N}(q) = q^{N}\prod_{p|q}\left(1- \frac{1}{p^{N}}\right).
% \end{align}
% \Yohsuke{define them before}
% \end{proposition}
% \fi
%%%%%%%%%%%%%%%%%%%%%%%%%%%%%%%%%%%%%%%
%%%%%%%%%%%%%%%%%%%%%%%%%%%%%%%%%%%%%%%

\section{Newton method}

In this section, we recall Newton method over both
archimedean and non-archimedean fields.
These can be found in many literature, but we include proofs in the appendix for completeness.

\begin{proposition}\label{prop:newton-method-arch}
Let $K$ be $\R$ or $\C$ with the usual absolute value $|\cdot|$.
Let $d \in \Z_{>0}$ and $B \in \R_{>0}$.
Then there are positive real numbers $C, C' >0$ with the following properties.
For any polynomial $f \in K[t]$ with $\deg f = d$ and any $ \alpha_{0} \in K$ with $| \alpha_{0}| \leq B$,
if 
\begin{align}
|f( \alpha_{0})| \leq \frac{C}{|f|}|f'( \alpha_{0})|^{2} \quad \text{and} \quad f'( \alpha_{0}) \neq 0,
\end{align}
then there is $ \alpha \in K$ such that 
\begin{align}
f( \alpha)=0 \quad \text{and}\quad
| \alpha - \alpha_{0}| \leq C'\left| \frac{f( \alpha_{0})}{f'( \alpha_{0})}\right|.
\end{align}
Here $|f| := \max\{ |a| \mid \text{$a$ is a coefficient of $f$}\}$.
Moreover, we can take $C'=2$.
\end{proposition}
\begin{proof}
This follows directly from more general \cref{lem:Hensel_archimedean}.
Note that the $F$ in \cref{lem:Hensel_archimedean}
with $|\alpha_0| \leq B$ is bounded by $|f|$ times a constant
depending only on $d$ and $B$.
\end{proof}

\begin{proposition}\label{prop:newton-method-na}
Let $(K,|\cdot |)$ be a field with complete non-archimedean absolute value.
Let $R = \{a \in K \mid |a|\leq 1\}$.
For any polynomial $f \in R[t]$ and $ \alpha_{0} \in R$, if
\begin{align}
|f( \alpha_{0})| < |f'( \alpha_{0})|^{2},
\end{align}
then there is $ \alpha \in R$ such that
\begin{align}
f( \alpha)=0 \quad \text{and}\quad
| \alpha - \alpha_{0}| \leq \left| \frac{f( \alpha_{0})}{ f'( \alpha_{0})} \right|.
\end{align}
\end{proposition}
\begin{proof}
See \cref{lem:newton-method-na-appendix}.
\end{proof}

\begin{lemma}\label{lem:applyingNtn}
Let $n, d \in \Z_{>0}$.
Let $p$ be a prime number.
Let $\aa \in \Z_{\rm prim}^{N_{d,n}}$ be an arbitrary primitive integer vector
and $f \in \Z[x_{0},\dots ,x_{n}]$ be the corresponding homogeneous polynomial of degree $d$.
For any $\xi \in \Z_{p}^{n+1}$ with $\|\xi \|_{p}=1$, $e \in \Z_{>0}$, and $l \in \Z_{\geq 0}$, 
if
\begin{gather}
|f(\xi)|_{p} \leq p^{-e},\quad
p^{-(e-l)} < 1,\\
p^{-e} < \max_{0 \leq i \leq n}\left| \frac{ \partial f}{ \partial x_{i}}(\xi)\right|_{p}^{2},\quad
p^{-l} \leq \max_{0 \leq i \leq n}\left| \frac{ \partial f}{ \partial x_{i}}(\xi)\right|_{p},
\end{gather}
then there is $ \widetilde{\xi} \in \Z_{p}^{n+1}$ with $\| \widetilde{\xi}\|_{p}=1$ such that
\begin{align}
f( \widetilde{\xi}) =0
\quad\text{and}\quad
d_{p}(\xi, \widetilde{\xi}) \leq p^{-(e-l)}.
\end{align}
\end{lemma}

\begin{proof}

Without loss of generality, we may assume
\begin{align}
    \max_{0 \leq i \leq n}\left| \frac{\partial f}{ \partial x_{i}}(\xi) \right|_p =
    \left| \frac{\partial f}{ \partial x_{0}}(\xi) \right|_p.
\end{align}
Let us write $\xi=(\xi_0, \dots, \xi_n)$ and set $g(x_0) := f(x_0,\xi_1,\dots, \xi_n)$.
Then $g(x_0)$ is a polynomial with coefficient in $\Z_p$ because $\xi \in \Z_p^{n+1}$.
We have
\begin{align}
    |g(\xi_0)|_p = |f(\xi)|_p \leq p^{-e} < |g'(\xi_0)|_p^2.
\end{align}
By \cref{prop:newton-method-na}, there is $\tilde{\xi}_0 \in \Z_p$ such that
$g(\tilde{\xi}_0) = 0$ and 
\begin{align}
    |\tilde{\xi}_0 - \xi_0|_p \leq \frac{|g(\xi_0)|_p}{|g'(\xi_0)|_p}
    = \frac{|f(\xi)|_p}{|(\partial f / \partial x_0)(\xi)|_p}
    \leq \frac{p^{-e}}{p^{-l}} = p^{-(e-l)}.
\end{align}
Set $\tilde{\xi} = (\tilde{\xi}_0, \xi_1, \dots, \xi_n)$.
Since $\|\xi\|_p=1$ and $|\tilde{\xi}_0 - \xi_0|_p \leq p^{-(e-l)} < 1$,
we have $\|\tilde{\xi}\|_p=1$. 
We have $f(\tilde{\xi}) = g(\tilde{\xi}_0) = 0$.
Moreover
\begin{align}
    d_p(\tilde{\xi}, \xi) 
    = \max_{1 \leq i \leq n}\{ |\tilde{\xi}_0\xi_i - \xi_i \xi_0 |_p  \}
    \leq |\tilde{\xi}_0 - \xi_0|_p \leq p^{-(e-l)}.
\end{align}
and we are done.
\end{proof}

\section{Asymptotic formula for \texorpdfstring{$\# {\V}_{d,n}^{\rm loc}(A;\xi,\s)$}{Vdnloc(A)}}\label{sec:asymptotic-vdnloc}
%\section{Asymptotic formula for $\# {\V}_{d,n}^{\rm loc}(A)$}

We give an asymptotic formula of the number of hypersurfaces $V \subset \P^n_\Q$
of bounded height such that $\prod_{p \leq \infty}V(\Q_p)$
is non-empty and contains a point that is close to a given point of
$\prod_{p \leq \infty}\P^n(\Q_p)$ with respect to the product topology.

\begin{notation}
Let $n \geq 3$ and $d \geq 2$. 
Let $R$ be a ring.
For $\aa \in R^{N_{d,n}}$, $f_{\aa} \in R[x_{0},\dots, x_{n}]$ denotes the corresponding homogeneous polynomial of degree $d$. 
When $R$ is a field and $\aa \in R^{N_{d,n}} \setminus \{0\}$,
$\aa$ can be considered as a homogeneous coordinates of a point $a \in \P^{N_{d,n}-1}(R)$.
In this case, $V_{+}(f_{\aa}) \subset \P^{n}_{R}$ is determined by $a$ and 
we write $V_{+}(f_{a})$ instead of $V_{+}(f_{\aa})$.
\end{notation}

{\bf Setup}
\begin{itemize}

\item 
Let $\s = (\s_{p})_{p\in M_\Q} \in \R^{M_\Q}$ such that
$0<\s_{p} \leq 1$ and $\s_p = 1$ for all but finitely many $p$.
For $p \neq \infty$, let
\[
e_{p}\coloneqq \min\{ e \in \Z_{\geq 0} \mid p^{-e} \leq \s_{p}\}.
\]
We have $p^{-e_{p}} \leq \s_{p} < p^{-(e_{p}-1)}$.
Also $\s_{p} < 1$ if and only if $e_{p}\geq 1$.
We set 
\[
q\coloneqq \prod_{p < \infty}p^{e_{p}}.
\]

\item Let 
\begin{align}
    S \coloneqq \{p \in M_\Q \mid \s_p <1  \} \cup \{ \infty\}.
\end{align}
We write $S_{\rm fin} \coloneqq S \setminus \{\infty\}$.

\item $\xi = (\xi_{p})_{p \in S}$, where $\xi_{p} \in \P^{n}(\Q_{p})$.
For $p \in S_{\rm fin}$, we pick $\xii_{p} \in \Z_{p}^{n+1}$
such that $\left\|\xii_{p}\right\|_{p} = 1$ and $\left[\xii_{p}\right] = \xi_{p}$. 
We also pick  $\xii_{\infty} \in \R^{n+1}$ such that
$\left\|\xii_{\infty}\right\| = 1$ and $\left[\xii_{\infty}\right] = \xi_{\infty}$.
\end{itemize}

\begin{definition}\ 
We use the following notation:
\begin{enumerate}
\item
For a set $W \subset \R^N$, we write
\begin{align}
    \R W \coloneqq \{rw \mid r \in \R, w \in W  \}.
\end{align}
It is the two-sided cone over $W$.

\item
\begin{align}
%&\Z^{N}_{\rm prim} = \{\aa \in \Z^{N} \mid \gcd(\aa) = 1\}\\
(\Z^{N}_{p})_{\rm prim} &\coloneqq \{ \aa \in \Z_{p}^{N} \mid \|\aa\|_{p}=1\}\\
(\Z/p^{e}\Z)^{N}_{\rm prim} &\coloneqq \{\aa \in (\Z/p^{e}\Z)^{N} \mid \aa \not\equiv 0 \ \mod p \}\\
S^{N-1} &\coloneqq \{\aa \in \R^{N} \mid \|\aa\|=1\}.
\end{align}
\item
\begin{align}
&\mu_{p}\coloneqq\mu_{p}^{N} \quad \text{the Haar measure on $\Z_{p}^{N}$ such that $\mu_{p}(\Z_{p}^{N})=1$}\\
%&\mu_{\rm can}= \mu_{\P^{N}_{\Q_{p}}, {\rm can}} \quad \text{the canonical measure on $\P^{N}(\Q_{p})$}\\
%&\mu_{\P^{N}_{\Q_{p}}} = \frac{1}{\mu_{\P^{N}_{\Q_{p}}, {\rm can}} (\P^{N}(\Q_{p}))}\mu_{\P^{N}_{\Q_{p}}, {\rm can}} \\
&\mu_{S^{N-1}}  \quad \text{the spherical measure on $S^{N-1}$}.
\end{align}
\end{enumerate}
\end{definition}

We sometimes use the following normalized measure on $\P^{N-1}(\Q_p)$:
\begin{align}\label{lem:comparison-of-measures}
    \mu_{\P^{N-1}_{\Q_p}} 
    &\coloneqq \frac{1}{\mu_p^N((\Z_p^N)_{\prim})}\left( (\Z_{p}^{N})_{\prim} \to \P^{N-1}(\Q_{p})\right)_{*}\left( \mu_{p}^{N}|_{( \Z_{p}^{N})_{\prim}}\right)\\
    &= \frac{1}{1-p^{-N}}\left( (\Z_{p}^{N})_{\prim} \to \P^{N-1}(\Q_{p})\right)_{*}\left( \mu_{p}^{N}|_{( \Z_{p}^{N})_{\prim}}\right)
\end{align}

\begin{remark}
    The measure $(\#\P^{N-1}(\F_p)/p^{N-1}) \mu_{\P^{N-1}_{\Q_p}}$
    is called the canonical measure.
\end{remark}

% \if0
% Note that we have
% \begin{lemma}
% Let $N \geq 2$.
% Then
% \begin{align}
% \left( (\Z_{p}^{N})_{\prim} \to \P^{N-1}(\Q_{p})\right)_{*}\left( \mu_{p}^{N}|_{( \Z_{p}^{N})_{\prim}}\right) = (1-p^{-1}) \mu_{\P^{N-1}_{\Q_{p}}, {\rm can}}.
% \end{align}
% \end{lemma}
% \begin{proof}
%     \Yohsuke{write}
% \end{proof}
% \fi

\begin{definition}\ 
\begin{enumerate}
\item
For $p \in M_{\Q}$, we set
\[
Z_{p} \coloneqq \{ a \in \P^{N_{d,n}-1}(\Q_{p}) \mid V_{+}(f_{a})(\Q_{p}) \neq  \emptyset\}.
\]
\item
For $p \in M_{\Q}$, we set
\begin{align}
Z_{p}(\xi,\s) \coloneqq
\left\{ a \in \P^{N_{d,n}-1}(\Q_{p}) \ \middle| \  \txt{$\exists \eta \in V_{+}(f_{a})(\Q_{p})$ such that\\ $d_{p}(\eta, \xi_{p}) \leq \s_{p}$}\right\}
\end{align}
Note that we obviously have $Z_{p}(\xi,\s) \subset Z_{p}$ and
\begin{align}
    Z_p(\xi,\s) = Z_p
\end{align}
for $p \notin S$.

\item
For $p \in M_{\Q} \setminus \{\infty\}$, we set
\begin{align}
W_{p} &\coloneqq \left((\Z^{N_{d,n}}_{p})_{\rm prim} \to \P^{N_{d,n}-1}(\Q_{p})\right)^{-1}(Z_{p})\\
W_{p}(\xi,\s) &\coloneqq \left((\Z^{N_{d,n}}_{p})_{\rm prim} \to \P^{N_{d,n}-1}(\Q_{p})\right)^{-1}(Z_{p}(\xi,\s)).
\end{align}
Also, we set
\begin{align}
W_{\infty} &\coloneqq \left( S^{N_{d,n}-1} \to \P^{N_{d,n}-1}(\R)\right)^{-1}(Z_{\infty})\\
W_{\infty}(\xi,\s) &\coloneqq \left( S^{N_{d,n}-1}\to \P^{N_{d,n}-1}(\R)\right)^{-1}(Z_{\infty}(\xi,\s)).
\end{align}

\item
For $P \in \R_{\geq 1}$, we set
\begin{align}
M_{\Q, \leq P}\coloneqq \left\{ p \in M_{\Q} \ \middle|\ \text{$p$ is a prime number and $p\leq P$}\right\}.
\end{align}
\end{enumerate}
\end{definition}

\begin{definition}\label{def:rhops}
    We write
    \begin{align}
        \rho_{p}(\xi,\s) &\coloneqq \mu_{\P^{N_{d,n}-1}_{\Q_{p}}}(Z_{p}(\xi,\s)) 
        \quad \text{for $p \in M_{\Q} \setminus \{ \infty\}$};\\
        \rho_{\infty}(\xi,\s) &\coloneqq \mu_{S^{N_{d,n}-1}}(W_{\infty}(\xi,\s)).
    \end{align}
\end{definition}

\begin{definition}
Let $A \geq 1$.
We define 
\begin{align}
\V_{d,n}(A) &\coloneqq \{V_{+}(f_{\aa}) \mid \aa \in \Z^{N_{d,n}}_{\rm prim}, \|\aa\| \leq A\},\\
\V_{d,n}^{\rm loc}(A;\xi,\s) &\coloneqq \left\{ V \in \V_{d,n}(A)\ \middle|\ 
\txt{$\forall p \in S, \exists \eta \in V(\Q_{p})$ such that $d_{p}(\eta, \xi_{p}) \leq \s_{p}$\\
$\forall p \in M_{\Q} \setminus S, V(\Q_{p}) \neq  \emptyset$}\right\}.
\end{align}
\end{definition}

Using our notation, we can rewrite
\begin{align}
\# \V_{d,n}^{\loc}(A; \xi, \s)= \frac{1}{2}\# \left\{ \aa \in \Z_{ \prim}^{ N_{d,n}}  \ \middle|\  \| \aa\| \leq A, \forall p \in M_{\Q}, [\aa] \in Z_{p}(\xi,\s)\right\}.
\end{align}

For $P \in \R_{\geq 1}$, we have
\begin{align}
\#\V_{d,n}^{\loc}(A;\xi,\s) 
= &\frac{1}{2} \# \left\{ \aa \in \Z_{ \prim}^{ N_{d,n}} \ \middle|\ \| \aa \|\leq A, \forall p \in S \cup M_{\Q, \leq P}, [\aa] \in Z_{p}(\xi,\s)\right\}\\
&- \frac{1}{2} \# \left\{ \aa \in \Z_{ \prim}^{ N_{d,n}} \ \middle|\ \txt{$ \| \aa \|\leq A, \forall p \in S \cup M_{\Q, \leq P}, [\aa] \in Z_{p}(\xi,\s)$\\
$\exists p > P$ prime number, $[\aa] \notin Z_{p}$} \right\}.
\end{align}

For $A,P \in \R_{\geq 1}$, we define
\begin{align}
&M(A,P):= \# \left\{ \aa \in \Z_{ \prim}^{ N_{d,n}} \ \middle|\ \| \aa \|\leq A, \forall p \in S \cup M_{\Q, \leq P}, [\aa] \in Z_{p}(\xi,\s)\right\}\\
&E(A,P):=\# \left\{ \aa \in \Z_{ \prim}^{ N_{d,n}} \ \middle|\ \txt{$ \| \aa \|\leq A, \forall p \in S \cup M_{\Q, \leq P}, [\aa] \in Z_{p}(\xi,\s)$\\
$\exists p > P$ prime number, $[\aa] \notin Z_{p}$} \right\}.
\end{align}

Then we have
\begin{align}\label{eq:Vdn=M-E}
\#\V_{d,n}^{\loc}(A;\xi,\s)  = \frac{1}{2}M(A,P) - \frac{1}{2}E(A,P).
\end{align}

\if0
{\bf Overview}
 \[\scriptsize
\xymatrix{
\txt{Lang-Weil range\\Reducible locus $X$} \ar@/_{20pt}/@<-20pt>[ddd] \ar[rd]& \txt{ Easy bound of $Z_{p}(\xi,\s)$ by\\ congruence condition and \\$d_{\infty}$-inequality} \ar@/_{10pt}/[lddd]&\\
\text{Lang-Weil estimate} \ar@<-5pt>[dd]& \txt{Boundary ball estimate \\ Volume estimate of $Z_{p}(\xi,\s)$ } \ar[dd]& \text{Newton method} \ar[l]\\
&&\txt{Lattice point \\counting} \ar@/_{10pt}/[lld] \ar[ld]\\
\text{Bounding $E(A,P)$} \ar[rd] &\txt{Asymptotic formula \\of $M(A,P)$} \ar[d]&\\
&\text{Optimize $P$}&
}
\]
\fi

We summarize our strategy to get an asymptotic formula of 
$\#\V_{d,n}^{\loc}(A;\xi,\s) $, i.e. the strategy of the proof of
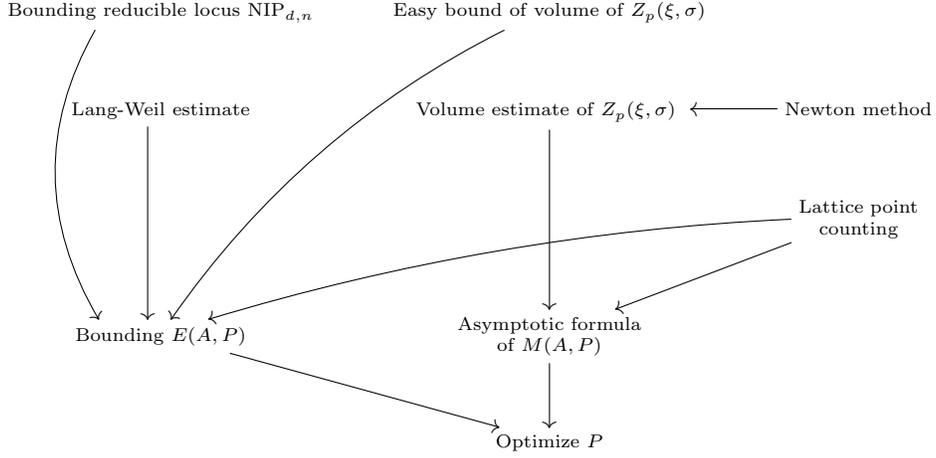
\cref{mainthm:denominator}.

\[\scriptsize
\xymatrix{
\text{Bounding reducible locus $\NIP_{d,n}$} \ar@/_{20pt}/@<-20pt>[ddd]& \text{Easy bound of volume of $Z_{p}(\xi,\s)$} \ar@/_{14pt}/[lddd]&\\
\text{Lang-Weil estimate} \ar@<-5pt>[dd]& \text{Volume estimate of $Z_{p}(\xi,\s)$ } \ar[dd]& \text{Newton method} \ar[l]\\
&&\parbox{15ex}{\centering Lattice point \\counting} \ar@/_{10pt}/[lld] \ar[ld]\\
\text{Bounding $E(A,P)$} \ar[rd] &\parbox{25ex}{\centering Asymptotic formula\\of $M(A,P)$}\ar[d]&\\
&\text{Optimize $P$}&
}
\]

\subsection{On the sets \texorpdfstring{$Z_{p}, Z_{p}( \xi, \s)$}{ZpZpsx}}
\label{subsec:desc-Zp}
\leavevmode
\medskip

{\bf A description of $Z_{p}, Z_{p}( \xi, \s)$.}

For $p \in M_{\Q}$,
consider the following diagram
 \[
\xymatrix{
& \mathcal{Z}_{p} \ar@{}[d]|{\bigcap} &\\
&\P^{n}_{\Q_{p}} \times_{\Q_{p}} \P^{ N_{d, n}-1}_{\Q_{p}}  \ar[ld]_{\pr_{1}} \ar[rd]^{\pr_{2}}&\\
\P^{n}_{\Q_{p}}& & \P^{ N_{d,n}-1}_{\Q_{p}}
}
\]
where $ \mathcal{Z}_{p}$ is the closed subscheme defined by
$f_{\aa}(\xx) =0$ where $\xx, \aa$ are homogeneous coordinates of $\P^{n}$ and $\P^{ N_{d,n}-1}$ respectively.
Then we have
\begin{align}
Z_{p} &= \pr_{2}( \mathcal{Z}_{p}(\Q_{p}))\\
Z_{p}(\xi, \s) & = \pr_{2}( \mathcal{Z}_{p}(\Q_{p}) \cap \pr_{1}^{-1}(B_{d_{p}}(\xi_{p},\s_{p})))
\end{align}
where $B_{d_{p}}(\xi_{p}, \s_{p}) = \left\{ \eta \in \P^{n}(\Q_{p}) \ \middle|\ d_{p}(\eta, \xi_{p}) \leq \s_{p} \right\}$.
In particular, $Z_{p}, Z_{p}(\xi, \s)$ are closed subsets of $\P^{ N_{d,n}-1}(\Q_{p})$ (with respect to the strong topology).

\begin{remark}
By this description, we can say that the boundary of $Z_{p}$ and $Z_{p}(\xi, \s)$ are contained in 
a proper Zariski closed subsets of $\P^{ N_{d,n}-1}(\Q_{p})$.
Indeed, the morphism $\pr_{2}|_{ \mathcal{Z}_{p}} \colon \mathcal{Z}_{p} \longrightarrow \P^{ N_{d,n}-1}_{\Q_{p}}$
is generically smooth because it is generically flat and general geometric fibers are smooth.
Pick a dense Zariski open subset $U \subset \P^{ N_{d,n}-1}_{\Q_{p}}$ such that 
$\pr_{2}|_{ \mathcal{Z}_{p}}$ is smooth over $U$.
Then the boundaries are contained in $(\P^{ N_{d,n}-1}_{\Q_{p}} \setminus U)(\Q_{p})$.
From this fact, we see that the boundaries have measure zero, but we need more
precise information about the boundaries (cf. \cref{boundary-ball-estimate}).
\end{remark}

\begin{lemma}\label{W-is-sa}
The sets $\R W_{\infty},\R W_{\infty}(\xi,\s) \subset \R^{ N_{d,n}}$ are semialgebraic.
More generally, consider the sets
\begin{align}
\mathcal{B}_{ N_{d,n}}(A) \cap \R W_{\infty}(\xi, \s) -\aa_{0} 
\end{align}
where $A \in \R$ and $\aa_0 \in \R^{N_{d,n}}$.
Then they form a semialgebraic family of subsets of $\R^{N_{d,n}}$
parametrized by 
\begin{align}
   (A,\s_{\infty}, \xii_{\infty}, \aa_0) \in \R \times \R \times \R^{n+1} \times \R^{N_{d,n}}. 
\end{align}
\end{lemma}

\begin{proof}
Define 
$ \widetilde{K} \subset \R^{ N_{d,n}} \times \R \times \R \times \R^{n+1} \times \R^{ N_{d,n}} \times \R^{n+1}$
as follows:
\begin{align}
(\aa, A, \s_{\infty}, \xii_{\infty}, \aa_{0},\xx)  \in \R^{ N_{d,n}} \times \R \times \R \times \R^{n+1} \times \R^{ N_{d,n}} \times \R^{n+1}
\end{align}
is a member of $ \widetilde{K}$ if and only if
\begin{align}
\aa + \aa_{0} \in \mathcal{B}_{ N_{d,n}}(A),\quad
\xx \neq 0,\quad
d_{\infty}([\xx],[\xii_\infty]) \leq \s_{\infty},\quad
f_{\aa + \aa_{0}}(\xx) = 0.
\end{align}
These are semialgebraic conditions and thus $\widetilde{K}$ is a semialgebraic set.
Then the projection $K$ of $ \widetilde{ K}$ to the first five factors is also semialgebraic and
the fiber of $K$ over the point $(A, \s_{\infty}, \xii_{\infty}, \aa_{0})$ is $\mathcal{B}_{ N_{d,n}}(A) \cap \R W_{\infty}(\xi, \s) -\aa_{0} $.

If we remove $A$ and the condition involving $A$, and set $\aa_0=0$,
we get the semialgebraicity of $\R W$ and $\R W_{\infty}(\xi, \s)$.

%%%%%%%%%%%%%%%%%%%%%%%%%%%%%%%%
\if0

Since $W_{\infty} = W_{\infty}(\xi,\s)$ when $\s_{\infty} = 1$, it is enough to show that $\R W_{\infty}(\xi,\s)$ is semi-algebraic.
Consider the following diagram:
 \[
\xymatrix{
& \mathcal{W} \ar@{}[d]|{\bigcap} &\\
&\R^{n+1} \times \R^{ N_{d,n}}  \ar[ld]_{\pr_{1}} \ar[rd]^{\pr_{2}}&\\
\R^{n+1}& & \R^{ N_{d,n}} 
}
\]
where $ \mathcal{W} = \{(\xx,\aa) \mid f_{\aa}(\xx) = 0, \xx \neq 0\}$.
Note that this is a semi-algebraic set.
We have
\begin{align}
\R W_{\infty}(\xi,\s) = \pr_{2} \left(  \mathcal{W} \cap \pr_{1}^{-1}(\{ \xx \in \R^{n+1} \setminus \{0\} \mid d_{\infty}([\xx],\xi_{\infty})\leq \s_{\infty} \}) \right).
\end{align}
The right hand side is semi-algebraic and we are done.

\fi
%%%%%%%%%%%%%%%%%%%%%%%%%%%%%%%%

\end{proof}

Let us define $\nu_{d,n}^{(i)}$ by
    \begin{align}
        \nu_{d,n}^{(i)}(\xx) :=  \left(  \frac{\partial M}{\partial x_i} (\xx) \right)_{M \in \mathcal{M}_{d,n}} \quad i=0,\dots,n
    \end{align}
so that 
    \begin{align}
        \frac{\partial f_\aa}{\partial x_i}(\xx) = \langle \aa, \nu_{d,n}^{(i)}(\xx) \rangle.
    \end{align}

\begin{proposition}\label{easy-bound-Z}
We have the following\textup{:}
\begin{enumerate}
\item
\begin{align}
\rho_{\infty}(\xi,\s)=\mu_{S^{ N_{d,n}-1}}(W_{\infty}(\xi,\s)) \asymp  \s_{\infty}
\end{align}
where the implicit constant depends only on $d,n$\textup{;}
\item
For all $p \in S_{\fin}$ and $\aa \in \Z_{p}^{ N_{d,n}} \setminus \{0\}$, if $[\aa] \in Z_{p}( \xi,\s)$, then
\begin{align}
| \langle \aa, \nu_{d,n}(  \xii_{p})\rangle|_{p} \leq \s_{p} \| \aa \|_{p}.
\end{align}
\end{enumerate}
\end{proposition}
\begin{proof}
    (1)
    First we prove $\rho_{\infty}(\xi, \s) \ll_{d,n} \s_\infty$.
    If $\aa \in W_{\infty}(\xi, \s)$, then there is $\etaa \in S^n$ such that
    $\langle \aa, \nu_{d,n}(\etaa) \rangle =0$ and $d_{\infty}([\etaa], [\xii_\infty]) \leq \s_\infty$.
    By replacing $\etaa$ with $-\etaa$, we may assume $\langle \etaa , \xii_\infty \rangle \geq 0$.
    Then we have
    \begin{align}
        |\langle \aa, \nu_{d,n}(\xii_\infty) \rangle|
        &= |\langle \aa, \nu_{d,n}(\xii_\infty) - \nu_{d,n}(\etaa) \rangle|
        \leq \|  \nu_{d,n}(\xii_\infty) - \nu_{d,n}(\etaa)  \|\\
        &\ll_{d,n} \| \xii_\infty - \etaa \|  \ll d_\infty([\xii_\infty], [\etaa]) \leq \s_\infty.
    \end{align}
    From this, we can easily deduce $\rho_\infty(\xi, \s) \ll_{d,n} \s_\infty$.

    Next we prove $\rho_{\infty}(\xi, \s) \gg_{d,n} \s_\infty$.
    Let us consider an orthogonal transformation of $\R^{n+1}$
    that sends $\xii_\infty$ to $(1,0,\dots, 0)$.
    This induces a linear automorphism $T$ on the coefficient space $\R^{N_{d,n}}$.
    Note that the operator norms of $T$ and $T^{-1}$ are bounded independently of $\xii_\infty$: 
    $\|T\| \ll_{d,n} 1$, $\|T^{-1}\| \ll_{d,n} 1$.
    Similarly, we have $|\det T|\ll_{d,n} 1$, $|\det T^{-1}| \ll_{d,n} 1$.
    Thus the density $\rho_\infty(\xi, \s)$ is independent of
    $\xii_\infty$ up to multiple by positive constants depending only on $d,n$.
    Therefore, we may assume $\xii_\infty = (1,0,\dots, 0)$.

    If $\xii_\infty = (1,0,\dots, 0)$, then $\nu_{d,n}(\xii_\infty)$ 
    and $\nu_{d,n}^{(1)}(\xii_\infty)$ are linearly independent.
    Thus there is $\aa_0 \in S^{N_{d,n}-1}$ such that 
    \begin{align}
        \langle \aa_0, \nu_{d,n}(\xii_\infty) \rangle = 0,\ 
        \langle \aa_0, \nu_{d,n}^{(1)}(\xii_\infty) \rangle =: \alpha_0 > 0.
    \end{align}
    Let us consider the set
    \begin{align}
        \Sigma(\tau) \coloneqq \left\{ \aa \in S^{N_{d,n}-1} \ \middle|\ 
        |\langle \aa, \nu_{d,n}(\xii_\infty) \rangle| \leq \tau,\ 
        \langle \aa, \nu_{d,n}^{(1)}(\xii_\infty) \rangle \geq \alpha_0/2
        \right\}
    \end{align}
    parametrized by $\tau>0$.
    We are applying Newton method to show that for $\aa \in \Sigma(\tau)$,
    the polynomial $f_\aa$ has a solution close to $\xii_\infty$.
    Indeed, let $f(t) = f_\aa(1,t,0,\dots,0)$.
    Then $f$ is a polynomial in $t$ with degree at most $d$ and
    \begin{align}
        &|f(0)| = |\langle \aa, \nu_{d,n}(\xii_\infty) \rangle| \leq \tau,\\
        &|f'(0)| = |\langle \aa, \nu_{d,n}^{(1)}(\xii_\infty) \rangle| \geq \frac{\alpha_0}{2},\\
        &|f| = \max\{\text{absolute value of coefficient of $f$}\} \leq 1.
    \end{align}
    By \cref{prop:newton-method-arch},
    there is $C > 0$ (depending only on $d$) with the following property:
    if $\tau \leq C \alpha_0^2/4$, then there is $y \in \R$ such that
    \begin{align}
        f(y) = 0
        \and
        |y| \leq \frac{4\tau}{\alpha_0}.
    \end{align}
    This means $f_\aa$ has a solution $(1,y,0,\dots,0)$, which satisfies
    \begin{align}
        d_\infty([(1,y,0,\dots,0)],\xii_\infty) = \frac{|y|}{\sqrt{1+y^2}} \leq \frac{4\tau}{\alpha_0}.
    \end{align}
    Therefore we in particular have
    \begin{align}
        \Sigma(\alpha_0 \s_\infty/4) \subset W_\infty(\xi,\s).
    \end{align}
    It is easy to see that
    \begin{align}
        \mu_{S^{N_{d,n}-1}}(\Sigma(\tau)) \gg_{n,d,\aa_0,\alpha_0} \tau
    \end{align}
    for $\tau \ll 1$.
    Since we choose $\aa_0, \alpha_0$ depending only on $d,n$, we are done.

    (2)
    Since $[\aa] \in Z_p(\xi, \s)$, there is $\etaa \in \Z_p^{n+1}$ with $\|\etaa\|_p=1$ such that
    $\langle \aa, \nu_{d,n}(\etaa) \rangle = 0$ and $d_p([\etaa], \xi_p) \leq \s_p < 1$.
    Note that since $d_p([\etaa], \xi_p) < 1$, we can take $\etaa$ so that at least one coordinate is unit and equal to the corresponding entry of $\xii_p$.
    Then we have
    \begin{align}
        |\langle \aa, \nu_{d,n}(\xii_p) \rangle|_p 
        &= |\langle \aa, \nu_{d,n}(\xii_p) - \nu_{d,n}(\etaa)\rangle|_p
        \leq \|\aa\|_p \|\nu_{d,n}(\xii_p) - \nu_{d,n}(\etaa) \|_p\\
        &\leq \|\aa\|_p \|\xii_p - \etaa \|_p = \|\aa\|_p d_p([\xii_p], [\etaa])
        \leq \s_p \|\aa\|_p.
    \end{align}
\end{proof}

\begin{definition}
Let $p$ be a prime number.
\begin{enumerate}
\item
Let $N \in \Z_{\geq 1}$.
For $\aa \in \Z_{p}^{N}$ and $r > 0$, we write
\begin{align}
B(\aa,r) \coloneqq \left\{ \xx \in \Z_{p}^{N} \ \middle|\ \| \xx - \aa \|_p \leq r\right\}.
\end{align}
Note that if $r = p^{-v}$ for some  $v \in \Z_{\geq 1}$, then
$B(\aa, p^{-v})$ is determined by the class $\aa\ \mod {p^{v}}$.
For $\aa' \in (\Z/p^{v}\Z)^{N}$, we write $B(\aa', p^{-v})$ the ball $B(\aa, p^{-v})$
where $\aa \in \Z_{p}^{N}$ is an arbitrary lift of $\aa'$.
\item
Let $v \in \Z_{\geq 1}$.
We set
\begin{align}
\partial W_{p}(p^{v})&\coloneqq 
\left\{ \aa \in (\Z/p^{v}\Z)^{ N_{d,n}} \ \middle|\ \txt{$B(\aa, p^{-v})\cap W_{p} \neq  \emptyset$\\ $B(\aa,p^{-v}) \not\subset W_{p}$}\right\},\\
\partial W_p(\xi,\s)(p^{v})&\coloneqq 
\left\{ \aa \in (\Z/p^{v}\Z)^{ N_{d,n}} \ \middle|\ \txt{$B(\aa, p^{-v})\cap W_{p}(\xi,\s) \neq  \emptyset$\\ $B(\aa,p^{-v}) \not\subset W_{p}(\xi,\s)$}\right\}.
\end{align}

\end{enumerate}
\end{definition}

\begin{proposition}\label{boundary-ball-estimate}
\ 
\begin{enumerate}
\item
Let $p$ be a prime number. For $v \in \Z_{\geq 1}$ with $v \geq e_p$, set
\begin{align}
\widetilde{v} \coloneqq \min \left\{\ceil*{\frac{v}{2}} , v-e_{p} +1\right\}.
\end{align}
Then we have
\begin{align}\label{ineq:boundary_ball_estimate_general}
\# \partial W_p(\xi, \s)(p^{v}) \leq \frac{p^{v N_{d,n}}}{p^{v- \widetilde{v}+(n+1)\min\{e_{p}, \widetilde{v}\}}}.
\end{align}
In particular, 
\begin{align}
\# \partial W_p(p^{v}) \leq \frac{p^{v N_{d,n}}}{p^{v- \ceil*{v/2}}}.
\end{align}

\item
For $p \in S_{\fin}$, we have
\begin{align}
\# \partial W_p(\xi, \s)(p^{e_{p}}) \leq  \frac{ p^{e_{p} N_{d,n}}}{p^{e_{p}+n}}.
\end{align}

\end{enumerate}
\end{proposition}

\begin{proof}
    The last assertion in (1) follows from \cref{ineq:boundary_ball_estimate_general} by setting $e_p=0$.
    Indeed, in this case we have $\widetilde{v} = \ceil*{\frac{v}{2}}$.
    (2) follows from (1) by setting $v=e_p$. Note that $p \in S_{\fin}$
    implies $e_p \geq 1$.
    Thus it is enough to prove the bound of  $\# \partial W_p(\xi, \s)(p^{v})$.
    From now on, we fix $v \in \Z_{\geq 1}$ with $v \geq e_p$.

    \begin{claim}\label{claim:boundary-ball}
        If $\aa \in W_p(\xi,\s)$ and $B(\aa, p^{-v}) \not\subset W_p(\xi,\s)$,
        then there is $\xx \in (\Z_p^{n+1})_{\prim}$ such that
        \begin{align}
            f_{\aa}(\xx)=0,\quad
            d_p\left([\xx],[\xii_p]\right) \leq p^{-e_p}
            \and
            \left| \frac{\partial f_{\aa}}{\partial x_i}(\xx)  \right|_p \leq p^{- \widetilde{v}} \quad \text{for all $0 \leq i \leq n$}.
        \end{align}
    \end{claim}
    \begin{proof}[Proof of \cref{claim:boundary-ball}]
        Since $\aa \in W_p(\xi, \s)$, there is $\xx \in (\Z_p^{n+1})_{\prim}$
        such that 
        \begin{align}
            f_{\aa}(\xx)=0
            \and
            d_p([\xx],[\xii_p]) \leq p^{-e_p}.
        \end{align}
        Take a point $\bb \in B(\aa,p^{-v}) \setminus W_p(\xi,\s)$.
        Then we have
        \begin{align}
            |f_{\bb}(\xx)|_p = |f_\bb(\xx) - f_\aa(\xx)|_p = |f_{\bb-\aa}(\xx)|_p \leq p^{-v}.
        \end{align}
        We first show
        \begin{align}\label{ineq:derivative-at-b}
            \left| \frac{\partial f_{\bb}}{\partial x_i}(\xx)  \right|_p \leq p^{- \widetilde{v}}
            \quad \text{for all $0 \leq i \leq n$}.
        \end{align}
        Suppose there is $i$ such that 
        \begin{align}
            \left| \frac{\partial f_{\bb}}{\partial x_i}(\xx)  \right|_p > p^{- \widetilde{v}}.
        \end{align}
        Since $2(\widetilde{v}-1)\le2(\lceil\frac{v}{2}\rceil-1)<v$, we have
        \begin{align}
            \left| \frac{\partial f_{\bb}}{\partial x_i}(\xx)  \right|_p^2 > p^{-v} \geq |f_\bb(\xx)|.
        \end{align}
        By \cref{lem:applyingNtn} with $e = v, l = \tilde{v} - 1$, we get $\yy \in (\Z_p^{n+1})_{\prim}$
        such that
        \begin{align}
            f_\bb(\yy)=0
            \and
            d_p([\yy],[\xx]) \leq p^{-(v - (\widetilde{v}-1))}.
        \end{align}
        We have
        \begin{align}
            d_p\left([\yy],[\xii_p]\right) 
            \leq \max\left\{ d_p\left([\yy],[\xx]\right), d_p\left([\xx],[\xii_p]\right) \right\}
            \leq \max\{ p^{-(v- \widetilde{v} + 1)} , p^{-e_p} \} \leq p^{-e_p}.
        \end{align}
        Here for the last inequality, we use
        \begin{align}
            v - \widetilde{v} + 1 \geq v - (v-e_p + 1) + 1 = e_p.
        \end{align}
        Existence of such $\yy$ contradicts to the assumption $\bb \notin W_p(\xi, \s)$.
        Thus we have proven \cref{ineq:derivative-at-b}.

        Finally, for all $i$, we get
        \begin{align}
            \left| \frac{\partial f_{\aa}}{\partial x_i}(\xx)  \right|_p
            &= \left| \frac{\partial f_{\aa}}{\partial x_i}(\xx)
            -\frac{\partial f_{\bb}}{\partial x_i}(\xx) +\frac{\partial f_{\bb}}{\partial x_i}(\xx) \right|_p\\
            &\leq \max\left\{\left| \frac{\partial f_{\aa}}{\partial x_i}(\xx)
            -\frac{\partial f_{\bb}}{\partial x_i}(\xx) \right|_p,
            \left| \frac{\partial f_{\bb}}{\partial x_i}(\xx) \right|_p\right\}\\
            &= \max\left\{\left| \frac{\partial f_{\aa-\bb}}{\partial x_i}(\xx) \right|_p,
            \left| \frac{\partial f_{\bb}}{\partial x_i}(\xx) \right|_p\right\}\\
            & \leq \max\{ p^{-v}, p^{- \widetilde{v}} \}\leq p^{-\widetilde{v}}.
        \end{align}
    \end{proof}
    In light of this claim, we consider the following set
    \begin{align}
        U \coloneqq \left\{ \aa \in (\Z/p^v \Z)^{N_{d,n}}_{\prim} \ \middle|\  
        \txt{$\exists \xx \in (\Z/p^v \Z)^{n+1}_{\prim }$ s.t.\\
        $\xx \equiv \xii_p\ \mod {p^{e_p}}$\\
        $\langle \aa, \nu_{d,n}(\xx) \rangle \equiv 0\ \mod {p^v}$\\
        $\langle \aa, \nu_{d,n}^{(i)}(\xx)\rangle \equiv 0\ \mod {p^{\widetilde{v}}}$ for all $i$}  \right\}.
    \end{align}
    By \cref{claim:boundary-ball}, we have
    \begin{align}
        \partial W_p(\xi,\s)(p^{v}) \subset U.
    \end{align}
    Thus it is enough to bound $\# U$.
    We do this in several steps.
    
    {\bf Step 1}.
    For $\xx \in (\Z/p^v \Z)^{n+1}_{\prim}$, let 
    \begin{align}
        K(\xx) \coloneqq \left\{ \aa \in (\Z/p^v \Z)^{N_{d,n}}_{\prim} \ \middle|\  
        \txt{$\langle \aa, \nu_{d,n}(\xx) \rangle \equiv 0\ \mod {p^v}$\\
        $\langle \aa, \nu_{d,n}^{(i)}(\xx)\rangle \equiv 0\ \mod {p^{\widetilde{v}}}$ for all $i$}  \right\}.
    \end{align}
    Consider two $\xx, \xx' \in (\Z/p^v \Z)^{n+1}_{\prim}$.
    We want to know when we have $K(\xx) = K(\xx')$.
    Suppose $\xx \equiv \xx'\ \mod {p^g}$ for some $g$ with $1 \leq g \leq v$.
    In this case, we can write $\xx = \xx' + p^g \yy$ for some 
    $\yy = (y_0, \dots, y_n) \in (\Z/p^v \Z)^{n+1}$.
    Let $\aa \in K(\xx')$.
    Then 
    \begin{align}
        \langle \aa, \nu_{d,n}(\xx) \rangle \equiv 
        \langle \aa, \nu_{d,n}(\xx') \rangle 
        + \sum_{i=0}^n \langle \aa, \nu_{d,n}^{(i)}(\xx') \rangle p^g y_i
        \ \mod {p^{2g}}
    \end{align}
    Thus if $2g \geq v$ and $\widetilde{v} + g \geq v$, then 
    $\langle \aa, \nu_{d,n}(\xx) \rangle \equiv 0 \ \mod {p^v}$.
    If $2g \geq v$, then $g\ge\lceil v/2\rceil\ge\widetilde{v}$.
    Thus $\xx \equiv \xx' \ \mod {p^{\widetilde{v}}}$ and hence
    \begin{align}
        \langle \aa, \nu_{d,n}^{(i)}(\xx)\rangle 
        \equiv \langle \aa, \nu_{d,n}^{(i)}(\xx')\rangle  \equiv 0 \ \mod {p^{\widetilde{v}}} 
    \end{align}
    for all $i$.
    Therefore if we set
    \begin{align}
        g_0 \coloneqq \max\left\{ \ceil*{\frac{v}{2}}, v - \widetilde{v} \right\},
    \end{align}
    then for $\xx, \xx' \in (\Z/p^v \Z)^{n+1}_{\prim}$, we have
    \begin{align}
        \xx \equiv \xx' \ \mod {p^{g_0}} \Longrightarrow K(\xx) = K(\xx').
    \end{align}

    {\bf Step 2}.
    For a given $\xx \in (\Z/p^v \Z)^{n+1}_{\prim}$, we bound 
    $\# K(\xx)$.
    Consider the following diagram
    \[
    \xymatrix{
    \Ker \varphi \ar[r] \ar[rd] & (\Z/p^v \Z)^{N_{d,n}} \ar[r]^{\varphi} \ar[d]_\pi & \Z/p^v \Z\\
     & (\Z/p^{\widetilde{v}}\Z)^{N_{d,n}} \ar[r]_J & (\Z/p^{\widetilde{v}}\Z)^{n+2} 
    }
    \]
    where $\pi$ is the canonical projection,
    $\varphi(\aa) = \langle \aa, \nu_{d,n}(\xx) \rangle$, and
    $J$ is the map defined by multiplying the following matrix from left:
    \begin{align}\label{eq:matrix-of-J}
    \renewcommand*{\arraystretch}{1.5}
        \begin{pmatrix}
            \nu_{d,n}(\xx)\\
            \nu_{d,n}^{(0)}(\xx)\\
            \vdots\\
            \nu_{d,n}^{(n)}(\xx)
        \end{pmatrix}.
    \end{align}
    Here we consider $\nu_{d,n}(\xx)$ and $\nu_{d,n}^{(i)}(\xx)$
    as row vectors with $N_{d,n}$ entries.
    Then we have $K(\xx) = \Ker \varphi \cap \pi^{-1}(\Ker J)$.
    Since $\xx \in (\Z/p^v \Z)^{n+1}_{\prim}$, 
    $\nu_{d,n}(\xx) \in (\Z/p^v \Z)^{N_{d,n}}_{\prim}$ as well.
    Thus $\varphi$ is surjective and hence $\Ker \varphi$ is a direct summand 
    of $(\Z/p^v \Z)^{N_{d,n}}$.
    Thus we have
    \begin{align}\label{eq:kerphi}
        \Ker \varphi \cap \Ker \pi \simeq (p^{\widetilde{v}}\Z/p^v\Z)^{N_{d,n}-1}
        \simeq (\Z/p^{v - \widetilde{v}}\Z)^{N_{d,n}-1}.
    \end{align}
    Next we claim that the $n+1$-minors of the matrix \cref{eq:matrix-of-J} 
    generate the trivial ideal $(1) = \Z/p^{\widetilde{v}} \Z$.
    Indeed, we may assume $x_0 \not\equiv 0 \ \mod p$ without loss of generality.
    The submatrix formed by columns and rows corresponding to the monomials 
    $x_0^d, x_0^{d-1}x_1,\dots, x_0^{d-1}x_n$
    and $\nu_{d,n}, \nu_{d,n}^{(1)}, \dots, \nu_{d,n}^{(n)}$ respectively
    is
    \begin{align}
        \begin{pmatrix}
            x_0^d & x_0^{d-1}x_1 & \cdots &\cdots & x_0^{d-1}x_n\\
            0 & x_0^{d-1} & 0 & \cdots & 0 \\
            \vdots & \ddots & \ddots &&  \vdots\\
            \vdots & & \ddots & \ddots & 0 \\
            0 & \cdots & \cdots & 0 & x_0^{d-1}
        \end{pmatrix}.
    \end{align}
    The determinant of this matrix is $x_0^{d + (d-1)n}$ and this is a unit.
    By, for example, Nakayama's lemma, $\Cok J$ is generated by 
    at most one element over $\Z/p^{\widetilde{v}}\Z$.
    Therefore we get
    \begin{align}\label{ineq:bound-of-kerj}
        \# \Ker J = \frac{p^{\widetilde{v}N_{d,n}}}{\# \Img J} 
        = \frac{p^{\widetilde{v}N_{d,n}}}{p^{\widetilde{v}(n+2)}} \# \Cok J
        \leq p^{\widetilde{v}(N_{d,n} -n -1)}.
    \end{align}
    Combining \cref{eq:kerphi} and \cref{ineq:bound-of-kerj}, we get
    \begin{align}
        \#K(\xx) &\leq (\#\Ker J) \# (\Ker \varphi \cap \Ker \pi)\\
        &\leq p^{\widetilde{v}(N_{d,n} -n -1) + (v - \widetilde{v})(N_{d,n}-1)}
        = p^{v(N_{d,n}-1) - \widetilde{v}n}.
    \end{align}
    
    {\bf Step 3}.
    By Step 1 and Step 2, we get
    \begin{align}
        \# U = \# 
        \bigcup_{\substack{\xx \in (\Z/p^v\Z)^{n+1}_{\prim}\\ \xx \equiv \xii_p \ \mod {p^{e_p}}}}
        K(\xx)
        \leq p^{v(N_{d,n}-1) - \widetilde{v}n} \# \mathcal{X}
    \end{align}
    where 
    \begin{align}
        \mathcal{X} \coloneqq 
        \Img \left( \{ \xx \in (\Z/p^v\Z)^{n+1}_{\prim} 
        \mid \xx \equiv \xii_p \ \mod {p^{e_p}} \}  \longrightarrow
        (\Z/p^{g_0}\Z)^{n+1} \right)
    \end{align}
    (Recall that we assumed $v \geq e_p$).
    Since
    \begin{align}
        \# \mathcal{X} = p^{\max\{(g_0 - e_p),0\} (n+1)},
    \end{align}
    we get
    \begin{align}
        \#U \leq p^{v(N_{d,n}-1) - \widetilde{v}n + \max\{(g_0 - e_p),0\} (n+1)}.
    \end{align}
    Finally we calculate the exponent.
    First we have
    \begin{align}
        &v(N_{d,n}-1) - \widetilde{v}n + \max\{(g_0 - e_p),0\} (n+1)\\
        &= v(N_{d,n}-1) - \widetilde{v}n 
        + \max\left\{\ceil*{\frac{v}{2}}-e_p, v- \widetilde{v}-e_p ,0\right\} (n+1)\\
        & = v(N_{d,n}-1) - \widetilde{v}n 
        + \max\left\{\ceil*{\frac{v}{2}}-e_p ,0\right\} (n+1)\\
        & = v(N_{d,n}-1) + \widetilde{v} 
        - (n+1)\left( \widetilde{v} - \max\left\{\ceil*{\frac{v}{2}}-e_p ,0\right\}\right)\\
        & = v(N_{d,n}-1) + \widetilde{v} 
        - (n+1) \min\left\{ \widetilde{v} - \ceil*{\frac{v}{2}} + e_p, \widetilde{v} \right\}.
    \end{align}
    Here for the second equality we use
    \begin{align}
        v- \widetilde{v} -e_p =
        v - \min\left\{ \ceil*{\frac{v}{2}} + e_p, v+1  \right\}
        \leq \max\left\{\ceil*{\frac{v}{2}} - e_p, 0  \right\}.
    \end{align}
    Observe that 
    \begin{align}
        &e_p \geq \widetilde{v} \Longrightarrow 
        \min\left\{ \widetilde{v} - \ceil*{\frac{v}{2}} + e_p, \widetilde{v} \right\}
        = \widetilde{v}\\
        &e_p < \widetilde{v} \Longrightarrow 
        e_p < \ceil*{\frac{v}{2}} \Longrightarrow 
        \min\left\{ \widetilde{v} - \ceil*{\frac{v}{2}} + e_p, \widetilde{v} \right\}
        =\widetilde{v} - \ceil*{\frac{v}{2}} + e_p = e_p.
    \end{align}
   Thus
   \begin{align}
       \min\left\{ \widetilde{v} - \ceil*{\frac{v}{2}} + e_p, \widetilde{v} \right\}
       = \min\{ e_p, \widetilde{v} \}
   \end{align}
   and the exponent becomes
   \begin{align}
       v(N_{d,n}-1) + \widetilde{v} - (n+1) \min\{ e_p, \widetilde{v} \}.
   \end{align}
   Therefore we get
   \begin{align}
       \# \partial W_p(\xi,\s)(p^{v}) \leq \# U \leq 
       \frac{p^{vN_{d,n}}}{p^{v - \widetilde{v} + (n+1)\min\{e_p,\widetilde{v}  \}}}.
   \end{align}
    
\end{proof}

\begin{proposition}\label{prop-measure-of-padic-solvable-locus}
Let $p \in S_{\fin}$.
Then we have
\begin{align}
\left(1- \frac{1}{p^{ N_{d,n}-1}} - \frac{1}{p^{n}} \right) \frac{1}{p^{e_{p}}}\leq \mu_{p}(W_{p}(\xi, \s))\leq \left(1- \frac{1}{p^{ N_{d,n}-1}} \right) \frac{1}{p^{e_{p}}}.
\end{align}
Or equivalently,
\begin{align}
    \left(1- \frac{1}{p^{ N_{d,n}-1}} - \frac{1}{p^{n}} \right) \frac{1}{p^{e_{p}}}\leq (1-p^{-N_{d,n}})\mu_{\P^{N_{d,n}-1}_{\Q_p}}(Z_{p}(\xi, \s))\leq \left(1- \frac{1}{p^{ N_{d,n}-1}} \right) \frac{1}{p^{e_{p}}}.
\end{align}
\end{proposition}

\begin{proof}
Consider the following set:
\begin{align}
Y \coloneqq \left\{ \aa \in ( \Z_{p}^{ N_{d,n}})_{\prim} \ \middle|\  \langle \aa , \nu_{d,n}(\xii_{p}) \rangle \equiv 0 \ \mod {p^{e_{p}}} \right\}.
\end{align}
Then we have 
\begin{claim}\label{claim:measure-of-W}\ 
\begin{enumerate}
\item
$Y = \left\{ \aa \in ( \Z_{p}^{ N_{d,n}})_{\prim}  \ \middle|\  B(\aa, p^{-e_{p}})\cap W_{p}(\xi,\s) \neq  \emptyset\right\}$
\item
$\mu_{p}(Y) - \frac{1}{p^{e_{p}+n}} \leq \mu_{p}(W_{p}(\xi, \s)) \leq \mu_{p}(Y)$
\item
$\mu_{p}(Y) = \left( 1 - \frac{1}{p^{ N_{d,n}-1}}\right) \frac{1}{p^{e_{p}}}$
\end{enumerate}
\end{claim}

\begin{proof}[Proof of \cref{claim:measure-of-W}]
    (1) Let $Y'$ denote the right hand side. Suppose $\aa \in Y'$.
    Take a point $\bb \in B(\aa, p^{-e_p}) \cap W_p(\xi,\s)$.
    Then we have
    \begin{align}
        \langle \aa, \nu_{d,n}(\xii_p) \rangle \equiv \langle \bb , \nu_{d,n}(\xii_p) \rangle
        \equiv 0 \ \mod {p^{e_p}}
    \end{align}
    and thus $\aa \in Y$.
    Conversely suppose $\aa \in Y$.
    Then we can write $\langle \aa, \nu_{d,n}(\xii_p) \rangle = p^{e_p} \alpha$
    for some $\alpha \in \Z_p$. We shall find $\cc \in \Z_p^{N_{d,n}}$
    such that
    \begin{align}
        \langle \aa + p^{e_p}\cc, \nu_{d,n}(\xii_p) \rangle = 0.
    \end{align}
    We have
    \begin{align}
        \langle \aa + p^{e_p}\cc, \nu_{d,n}(\xii_p) \rangle = 
        p^{e_p}\alpha + p^{e_p} \langle \cc , \nu_{d,n}(\xii_p) \rangle.
    \end{align}
    Since $\|\xii_p\|_p = 1$, we have $\| \nu_{d,n}(\xii_p)\|_p = 1$ as well
    and thus $\langle{}\cdot{},\nu_{d,n}(\xii_p) \rangle \colon \Z_p^{N_{d,n}} \longrightarrow \Z_p$
    is surjective. Therefore we can find $\cc \in \Z_p^{N_{d,n}}$
    such that $\langle \cc ,\nu_{d,n}(\xii_p) \rangle = - \alpha$ and we are done.

    (2)
    Obviously we have
    \begin{align}
        \bigcup_{\substack{\aa \in (\Z_p^{N_{d,n}})_{\prim} \\ B(\aa,p^{-e_p})\subset W_p(\xi,\s)}} B(\aa,p^{-e_p}) \subset W_p(\xi,\s)
        \subset \bigcup_{\substack{\aa \in (\Z_p^{N_{d,n}})_{\prim} \\ B(\aa,p^{-e_p})\cap W_p(\xi,\s) \neq \emptyset}} B(\aa,p^{-e_p}) = Y.
    \end{align}
    Thus 
    \begin{align}
        \mu_p(Y) - \mu_p\left( \bigcup_{ \aa \in \partial W_p(\xi,\s)(p^{e_p}) }B(\aa, p^{-e_p})  \right) \leq \mu_p(W_p(\xi,\s)) \leq \mu_p(Y).
    \end{align}
    By \cref{boundary-ball-estimate}(2), we have
    \begin{align}
        \mu_p\left( \bigcup_{ \aa \in \partial W_p(\xi,\s)(p^{e_p}) }B(\aa, p^{-e_p})  \right)
        \leq \frac{p^{e_p N_{d,n}}}{p^{e_p + n}} \frac{1}{p^{e_p N_{d,n}}} = \frac{1}{p^{e_p + n}}
    \end{align}
    and we are done.

    (3)
    We have
    \begin{align}
        Y =&\left\{ \aa \in \Z_p^{N_{d,n}} \ \middle|\ \langle \aa,  \nu_{d,n}(\xii_p) \rangle \equiv 0 \ \mod {p^{e_p}} \right\}\\
        &\setminus 
        p\left\{ \aa \in \Z_p^{N_{d,n}} \ \middle|\ \langle \aa, \nu_{d,n}(\xii_p) \rangle \equiv 0 \ \mod {p^{e_p-1}} \right\}.
    \end{align}
    Thus
    \begin{align}
        \mu_p(Y) = \frac{1}{p^{e_p}} - \frac{1}{p^{e_p-1}} \frac{1}{p^{N_{d,n}}}
    \end{align}
    and this is what we wanted.
\end{proof}
The first assertion follows from this claim.
The second one follows from \cref{lem:comparison-of-measures}.
\end{proof}

\begin{lemma}\label{lem:bound-non-soluble-locus}
We have 
\begin{align}
\mu_{\P^{ N_{d,n}-1}_{\Q_{p}}}(\P^{ N_{d,n}-1}(\Q_{p}) \setminus Z_{p}) \ll \frac{1}{p^{2}}
\end{align}
where the implicit constant depends only on $d,n$.
\end{lemma}
\begin{proof}
    By \cref{lem:comparison-of-measures}, it is enough to show that
    \begin{align}
        \mu_p((\Z_p^{N_{d,n}})_{\prim} \setminus W_p) \ll \frac{1}{p^2}.
    \end{align}
    By \cref{Lang-Weil-range}, we have
    \begin{align}
        (\Z_p^{N_{d,n}})_{\prim} \setminus W_p
        \subset \{ \aa \in \Z_p^{N_{d,n}} \mid \aa\ \mod p \in \NIP_{d,n}(\F_p) \}.
    \end{align}
    Therefore
    \begin{align}
        &\mu_p((\Z_p^{N_{d,n}})_{\prim} \setminus W_p) \\
        &\leq \mu_p\left( \{ \aa \in \Z_p^{N_{d,n}} \mid \aa\ \mod p \in \NIP_{d,n}(\F_p) \}\right)\\
        &= \frac{\# \NIP_{d,n}(\F_p)}{p^{N_{d,n}}} \ll \frac{p^{N_{d,n}-2}}{p^{N_{d,n}}} = \frac{1}{p^2}.
    \end{align}
    Here we use \cref{dim-of-NIP}.
\end{proof}

\subsection{Bounding \texorpdfstring{$E(A,P)$}{EAP}}
\leavevmode
In this subsection, we prove the following upper bound of $E(A,P)$.
The idea of the proof is based on the method called Ekedahl sieve,
originally carried out in \cite{Ekedahl:InfiniteCRT}.
A quantitative version is developed in \cite[Theorem 3.3]{bhargava2014geometric},
which is applicable to fairly general situation.
Our proof is similar to those of them, but we need an estimate that is independent of
$\s_\infty$ and $q$, which is not directly follows.

\begin{proposition}\label{bound-of-E}
There exists $P_{0} \in \R_{\geq 1}$ that depends only on $d,n$ such that
for all $A \geq 1$ and $P > P_{0}$, we have
\begin{align}
E(A,P) \ll \frac{1}{P\log 2P} \frac{\s_{\infty} A^{ N_{d,n}}}{q} + A^{ N_{d,n}-1}\log \log 3A
\end{align}
where the implicit constant depends only on $d,n$.
\end{proposition}

We are proving this proposition in several steps.

Recall that we defined the subscheme $\NIP_{d,n} \subset \A^{ N_{d,n}}_{\Z}$, which contains 
all non-irreducible polynomials.
By \cref{dim-of-NIP}, we have $\dim \NIP_{d,n} \leq N_{d,n}-1$.
Since $\NIP_{d,n}$ is of finite type over $\Z$, we have $\dim (\NIP_{d,n})_{\Q} \leq N_{d,n}-2$,
where $(\NIP_{d,n})_{\Q}$ is the generic fiber over $\Z$.
In particular, we can pick a linear projection
\begin{align}
\pi_{\Q} \colon \A^{ N_{d,n}}_{\Q} \longrightarrow \A^{ N_{d,n}-1}_{\Q}
\end{align}
such that the induced morphism
\begin{align}
(\NIP_{d,n})_{\Q} \longrightarrow \A^{ N_{d,n}-1}_{\Q}
\end{align}
is finite.
By composing a linear automorphism if necessary, we may assume that 
there is a linear morphism $\pi \colon \A^{ N_{d,n}}_{\Z} \longrightarrow \A^{ N_{d,n}-1}_{\Z}$
such that the base change of $\pi$ to $\Q$ is $\pi_{\Q}$.
By a standard argument of ``passage to the limit'', 
there is $R$, which is the localization of $\Z$ by a single element, with the following properties.
Consider the following diagram:
 \[
\xymatrix{
(\NIP_{d,n})_{\Q} \ar@{}[d]|{\bigcap}  \ar[r] & (\NIP_{d,n})_{R}   \ar@{}[d]|{\bigcap} \ar[r]& \NIP_{d,n} \ar@{}[d]|{\bigcap} \\
\A^{ N_{d,n}}_{\Q} \ar[r] \ar[d]_{\pi_{\Q}} &\A^{ N_{d,n}}_{R} \ar[d]_{\pi_{R}} \ar[r]& \A^{ N_{d,n}}_{\Z} \ar[d]^{\pi} \\
*+[r]{\A^{ N_{d,n}-1}_{\Q}} \ar[r] & *+[r]{\A^{ N_{d,n}-1}_{R}}  \ar[r]& *+[r]{\A^{ N_{d,n}-1}_{\Z}} 
}
\]
where $(\NIP_{d,n})_{R}, \pi_{R}$ are the base change to $R$.
Then $\pi_{R}|_{(\NIP_{d,n})_{R}}$ is a finite morphism and 
the coordinate ring $ \mathcal{O}((\NIP_{d,n})_{R})$ of $(\NIP_{d,n})_{R}$ is generated by, say, $r$
elements over $ \mathcal{O}(\A^{ N_{d,n}-1}_{R})$.
Since $\dim (\NIP_{d,n})_{\Q} \leq  N_{d,n}-2$, $\NIP_{d,n}$ does not dominate $\A^{ N_{d,n}-1}_{\Z}$.
We pick a non-zero polynomial $f \in \mathcal{O}(\A^{ N_{d,n}-1}_{\Z})$
such that 
\begin{align}
\pi(\NIP_{d,n}) \subset V(f) \subset \A^{ N_{d,n}-1}_{\Z}.
\end{align}
Here $V(f)$ is the closed subscheme of $\A^{ N_{d,n}-1}_{\Z}$ defined by $f$.

Now we take $P_{0} \in \R_{\geq 1}$ so that 
\cref{Lang-Weil-range} works for this $P_{0}$ and also
any prime $p > P_{0}$ generates a prime ideal of $R$.

Choosing $P_{0}$ as above, for $A \geq 1$ and $P>P_{0}$, we have
\begin{align}
E(A,P)
=& \# \left\{ \aa \in \Z_{ \prim}^{ N_{d,n}} \ \middle|\ \txt{$ \| \aa \|\leq A, \forall p \in S \cup M_{\Q, \leq P}, [\aa] \in Z_{p}(\xi,\s)$\\
$\exists p > P$ prime number, $[\aa] \notin Z_{p}$} \right\}\\
\leq & \#\left\{ \aa \in \Z^{ N_{d,n}} \ \middle|\ \txt{$ 0<\| \aa \|\leq A, \forall p \in S, [\aa] \in Z_{p}(\xi,\s)$\\
$\exists p > P$ such that $p \notin S$ and $\aa\ \mod p \in \NIP_{d,n}(\F_{p})$}  \right\}\\
\leq& \#\left\{ \aa \in \Z^{ N_{d,n}} \ \middle|\ \txt{$ 0<\| \aa \|\leq A, \forall p \in S, [\aa] \in Z_{p}(\xi,\s)$\\
$P < \exists p \leq A$ s.t. $p \notin S$, $\aa\ \mod p \in \NIP_{d,n}(\F_{p})$}  \right\}\\
+&\#\left\{ \aa \in \Z^{ N_{d,n}} \ \middle|\ \txt{$ 0<\| \aa \|\leq A, \forall p \in S, [\aa] \in Z_{p}(\xi,\s)$\\
$\exists p > \max\{P,A\}$ s.t. $p \notin S$, $\aa\ \mod p \in \NIP_{d,n}(\F_{p})$}  \right\}\\
\leq & \#\left\{ \aa \in \Z^{ N_{d,n}} \ \middle|\ \txt{$ 0<\| \aa \|\leq A, \forall p \in S, [\aa] \in Z_{p}(\xi,\s)$\\
$P < \exists p \leq A$ s.t. $p \notin S$, $\aa\ \mod p \in \NIP_{d,n}(\F_{p})$}  \right\}\\
+&\#\left\{ \aa \in \Z^{ N_{d,n}} \ \middle|\ \txt{$ \| \aa \|\leq A$\\
$\exists p > \max\{P,A\}$ s.t. $\aa\ \mod p \in \NIP_{d,n}(\F_{p})$}  \right\}.
\end{align}
Here we use \cref{Lang-Weil-range} for the first inequality.
We set
\begin{align}
E_{1}(A,P) &\coloneqq \#\left\{ \aa \in \Z^{ N_{d,n}} \ \middle|\ \txt{$ 0<\| \aa \|\leq A, \forall p \in S, [\aa] \in Z_{p}(\xi,\s)$\\
$P < \exists p \leq A$ s.t. $p \notin S$, $\aa\ \mod p \in \NIP_{d,n}(\F_{p})$}  \right\}\\
E_{2}(A,P) &\coloneqq \#\left\{ \aa \in \Z^{ N_{d,n}} \ \middle|\ \txt{$ \| \aa \|\leq A$\\
$\exists p > \max\{P,A\}$ s.t. $\aa\ \mod p \in \NIP_{d,n}(\F_{p})$}  \right\}.
\end{align}

We first bound $E_{2}(A,P)$.

\begin{lemma}\label{bound-of-E2}
Let $P_{0}$ as above.
For $A\geq 1$ and $P>P_{0}$, we have
\begin{align}
E_{2}(A,P) \ll A^{ N_{d,n}-1}
\end{align}
where the implicit constant depends only on $d,n$.
\end{lemma}

\begin{proof}
We may assume $A\ge2$.
We further divide $E_{2}(A,P)$ into two part as follows:
\begin{align}
E_{2}(A,P) \leq 
 &\#\left\{ \aa \in \Z^{ N_{d,n}} \ \middle|\ \txt{$ \| \aa \|\leq A, f(\pi(\aa))\neq 0$\\
$\exists p > \max\{P,A\}$ s.t. $\aa \mod p \in \NIP_{d,n}(\F_{p})$}  \right\}\\
&+  \#\left\{ \aa \in \Z^{ N_{d,n}} \ \middle|\ \txt{$ \| \aa \|\leq A, f(\pi(\aa))=0$\\
$\exists p > \max\{P,A\}$ s.t. $\aa \mod p \in \NIP_{d,n}(\F_{p})$}  \right\}\\
\leq & 
\#\left\{ \aa \in \Z^{ N_{d,n}} \ \middle|\ \txt{$ \| \aa \|\leq A, f(\pi(\aa))\neq 0$\\
$\exists p > \max\{P,A\}$ s.t. $\aa \mod p \in \NIP_{d,n}(\F_{p})$}  \right\}\\
&+  \#\left\{ \aa \in \Z^{ N_{d,n}} \ \middle|\  \| \aa \|\leq A, f(\pi(\aa))=0 \right\}.
\end{align}
Since $f\circ \pi$ is a non-zero polynomial function on $\A^{ N_{d,n}}_{\Z}$,
we have
\begin{align}
\#\left\{ \aa \in \Z^{ N_{d,n}} \ \middle|\  \| \aa \|\leq A, f(\pi(\aa))=0 \right\} \ll A^{  N_{d,n}-1} \label{eq:E2bound-1}
\end{align}
where the implicit constant depends only on $d,n,\pi,f$.

Since $\pi \colon \Z^{ N_{d,n}} \longrightarrow \Z^{ N_{d,n}-1}$ is a linear map,
there is $ \| \pi \| \geq 0$ such that $ \| \pi(\xx) \| \leq  \| \pi \| \| \xx \|$ for all $\xx \in \Z^{ N_{d,n}}$.
We have
\begin{align}
&\#\left\{ \aa \in \Z^{ N_{d,n}} \ \middle|\ \txt{$ \| \aa \|\leq A, f(\pi(\aa))\neq 0$\\
$\exists p > \max\{P,A\}$ s.t. $\aa \mod p \in \NIP_{d,n}(\F_{p})$}  \right\}\\
& \leq \sum_{ p > \max\{P,A\}} \sum_{\substack{\aa \in \Z^{ N_{d,n}} \\ \| \aa \|\leq A \\ f(\pi(\aa)) \neq 0 \\ \aa\ \mod p \in \NIP_{d,n}(\F_{p}) }} 1\\
& \leq 
\sum_{\substack{\bb \in \Z^{ N_{d,n}-1} \\ f(\bb) \neq 0 \\ \| \bb \| \leq \| \pi \|A}}
\sum_{\substack{ p > \max\{P,A\} \\ p|f(\bb)}}
\sum_{\substack{\aa \in \Z^{ N_{d,n}} \\ \| \aa \| \leq A \\ \aa\ \mod p \in \NIP_{d,n}(\F_{p}) \\ \pi(\aa)=\bb}} 1.\label{eq:E2bound-2}
\end{align}

Since $p > \max\{P,A\} \geq P > P_{0}$, every fiber of
\begin{align}
\pi \colon \NIP_{d,n}(\F_{p}) \longrightarrow \A^{ N_{d,n}-1}(\F_{p})
\end{align}
consists at most $r$ elements.
Noting $p > A$, we get
\begin{align}
\sum_{\substack{\aa \in \Z^{ N_{d,n}} \\ \| \aa \| \leq A \\ \aa \mod p \in \NIP_{d,n}(\F_{p}) \\ \pi(\aa)=\bb}} 1 
\ll 1
\end{align}
where the implicit constant depends only on $d,n, r$.
Thus \cref{eq:E2bound-2} is bounded by constant multiple of
\begin{align}
\sum_{\substack{\bb \in \Z^{ N_{d,n}-1} \\ f(\bb) \neq 0 \\ \| \bb \| \leq \| \pi \|A}}
\sum_{\substack{ p > \max\{P,A\} \\ p|f(\bb)}}1.
\end{align}
Let 
\begin{align}
\a = \sum_{\substack{ p > \max\{P,A\} \\ p|f(\bb)}}1.
\end{align}
Then, for $\bb \in \Z^{ N_{d,n}-1}$ with $f(\bb)\neq0$, we get
\begin{align}
A^{\a}\le\max\{P,A\}^{\a} \leq \prod_{\substack{p > \max\{P,A\} \\ p | f(\bb)}} p \leq |f(\bb)| \ll A^{\deg f}.
\end{align}
Thus we get
\begin{align}
\alpha=
\sum_{\substack{ p > \max\{P,A\} \\ p|f(\bb)}}1 \ll 1,
\end{align}
where the implicit constant depends only on $\pi, f$.
Therefore we get
\begin{align}
&\#\left\{ \aa \in \Z^{ N_{d,n}} \ \middle|\ \txt{$ \| \aa \|\leq A, f(\pi(\aa))\neq 0$\\
$\exists p > \max\{P,A\}$ s.t. $\aa \mod p \in \NIP_{d,n}(\F_{p})$}  \right\}\\
&\ll  \sum_{\substack{\bb \in \Z^{ N_{d,n}-1} \\ f(\bb) \neq 0 \\ \| \bb \| \leq \| \pi \|A}}1
\ll A^{ N_{d,n}-1}
\end{align}
where the implicit constants depend only on $d,n,r,\pi,f$.
Combined with \cref{eq:E2bound-1}, we get
\begin{align}
E_{2}(A,P) \ll A^{ N_{d,n}-1}
\end{align}
where the implicit constant depends only on $d,n,r,\pi,f$.
Since we chose $\pi, f, r$ depending only on $d,n$, we are done.
\end{proof}

Next we are going to bound $E_{1}(A,P)$.

\begin{lemma}\label{lem-E1-bound}
For $A\geq 1$ and $P\geq 1$, we have
\begin{align}
E_{1}(A,P) \ll \frac{1}{P\log 2P} \frac{\s_{\infty}A^{ N_{d,n}}}{q} + A^{ N_{d,n}-1}\log \log 3A
\end{align}
where the implicit constant depends only on $d,n$.
\end{lemma}

\begin{proof}
We have
\begin{align}
&E_{1}(A,P)\\
&= \#\left\{ \aa \in \Z^{ N_{d,n}} \ \middle|\ \txt{$ 0<\| \aa \|\leq A, \forall p \in S, [\aa] \in Z_{p}(\xi,\s)$\\
$P < \exists p \leq A$ s.t. $p \notin S$, $\aa\ \mod p \in \NIP_{d,n}(\F_{p})$}  \right\}\\
&= \# \left\{ \aa \in \Z^{ N_{d,n}} \ \middle|\   \txt{$0 < \| \aa \| \leq A, \aa \in \R W_{\infty}(\xi,\s), \forall p \in S_{\fin}, [\aa] \in Z_{p}(\xi,\s)$\\
$P < \exists p \leq A$ s.t. $p \notin S$, $\aa\ \mod p \in \NIP_{d,n}(\F_{p})$} \right\}\\
& \leq \# \left\{ \aa \in \Z^{ N_{d,n}} \ \middle|\ \txt{$\aa \in \mathcal{B}_{ N_{d,n}}(A) \cap \R W_{\infty}(\xi,\s)$\\ 
$\forall p \in S_{\fin}, | \langle \aa , \nu_{d,n}(\xii_{p}) \rangle|_{p} \leq \s_{p} \| \aa \|_{p}$\\
$P < \exists p \leq A$ s.t. $p \notin S$, $\aa\ \mod p \in \NIP_{d,n}(\F_{p})$}  \right\}.
\end{align}
We use \cref{easy-bound-Z} for the last inequality.
By Chinese remainder theorem, we can pick a $\xii_{0} \in \Z^{ N_{d,n}}$ such that
\begin{align}
\xii_{0} \equiv \xii_{p}\ \mod {p^{e_{p}}} \quad \text{for $p \in S_{\fin}$}.
\end{align}
Note that since $ \| \xii_{p} \|_{p}=1$, $\gcd(q, \xii_{0})=1$. 
We get
\begin{align}
E_{1}(A,P)  \leq 
\# \left\{ \aa \in \Z^{ N_{d,n}} \ \middle|\ \txt{$\aa \in \mathcal{B}_{ N_{d,n}}(A) \cap \R W_{\infty}(\xi,\s)$\\ 
$ \langle \aa , \nu_{d,n}(\xii_{0}) \rangle \equiv 0\ \mod q$\\
$P < \exists p \leq A$ s.t. $p \notin S$, $\aa\ \mod p \in \NIP_{d,n}(\F_{p})$}  \right\}.
\end{align}
Let 
\begin{align}
\L = \left\{ \aa \in \Z^{ N_{d,n}} \ \middle|\ \langle \aa, \nu_{d,n}(\xii_{0}) \rangle \equiv 0\ \mod q \right\}.
\end{align}
This is a lattice of rank $ N_{d,n}$. Using this lattice, we see
\begin{equation}
\label{lem-E1-bound:ineq1}
\begin{aligned}
&E_{1}(A,P) \\
&\leq \sum_{\substack{P < p \leq A\\ p \notin S}} \sum_{x \in \NIP_{d,n}(\F_{p})}
\# \left\{ \aa \in \L \cap \mathcal{B}_{ N_{d,n}}(A) \cap \R W_{\infty}(\xi, \s) \ \middle|\ \aa\ \mod p = x\right\}. 
\end{aligned}
\end{equation}
%{\Yuta{Now we can use the general lemma in the main file.}}

The set under the sum is non-empty only if
\begin{align}
\# \left\{ \aa \in \L \ \middle|\ \aa\ \mod p = x\right\} \neq  \emptyset.
\end{align}
When this is the case, pick $\aa_{0} \in \L$ such that $\aa_{0}\ \mod p = x$.
Then we have
\begin{align}
&\left\{ \aa \in \L \cap \mathcal{B}_{ N_{d,n}}(A) \cap \R W_{\infty}(\xi, \s) \ \middle|\ \aa\ \mod p = x\right\}\\
&= \L \cap \mathcal{B}_{ N_{d,n}}(A) \cap \R W_{\infty}(\xi, \s) \cap (\aa_{0} + p\Z^{ N_{d,n}})\\
&= (\aa_{0} + p \L) \cap  \mathcal{B}_{ N_{d,n}}(A) \cap \R W_{\infty}(\xi, \s). 
\end{align}
Here for the last equality, we use $\L \cap p\Z^{ N_{d,n}} = p\L$, which follows from $p \nmid q$.
Thus we get
\begin{align}
&\# \left\{ \aa \in \L \cap \mathcal{B}_{ N_{d,n}}(A) \cap \R W_{\infty}(\xi, \s) \ \middle|\ \aa\ \mod p = x\right\} \\
&= \# \left( p\L \cap \left( \mathcal{B}_{ N_{d,n}}(A) \cap \R W_{\infty}(\xi, \s) -\aa_{0}  \right)  \right).
\end{align}
We use \cref{lem:BarroeroWidmer} to estimate this number.
By \cref{W-is-sa}, the set
\begin{align}
\mathcal{B}_{ N_{d,n}}(A) \cap \R W_{\infty}(\xi, \s) -\aa_{0} 
\end{align}
comes from a semi-algebraic family of subsets of $\R^{ N_{d,n}}$
parametrized by $A, \s_{\infty}, \xii_{\infty}$, and $\aa_{0}$.
By \cref{lem:BarroeroWidmer},
we get
\begin{align}
&\# \left( p\L \cap \left( \mathcal{B}_{ N_{d,n}}(A) \cap \R W_{\infty}(\xi, \s) -\aa_{0}  \right)  \right)\\
&= \frac{\vol_{N_{d,n}}\left( \mathcal{B}_{ N_{d,n}}(A) \cap \R W_{\infty}(\xi, \s)\right)}{ \det (p\L)}
+O\left(\sum_{\nu=0}^{ N_{d,n}-1}  \frac{V_{\nu}\left( \mathcal{B}_{ N_{d,n}}(A) \cap \R W_{\infty}(\xi, \s)\right)}{\l_{1}(p\L) \cdots \l_{\nu}(p\L)} \right)\\
& \ll \frac{\s_{\infty} A^{ N_{d,n}}}{ p^{ N_{d,n}} \det \L} + O\left( \sum_{\nu=0}^{ N_{d,n}-1} \frac{A^{\nu}}{p^{\nu}}  \right)
\ll \frac{\s_{\infty} A^{ N_{d,n}}}{ qp^{ N_{d,n}}} + \frac{A^{ N_{d,n}-1}}{p^{ N_{d,n}-1}}
\end{align}
where the implicit constants depend only on $d,n$.
Here we use \cref{easy-bound-Z} for the first ``$\ll$'' and
the fact $\det \L = q$ and $p \leq A$ for the second ``$\ll$''.

Plugging the above estimate into \cref{lem-E1-bound:ineq1}, we get
\begin{align}
E_{1}(A,P) \ll \sum_{P < p \leq A} \# \NIP_{d,n}(\F_{p})\left( \frac{\s_{\infty} A^{ N_{d,n}}}{ qp^{ N_{d,n}}} + \frac{A^{ N_{d,n}-1}}{p^{ N_{d,n}-1}} \right).
\end{align}
Since $\dim (\NIP_{d,n})_\Q \leq N_{d,n}-2$ (\cref{dim-of-NIP}), we have
\begin{align}
\# \NIP_{d,n}(\F_{p}) \ll p^{ N_{d,n}-2}
\end{align}
where the implicit constant depends only on $d,n$.
Therefore, we get
\begin{align}
E_{1}(A,P)& \ll \sum_{P < p \leq A} \left( \frac{\s_{\infty} A^{ N_{d,n}}}{q} \frac{1}{p^{2}} + A^{ N_{d,n}-1} \frac{1}{p}     \right)\\
& \ll \frac{1}{P\log 2P} \frac{\s_{\infty} A^{ N_{d,n}}}{q} + A^{ N_{d,n}-1} \log \log 3A
\end{align}
where the implicit constants depends only on $d,n$.
Here we use the following two bound that follows from the prime number theorem:
\begin{align}
\sum_{P < p \leq A} \frac{1}{p^{2}} \leq \sum_{P < p} \frac{1}{p^{2}} \ll \frac{1}{P\log 2P}
\and
\sum_{P<p \leq A} \frac{1}{p} \leq \sum_{p \leq A} \frac{1}{p} \ll \log \log 3A.
\end{align}
\end{proof}

\begin{proof}[Proof of \cref{bound-of-E}]
This follows from \cref{lem-E1-bound} and \cref{bound-of-E2}.
\end{proof}

\subsection{Asymptotic formula for  \texorpdfstring{$M(A,P)$}{MAP}}

Recall that
\begin{align}
M(A,P)
\coloneqq{}&\# \left\{ \aa \in \Z_{ \prim}^{ N_{d,n}} \ \middle|\ \| \aa \|\leq A, \forall p \in S \cup M_{\Q, \leq P},\ [\aa] \in Z_{p}(\xi,\s)\right\}\\
={}&\# \left\{ \aa \in \Z_{\prim}^{ N_{d,n}} \ \middle|\ \txt{ $\| \aa \|\leq A, \aa \in \R W_{\infty}(\xi,\s)$\\ $ \forall p \in S_{\fin} \cup M_{\Q, \leq P},\ \aa \in W_{p}(\xi,\s)$} \right\}.
\end{align}
We are going to get an asymptotic formula for $M(A,P)$.
The difficulty is that we do not know the precise shape of the subsets 
$W_p(\xi, \s) \subset \Z_p^{N_{d,n}}$.
We approximate the set $W_p(\xi,\s)$ by finitely many disjoint balls with the same radius from inside and outside. 
Contained in a $p$-adic ball is equivalent to a certain congruence condition,
so we can apply \cref{prop:lattice_count_primitive_local_general} to count lattice points.

Let us start with preparing the balls. Consider the following two sets of balls in $ \Z_{p}^{ N_{d,n}}$, which approximate the sets $W_{p}(\xi, \s)$ from inside and outside:
for prime $p$ and $v_{p} \in \Z_{\geq 1}$, we define
\begin{align}
\Omega_{0}(p^{v_{p}}) &\coloneqq \left\{ B(\cc, p^{-v_{p}}) \ \middle|\  \cc \in \Z_{p}^{ N_{d,n}}, B(\cc, p^{-v_{p}}) \subset W_{p}(\xi,\s)  \right\}\\
\Omega_{1}(p^{v_{p}}) &\coloneqq \left\{ B(\cc, p^{-v_{p}}) \ \middle|\  \cc \in \Z_{p}^{ N_{d,n}}, B(\cc, p^{-v_{p}}) \cap W_{p}(\xi,\s) \neq  \emptyset  \right\}.
\end{align}

\begin{lemma}\label{lem:bound-of-omega}
Let $p$ be a prime. 
For $v_{p} \in \Z_{\geq 1}$, we set
\begin{align}
\widetilde{v_{p}} &= \min \left\{  \ceil*{ \frac{v_{p}}{2}}, v_{p}-e_{p}+1\right\}\\
\nu(v_{p},e_{p}) & = v_{p} - \widetilde{v_{p}} +(n+1)\min\{e_{p}, \widetilde{v_{p}}\}.
\end{align}
Then we have
\begin{enumerate}
\item
\begin{align}
\# \left(  \Omega_{1}(p^{v_{p}}) \setminus \Omega_{0}(p^{v_{p}})\right) \leq \frac{p^{v_{p} N_{d,n}}}{p^{\nu(v_{p},e_{p})}}
\end{align}
\item
\begin{align}
\frac{ \# \Omega_{1}(p^{v_{p}}) }{p^{v_{p} N_{d,n}}} &\leq (1-p^{- N_{d,n}}) \mu_{\P^{ N_{d,n}-1}_{\Q_p}}(Z_{p}(\xi,\s)) + \frac{1}{p^{\nu(v_{p},e_{p})}}\\
\frac{ \# \Omega_{0}(p^{v_{p}}) }{p^{v_{p} N_{d,n}}} &\geq (1-p^{- N_{d,n}}) \mu_{\P^{ N_{d,n}-1}_{\Q_p}}(Z_{p}(\xi,\s)) - \frac{1}{p^{\nu(v_{p},e_{p})}}.
\end{align}
\end{enumerate}
\end{lemma}
\begin{proof}
(1) is a restatement of \cref{boundary-ball-estimate}.
For (2), recall that 
\begin{align}
\mu_{p}(B(\cc, p^{-v_{p}})) = \frac{1}{p^{v_{p} N_{d,n}}}
\end{align}
for all $\cc \in \Z_{p}^{ N_{d,n}}$.
Thus, combined with (1), we get
\begin{align}
\frac{ \# \Omega_{1}(p^{v_{p}}) }{p^{v_{p} N_{d,n}}}  - \frac{1}{p^{\nu(v_{p},e_{p})}} \leq \frac{ \# \Omega_{0}(p^{v_{p}}) }{p^{v_{p} N_{d,n}}} 
\leq \mu_{p}(W_{p}(\xi,\s)) &\leq \frac{ \# \Omega_{1}(p^{v_{p}}) }{p^{v_{p} N_{d,n}}} \\
&\leq \frac{ \# \Omega_{0}(p^{v_{p}}) }{p^{v_{p} N_{d,n}}} + \frac{1}{p^{\nu(v_{p},e_{p})}}.
\end{align}
By definition of $\mu_{\P^{N_{d,n}-1}_{\Q_p}}$, we have
\begin{align}
\mu_{p}(W_{p}(\xi,\s)) = (1- p^{ - N_{d,n}}) \mu_{\P^{ N_{d,n}-1}_{\Q_p}}(Z_{p}(\xi,\s))
\end{align}
and we are done.
\end{proof}

\begin{proposition}\label{prop:M-upper-lower-bound-general-form}
Let $v_{p} \in \Z_{\geq 1}$ for each $p \in S_{\fin} \cup M_{\Q, \leq P}$.
Set
\begin{align}
q' \coloneqq \prod_{p \in S_{\fin} \cup M_{\Q, \leq P}} p^{v_{p}}.
\end{align}
Recall that we write 
\begin{align}
    \rho_{\infty}(\xi, \s) &= \mu_{S^{ N_{d,n}-1}}(W_{\infty}(\xi,\s)),\\
    \rho_{p}(\xi, \s) &= \mu_{\P^{N_{d,n}-1}_{\Q_p}}(Z_p(\xi,\s)) \quad p \neq \infty.
\end{align}
For $A, P \geq 1$, we have
\begin{align}
&M(A,P)\\
&\leq 
\frac{V_{ N_{d,n}}\prod_{p \in M_{\Q}} \rho_{p}(\xi, \s)}{\zeta( N_{d,n})} A^{ N_{d,n}}
\prod_{p \in S_{\fin} \cup M_{\Q, \leq P}}
\left(1 + \frac{1}{(1-p^{- N_{d,n}}) \rho_{p}(\xi, \s) p^{\nu(v_{p},e_{p})}} \right)\\
&\times \left(1 + O\left( \frac{1}{P\log 2P}\right) \right) 
\left(1 + O\left(  \frac{J_{ N_{d,n}}(q')}{\rho_{\infty}(\xi, \s) q'^{ N_{d,n}-1} A}  + \frac{J_{ N_{d,n}}(q')\log 2A}{\rho_{\infty}(\xi, \s) \f(q') A^{N_{d,n}-1}}\right)   \right)
\end{align}
and
\begin{align}
&M(A,P)\\
&\geq
\frac{V_{ N_{d,n}}\prod_{p \in M_{\Q}} \rho_{p}(\xi, \s)}{\zeta( N_{d,n})} A^{ N_{d,n}}
\prod_{p \in S_{\fin} \cup M_{\Q, \leq P}}
\left(1 - \frac{1}{(1-p^{- N_{d,n}}) \rho_{p}(\xi, \s) p^{\nu(v_{p},e_{p})}} \right)\\
&\times \left(1 + O\left( \frac{1}{P\log 2P}\right) \right) 
\left(1 + O\left(  \frac{J_{ N_{d,n}}(q')}{\rho_{\infty}(\xi, \s) q'^{ N_{d,n}-1} A}  + \frac{J_{ N_{d,n}}(q')\log 2A}{\rho_{\infty}(\xi, \s) \f(q') A^{N_{d,n}-1}}\right)   \right)
\end{align}
where the implicit constants depend only on $d,n$.
\end{proposition}
\begin{proof}
Let us write 
$
\mathcal{S} = S_{\fin} \cup M_{\Q, \leq P}.
$
We have
\begin{align}
M(A,P) &\leq  \sum_{\substack{(B_{p}) \in \prod_{p \in \mathcal{S}} \Omega_{1}(p^{v_{p}})  }} \# \left\{ \aa \in \Z^{ N_{d,n}}_{\prim} \ \middle|\  \txt{$\aa \in \ball{ N_{d,n}}{A} \cap \R W_{\infty}(\xi,\s)$\\
$\forall p \in \mathcal{S}, \aa \in B_{p}$} \right\},\\
M(A,P) &\geq  \sum_{\substack{(B_{p}) \in \prod_{p \in \mathcal{S}} \Omega_{0}(p^{v_{p}})  }} \# \left\{ \aa \in \Z^{ N_{d,n}}_{\prim} \ \middle|\  \txt{$\aa \in \ball{ N_{d,n}}{A} \cap \R W_{\infty}(\xi,\s)$\\
$\forall p \in \mathcal{S}, \aa \in B_{p}$} \right\}.
\end{align}

We identify a ball in $ \Z_{p}^{ N_{d,n}}$ of radius $p^{-v_{p}}$ with an element of $(\Z/p^{v_{p}}\Z)^{ N_{d,n}}$.
By Chinese remainder theorem, we have $\prod_{p \in \mathcal{S}} (\Z/p^{v_{p}}\Z)^{ N_{d,n}} \simeq (\Z/q' \Z)^{ N_{d,n}} $.
Let $ \mathcal{C}_{i}\  (i=0,1)$ be the image of $\prod_{p \in \mathcal{S}} \Omega_{i}(p^{v_{p}})\ (i=0,1)$ by this bijection:
 \[
\xymatrix{
\prod_{p \in \mathcal{S}} (\Z/p^{v_{p}}\Z)^{ N_{d,n}} \ar[r]^(.6){\sim} & (\Z/q' \Z)^{ N_{d,n}}\\
\prod_{p \in \mathcal{S}} (\Z/p^{v_{p}}\Z)^{ N_{d,n}}_{\prim} \ar@{}[u]|{\bigcup} \ar[r]^(.6){\sim} & (\Z/q' \Z)^{ N_{d,n}}_{\prim} \ar@{}[u]|{\bigcup}\\
\prod_{p \in \mathcal{S}} \Omega_{i}(p^{v_{p}}) \ar@{}[u]|{\bigcup} \ar[r]^(.6){\sim} & \mathcal{C}_{i} \ar@{}[u]|{\bigcup}
}
\]
Here note that all the elements of $\Omega_{i}(p^{v_{p}})$ are contained in $(\Z/p^{v_{p}}\Z)^{ N_{d,n}}_{\prim}$  because 
$W_{p}(\xi,\s) \subset ( \Z_{p}^{ N_{d,n}})_{\prim}$.
Also, by the definition of $W_{p}(\xi,\s)$, $ \mathcal{C}_{i}$ are stable under multiplication by any elements of $(\Z/q'\Z)^{\times}$.
With this notation, we get
\begin{align}
    &M(A,P)\\
    &\leq \sum_{\cc \in \mathcal{C}_1} \# \left\{ \aa \in \Z^{ N_{d,n}}_{\prim} \ \middle|\  \txt{$\aa \in \ball{ N_{d,n}}{A} \cap \R W_{\infty}(\xi,\s)$\\
    $\aa \equiv \cc\ \mod{q'}$} \right\}\\
    & = \sum_{\cc \in \mathcal{C}_1}
        \#\left( \Zprim^{N_{d,n}} \cap \ball{N_{d,n}}{A}\cap \R W_{\infty}(\xi,\s) \cap (\cc + q' \Z^{N_{d,n}})\right)\\
    & = \frac{1}{\f(q')} \sum_{u \in (\Z/q'\Z)^\times} \sum_{\cc \in \mathcal{C}_1}
          \#\left( \Zprim^{N_{d,n}} \cap \ball{N_{d,n}}{A}\cap \R W_{\infty}(\xi,\s) \cap (u\cc + q' \Z^{N_{d,n}})\right) \\
    & = \frac{1}{\f(q')}  \sum_{\cc \in \mathcal{C}_1} \sum_{u \in (\Z/q'\Z)^\times}
          \#\left( \Zprim^{N_{d,n}} \cap \ball{N_{d,n}}{A}\cap \R W_{\infty}(\xi,\s) \cap (u\cc + q' \Z^{N_{d,n}})\right).
\end{align}
Similarly,
\begin{align}
    &M(A,P)\\
    &\geq \sum_{\cc \in \mathcal{C}_0} \# \left\{ \aa \in \Z^{ N_{d,n}}_{\prim} \ \middle|\  \txt{$\aa \in \ball{ N_{d,n}}{A} \cap \R W_{\infty}(\xi,\s)$\\
    $\aa \equiv \cc\ \mod{q'}$} \right\}\\
    & = \frac{1}{\f(q')}  \sum_{\cc \in \mathcal{C}_0} \sum_{u \in (\Z/q'\Z)^\times}
          \#\left( \Zprim^{N_{d,n}} \cap \ball{N_{d,n}}{A}\cap \R W_{\infty}(\xi,\s) \cap (u\cc + q' \Z^{N_{d,n}})\right).
\end{align}

Note that $\R W_{\infty}(\xi,\s)$ form a semialgebraic family of homogeneous sets
when $\xi$ and $\s$ varies.
Applying \cref{prop:lattice_count_primitive_local_general} with
$\Gamma = \Lambda = \Z^{N_{d,n}}$, for all $A \geq 1$ we have
\begin{align}
    &\sum_{u \in (\Z/q'\Z)^\times}\#\left( \Zprim^{N_{d,n}} \cap \ball{N_{d,n}}{A}\cap \R W_{\infty}(\xi,\s) \cap (u\cc + q' \Z^{N_{d,n}})\right)\\
    & = \frac{\f(q') \vol_{N_{d,n}}(\R W_{\infty}(\xi,\s) \cap \ball{N_{d,n}}{1}) A^{N_{d,n}}}{J_{N_{d,n}}(q') \zeta(N_{d,n})} 
    + O\left( \frac{\f(q') A^{N_{d,n}-1}}{q'^{N_{d,n}-1}} + A\log 2A \right)\\
    & = \frac{\f(q') \rho_{\infty}(\xi, \s) V_{N_{d,n}} A^{N_{d,n}}}{J_{N_{d,n}}(q') \zeta(N_{d,n})} 
    + O\left( \frac{\f(q') A^{N_{d,n}-1}}{q'^{N_{d,n}-1}} + A\log 2A \right)
\end{align}
where the implicit constants depend only on $d,n$.
Therefore we get
\begin{equation}
\begin{aligned}\label{MAP-boudn-C1-C0}
    &M(A,P) \leq \# \mathcal{C}_1 \left\{
    \frac{ \rho_{\infty}(\xi, \s) V_{N_{d,n}} A^{N_{d,n}}}{J_{N_{d,n}}(q') \zeta(N_{d,n})} 
    + O\left( \frac{ A^{N_{d,n}-1}}{q'^{N_{d,n}-1}} + \frac{A\log 2A}{\f(q')} \right)
    \right\}\\
    &M(A,P) \geq \# \mathcal{C}_0 \left\{
    \frac{ \rho_{\infty}(\xi, \s) V_{N_{d,n}} A^{N_{d,n}}}{J_{N_{d,n}}(q') \zeta(N_{d,n})} 
    + O\left( \frac{ A^{N_{d,n}-1}}{q'^{N_{d,n}-1}} + \frac{A\log 2A}{\f(q')} \right)
    \right\}.
\end{aligned}
\end{equation}

Now by \cref{lem:bound-of-omega} we have
\begin{align}
    \# \mathcal{C}_1 &= \prod_{p \in \mathcal{S}} \# \Omega_{1}(p^{v_p})
    \leq \prod_{p \in \mathcal{S}} p^{v_p N_{d,n}} 
    \left\{(1-p^{-N_{d,n}})\rho_{p}(\xi,\s) + \frac{1}{p^{\nu(v_p, e_p)}}  \right\}\\
    &= J_{N_{d,n}}(q') \prod_{p \in \mathcal{S}} \rho_{p}(\xi,\s)
    \left( 1 + \frac{1}{(1-p^{-N_{d,n}})\rho_{p}(\xi,\s)p^{\nu(v_p, e_p)}}\right)
\end{align}
and
\begin{align}
    \# \mathcal{C}_0 &= \prod_{p \in \mathcal{S}} \# \Omega_{0}(p^{v_p})
    \geq \prod_{p \in \mathcal{S}} p^{v_p N_{d,n}} 
    \left\{(1-p^{-N_{d,n}})\rho_{p}(\xi,\s) - \frac{1}{p^{\nu(v_p, e_p)}}  \right\}\\
    &= J_{N_{d,n}}(q') \prod_{p \in \mathcal{S}} \rho_{p}(\xi,\s)
    \left( 1 - \frac{1}{(1-p^{-N_{d,n}})\rho_{p}(\xi,\s)p^{\nu(v_p, e_p)}}\right).
\end{align}
Plugging into \cref{MAP-boudn-C1-C0}, we get
\begin{equation}
    \begin{aligned}\label{ineq:MAP-bounds}
        M(A,P)& \leq  \prod_{p \in \mathcal{S}}  \rho_{p}(\xi,\s)
    \left( 1 + \frac{1}{(1-p^{-N_{d,n}})\rho_{p}(\xi,\s)p^{\nu(v_p, e_p)}}\right)\\
    &\times \left\{
    \frac{ \rho_{\infty}(\xi, \s) V_{N_{d,n}} A^{N_{d,n}}}{\zeta(N_{d,n})} 
    + O\left( \frac{J_{N_{d,n}}(q') A^{N_{d,n}-1}}{q'^{N_{d,n}-1}} + \frac{J_{N_{d,n}}(q')A\log 2A}{\f(q')} \right)
    \right\}\\
    M(A,P)& \geq  \prod_{p \in \mathcal{S}}  \rho_{p}(\xi,\s)
    \left( 1 - \frac{1}{(1-p^{-N_{d,n}})\rho_{p}(\xi,\s)p^{\nu(v_p, e_p)}}\right)\\
    &\times \left\{
    \frac{ \rho_{\infty}(\xi, \s) V_{N_{d,n}} A^{N_{d,n}}}{\zeta(N_{d,n})} 
    + O\left( \frac{J_{N_{d,n}}(q') A^{N_{d,n}-1}}{q'^{N_{d,n}-1}} + \frac{J_{N_{d,n}}(q')A\log 2A}{\f(q')} \right)
    \right\}
    \end{aligned}
\end{equation}
where the implicit constant depend only on $d,n$.

We want to replace $\prod_{p \in \mathcal{S}}  \rho_{p}(\xi,\s)$
with $\prod_{p \in M_\Q \setminus \{\infty\}}  \rho_{p}(\xi,\s)$.
By \cref{lem:bound-non-soluble-locus}, we have
\begin{align}
    \prod_{\substack{p>P\\p \notin S}}  \rho_{p}(\xi,\s)
    = \prod_{\substack{p>P\\ p \notin S}} \mu_{\P^{N_{d,n}-1}_{\Q_p}}(Z_p)
    &= \prod_{\substack{p>P\\ p \notin S}} \left( 1 - \mu_{\P^{N_{d,n}-1}_{\Q_p}}(\P^{N_{d,n}-1}(\Q_p) \setminus Z_p) \right)\\
    &= 1 + O\left( \frac{1}{P\log 2P} \right)
\end{align}
for $P \geq 1$ where the implicit constant depends only on $d,n$. 
Here we used the following bound that follows from the prime number theorem:
\begin{align}
\prod_{p>P}\biggl(1+O\biggl(\frac{1}{p^{2}}\biggr)\biggr)
=
1+O\biggl(\frac{1}{P\log2P}\biggr)
\quad\text{for $P\ge1$}.
\end{align}
(See also \cref{lem:bound-terms-in-product} below.)
% Note that we have
% \begin{align}
%     \Biggl(\prod_{\substack{p>P\\p \notin S}}  \rho_{p}(\xi,\s) \Biggr)^{-1}
%     = 1 + O\left( \frac{1}{P\log 2P} \right)
% \end{align}
% as well.
Replacing  $\prod_{p \in \mathcal{S}}  \rho_{p}(\xi,\s)$ 
with 
\begin{align}
    &\Biggl(\prod_{\substack{p \in M_\Q\\p<\infty}}  \rho_{p}(\xi,\s) \Biggr)
    \biggl(\prod_{\substack{p>P\\p \notin S}}  \rho_{p}(\xi,\s) \biggr)^{-1}
    =\Biggl(\prod_{\substack{p \in M_\Q \\ p < \infty} }  \rho_{p}(\xi,\s) \Biggr)
    \left(1 + O\left( \frac{1}{P\log 2P} \right) \right)
\end{align}
in \cref{ineq:MAP-bounds}, we get the desired inequalities.
\end{proof}

Now we are going to choose $v_p$'s, which means we are adjusting
the radii of the small balls utilized to approximate the sets
$W_p(\xi, \s)$.
Note that as $v_p$'s increase,
$\nu(v_p,e_p)$ and $q'$ are getting large.
This means the product factor 
\begin{align}
    \prod_{p \in S_{\fin} \cup M_{\Q,\leq P}}\left(1 \pm \frac{1}{(1-p^{- N_{d,n}}) \rho_{p}(\xi, \s) p^{\nu(v_{p},e_{p})}} \right)
\end{align}
in \cref{prop:M-upper-lower-bound-general-form}
approaches $1$ as we make $v_p$'s larger.
However, this increase in $v_p$'s also make the big-O terms in 
\cref{prop:M-upper-lower-bound-general-form} large.
Therefore, our task is to choose $v_p$'s in a manner that balances these effects.
We are choosing $v_p$'s as in the form specified in the following lemma.
\begin{lemma}\label{lem:choice-of-vp}
    Let $g_p \geq 1$ for $p \leq P$.
    Set 
    \begin{align}
        v_p \coloneqq
        \begin{cases}
            g_p  & \text{when $p \leq P$ and  $p \notin S$},\\
            e_p + g_p & \text{when $p \leq P$ and $p \in S$},\\
            e_p & \text{when $p>P$ and  $p \in S$}.
        \end{cases}
    \end{align}
    Then we have
    \begin{align}
        &\nu(v_p, e_p) = g_p - \ceil*{ \frac{g_p}{2}} & \text{when $p \leq P$ and  $p \notin S$},\\
        &\nu(v_p, e_p) \geq e_p + n + g_p - \ceil*{ \frac{g_p}{2}} & \text{when $p \leq P$ and $p \in S$},\\
        &\nu(v_p, e_p) = e_p + n & \text{when $p>P$ and  $p \in S$}. 
    \end{align}
\end{lemma}
\begin{proof}
    When $p \leq P$ and $p \notin S$, we have $e_p=0$ and $\widetilde{v_p} = \ceil*{\frac{g_p}{2}}$.
    Thus we get
    \begin{align}
        \nu(v_p,e_p) = \nu(g_p,0) = g_p - \ceil*{\frac{g_p}{2}}.
    \end{align}

    When $p \leq P$ and $p \in S$, we have
    \begin{align}
        \widetilde{v_p} = \min\left\{ \ceil*{\frac{e_p + g_p}{2}}, g_p + 1  \right\}.
    \end{align}
    If $e_p \geq g_p + 1$, then $\widetilde{v_p} = g_p + 1$.
    Then 
    \begin{align}
        \nu(v_p, e_p) &= \nu(e_p + g_p, e_p) = e_p+g_p - (g_p +1) + (n+1)(g_p+1)\\
        &\geq e_p + (n+1)g_p + n.
    \end{align}
    If $e_p \leq g_p$, then 
    \begin{align}
        \widetilde{v_p} = \ceil*{\frac{e_p+g_p}{2}} \geq e_p.
    \end{align}
    Thus
    \begin{align}
        \nu(v_p,e_p) &= \nu(e_p+g_p,e_p) = e_p + g_p - \ceil*{\frac{e_p+g_p}{2}} + (n+1)e_p\\
        & \geq e_p + g_p - \ceil*{\frac{e_p}{2}} - \ceil*{\frac{g_p}{2}}+ (n+1)e_p
        \geq (n+1)e_p + g_p - \ceil*{\frac{g_p}{2}}.
    \end{align}

    Finally, when $p>P$ and $p \in S$, we have $\widetilde{v_p} = 1$.
    Thus we get
    \begin{align}
        \nu(v_p,e_p) = \nu(e_p,e_p) = e_p -1 +n+1 = e_p + n.
    \end{align}
\end{proof}

\begin{lemma}\label{lem:bound-terms-in-product}\ 
    \begin{enumerate}
        \item For $l \in \Z_{>0}$ and $a_i \in \R$, $i=1,\dots, l$, if
        $\sum_{i=1}^l |a_i| < 1/2$, then we have
        \begin{align}
            \prod_{i=1}^l (1 + a_i) = 1 + O\left(\sum_{i=1}^l |a_i|\right).
        \end{align}

        \item
        We have
        \begin{alignat}{2}
            (1-p^{- N_{d,n}}) \rho_{p}(\xi, \s) &\gg \frac{1}{p^{e_p}} &\qquad p &\in S_{\fin}\\
            (1-p^{- N_{d,n}}) \rho_{p}(\xi, \s) &\gg 1 &\qquad p &\notin S
        \end{alignat}
        where the implicit constants depend at most on $d,n$.
    \end{enumerate}
\end{lemma}
\begin{proof}
    (1)
    Since $\sum_{i=1}^l |a_i| < 1/2$, we have
    \begin{align}
        \prod_{i=1}^l (1 + a_i) \leq \prod_{i=1}^l \exp(|a_i|)=
        \exp\left(\sum_{i=1}^l |a_i|\right) = 1 + O\left(\sum_{i=1}^l |a_i|\right)
    \end{align}
    and
    \begin{align}
        \prod_{i=1}^l (1 + a_i) &\geq \prod_{i=1}^l (1 - |a_i|)
        \geq \prod_{i=1}^l\exp\left( -2|a_i|\right) 
        = \exp\left( -2 \sum_{i=1}^l|a_i| \right)\\
        &= 1 + O\left(\sum_{i=1}^l |a_i|\right).
    \end{align}

    (2)
    First inequality follows from \cref{prop-measure-of-padic-solvable-locus}.
    Second inequality follows from \cref{lem:bound-non-soluble-locus}.
    
\end{proof}

\begin{lemma}\label{lem:choosing-vps-according-to-parameter}
    There is a constant $C \geq 1$ depending only on $d,n$ with the following property.
    For any $V \geq 2$ and $P \geq 1$, there are 
    $v_p \in \Z_{\geq1}$ for each $p \in S_{\fin}\cup M_{\Q,\leq P}$ such that
    \begin{align}
        \sum_{p \in S_{\fin}\cup M_{\Q,\leq P}} 
        \frac{1}{(1-p^{-N_{d,n}})\rho_p(\xi,\s)p^{\nu(v_p,e_p)}}
        \leq C \left( \frac{\pi(P)}{V} + \frac{1}{P^{n-1}} \right)
    \end{align}
    and
    \begin{align}
        q' \leq q (P^3V^2)^{\pi(P)}.
    \end{align}
    Here $\pi(P)$ is the number of prime numbers less than or equal to $P$.
\end{lemma}
\begin{proof}
    Let us write
    \begin{align}
        a_p = \frac{1}{(1-p^{-N_{d,n}})\rho_p(\xi,\s)p^{\nu(v_p,e_p)}}.
    \end{align}
    We choose $v_p$'s as in the form of \cref{lem:choice-of-vp} and
    we specify $g_p$'s later.
    Then we have
    \begin{align}
        &\sum_{p \in S_{\fin}\cup M_{\Q,\leq P}} a_p\\
        & = \sum_{\substack{p \leq P \\ p \notin S}}a_p
        + \sum_{\substack{ p \in S_{\fin} \\ p \leq P}} a_p
        + \sum_{\substack{p \in S_{\fin} \\ p > P}} a_p\\
        &\ll 
        \sum_{\substack{p \leq P \\ p \notin S}} \frac{1}{p^{\nu(v_p,e_p)}}
        + \sum_{\substack{ p \in S_{\fin} \\ p \leq P}} \frac{p^{e_p}}{p^{\nu(v_p,e_p)}}
        + \sum_{\substack{p \in S_{\fin} \\ p > P}} \frac{p^{e_p}}{p^{\nu(v_p,e_p)}}
        & \text{by \cref{lem:bound-terms-in-product}(2)}\\
        & \leq 
        \sum_{\substack{p \leq P \\ p \notin S}} 
        \frac{1}{p^{g_p - \ceil*{ \frac{g_p}{2} }}}
        + \sum_{\substack{ p \in S_{\fin} \\ p \leq P}} 
        \frac{1}{p^{g_p - \ceil*{ \frac{g_p}{2}}+n}}
        + \sum_{\substack{p \in S_{\fin} \\ p > P}} 
        \frac{1}{p^{n}}
        & \text{by \cref{lem:choice-of-vp}}\\
        & \leq \sum_{p \leq P} \frac{1}{p^{g_p - \ceil*{ \frac{g_p}{2} }}} 
        + \sum_{p > P}  \frac{1}{p^{n}}
    \end{align}
    where the implicit constant depends only on $d,n$.
    Now for a given $V \geq 2$, let $g_p \in \Z_{\geq 1}$ be such that
    \begin{align}\label{ineq:choose-gp}
        p^{g_p - \ceil*{ \frac{g_p}{2}} -1} < V \leq p^{g_p - \ceil*{ \frac{g_p}{2}}}.
    \end{align}
    Then we get
    \begin{align}
        \sum_{p \in S_{\fin}\cup M_{\Q,\leq P}} a_p \ll
        \frac{\pi(P)}{V} + \frac{1}{P^{n-1}}.
    \end{align}
    To get the bound of $q'$, note that
    \begin{align}
        2 \left(g_p - \ceil*{ \frac{g_p}{2}} \right) + 1 \geq g_p.
    \end{align}
    Thus
    \begin{align}
        q' = \prod_{p \in S_{\fin}\cup M_{\Q,\leq P} } p^{v_p}
        &= \prod_{\substack{p \leq P \\ p \notin S}} p^{v_p}
        \prod_{\substack{ p \in S_{\fin} \\ p \leq P}} p^{v_p}
        \prod_{\substack{p \in S_{\fin} \\ p > P}}  p^{v_p}\\
        &= \prod_{\substack{p \leq P \\ p \notin S}} p^{g_p}
        \prod_{\substack{ p \in S_{\fin} \\ p \leq P}} p^{e_p+g_p}
        \prod_{\substack{p \in S_{\fin} \\ p > P}}  p^{e_p}\\
        &=q \prod_{p \leq P} p^{g_p}
        \leq q \prod_{p \leq P} p (pV)^2 \leq q (P^3V^2)^{\pi(P)}.
    \end{align}
    
\end{proof}

\begin{proposition}\label{prop:asymp-formula-MAP}
    There is a constant $C > 2$ depending only on $d,n$ with the following
    property.
    For any $A,P \geq 1$ and $V\geq 2$ such that
    \begin{align}
        \frac{\pi(P)}{V} + \frac{1}{P^{n-1}} \leq \frac{1}{C},
    \end{align}
    we have
    \begin{align}
        M(A,P)= &
        \frac{V_{ N_{d,n}}\prod_{p \in M_{\Q}} \rho_{p}(\xi, \s)}{\zeta( N_{d,n})} A^{ N_{d,n}} \Biggl( 1 + 
        O\Biggl(
        \frac{1}{P\log 2P} + \frac{\pi(P)}{V}  \\
        &+\frac{q(P^3V^2)^{\pi(P)}}{\s_{\infty}A}+
        \frac{q^{N_{d,n}-1}(\log\log 3q) (P^3V^2)^{\pi(P)N_{d,n}}\log 2A}{\s_{\infty}A^{N_{d,n}-1}}
        \Biggr)
        \Biggr)
    \end{align}
\end{proposition}
\begin{proof}
    The task is to simplify the last three factors in the
    upper and lower bounds of $M(A,P)$ in the statement of
    \cref{prop:M-upper-lower-bound-general-form}.
    Given $A,P \geq 1$ and $V \geq 2$,
    \cref{lem:choosing-vps-according-to-parameter} and
    \cref{lem:bound-terms-in-product}(1) say if
    \begin{align}
        \frac{\pi(P)}{V} + \frac{1}{P^{n-1}}
    \end{align}
    is sufficiently small (depending only on $d,n$),
    we can choose $v_p \in \Z_{\geq 1}$ so that
    \begin{align}
        \prod_{p \in S_{\fin} \cup M_{\Q, \leq P}}
        \left(1 \pm \frac{1}{(1-p^{- N_{d,n}}) \rho_{p}(\xi, \s) p^{\nu(v_{p},e_{p})}} \right)= 
        1+ O\left(\frac{\pi(P)}{V} + \frac{1}{P^{n-1}} \right)
    \end{align}
    and
    \begin{align}\label{ineq:bound-of-q'}
        q' \leq q(P^3V^2)^{\pi(P)}
    \end{align}
    where the implicit constant depends only on $d,n$.
    Since
    \begin{align}
        \frac{\pi(P)}{V} + \frac{1}{P^{n-1}} \ll 1 
        \quad \text{and} \quad \frac{1}{P \log 2P} \ll 1,
    \end{align}
    we get
    \begin{align}
        &\left(1 + O\left(\frac{\pi(P)}{V} + \frac{1}{P^{n-1}} \right) \right)
        \left(1 + O\left( \frac{1}{P\log 2P}\right) \right) \\
        &\times \left(1 + O\left(  \frac{J_{ N_{d,n}}(q')}{\rho_{\infty}(\xi, \s) q'^{ N_{d,n}-1} A}  + \frac{J_{ N_{d,n}}(q')\log 2A}{\rho_{\infty}(\xi, \s) \f(q') A^{N_{d,n}-1}}\right)   \right)\\
        &=1 + O\left(\frac{\pi(P)}{V} + \frac{1}{P^{n-1}} + \frac{1}{P\log 2P}
        +\frac{q'}{\s_{\infty}A} + \frac{{q'}^{N_{d,n}-1}\log\log 3q' \log 2A}{\s_{\infty}A^{N_{d,n}-1}}
        \right)\\
        &= 1+ O\Biggl(
        \frac{1}{P\log 2P} + \frac{\pi(P)}{V} + 
        \frac{q(P^3V^2)^{\pi(P)}}{\s_{\infty}A}\\
        &\qquad \qquad \qquad  +\frac{q^{N_{d,n}-1}(\log\log3q) (P^3V^2)^{\pi(P)N_{d,n}}\log 2A}{\s_{\infty}A^{N_{d,n}-1}}
        \Biggr)
    \end{align}
    Here for the first equality, we use \cref{easy-bound-Z}(1) and the fact
    $J_{N_{d,n}}(q')/\f(q') \ll {q'}^{N_{d,n}-1}\log\log3q'$, and for the second equality,
    we use \cref{ineq:bound-of-q'} and our assumption $n \geq 3$.
    Plugging this into the two inequalities in \cref{prop:M-upper-lower-bound-general-form}, we are done.
\end{proof}

\subsection{Asymptotic formula of \texorpdfstring{$\#\V_{d,n}^{\loc}(A;\xi,\s)$}{VdnlocAsxi}}

Recall we have
\begin{align}
    \#\V_{d,n}^{\loc}(A;\xi,\s) = \frac{1}{2}M(A,P) - \frac{1}{2}E(A,P)
\end{align}
for $A,P \geq 1$ \cref{eq:Vdn=M-E}.
By \cref{bound-of-E}, \cref{prop:asymp-formula-MAP}, 
there are constants $P_0, C_0 \geq 1$ depending only on $d,n$ such that 
for 
\begin{align}
    A \geq 1, \quad P \geq P_0, \quad V \geq C_0 \pi(P),
\end{align}
we have
\begin{align}
    \#\V_{d,n}^{\loc}(A;\xi,\s) &=
    \frac{V_{ N_{d,n}}\prod_{p \in M_{\Q}} \rho_{p}(\xi, \s)}{2\zeta( N_{d,n})} A^{ N_{d,n}}
        \Biggl(1+O\Biggl(
        \frac{1}{P\log 2P} + \frac{\pi(P)}{V} \\
        &+ \frac{q(P^3V^2)^{\pi(P)}}{\s_{\infty}A}+
        \frac{q^{N_{d,n}-1}(\log\log 3q)(P^3V^2)^{\pi(P)N_{d,n}}\log 2A}{\s_{\infty}A^{N_{d,n}-1}}
        \Biggr)\Biggr)\\
        &+ O\left( \frac{1}{P\log 2P} \frac{\s_{\infty} A^{ N_{d,n}}}{q} + A^{ N_{d,n}-1}\log \log 3A\right)
\end{align}
where the implicit constants depend only on $d,n$.
By \cref{prop-measure-of-padic-solvable-locus}, \cref{easy-bound-Z}, and 
\cref{lem:bound-non-soluble-locus}, we have 
\begin{align}
    \frac{V_{ N_{d,n}}\prod_{p \in M_{\Q}} \rho_{p}(\xi, \s)}{2\zeta( N_{d,n})} 
    \asymp_{d,n} \frac{\s_{\infty}}{q}.
\end{align}
Thus if we write
\begin{align}
    &\#\V_{d,n}^{\loc}(A;\xi,\s) = 
    \frac{V_{ N_{d,n}}\prod_{p \in M_{\Q}} \rho_{p}(\xi, \s)}{2\zeta( N_{d,n})} A^{N_{d,n}}
    \left(1 + R(A;\xi,\s)\right),
\end{align}
then
\begin{align}\label{eq:bound-of-R}
    R(A;\xi,\s) \ll &\frac{1}{P\log 2P} + \frac{\pi(P)}{V}
    + \frac{q(P^3V^2)^{\pi(P)}}{\s_{\infty}A}\\
    &+\frac{q^{N_{d,n}-1}(\log\log 3q)(P^3V^2)^{\pi(P)N_{d,n}}\log 2A}{\s_{\infty}A^{N_{d,n}-1}} \\
    &+ \frac{q \log\log 3A}{\s_{\infty}A}
\end{align}
provided $A \geq 1,P \geq P_0$, and $V \geq C_0 \pi(P)$.
Here the implicit constant depend only on $d,n$.
By adjusting parameters $V$ and $P$, we get the following.
\begin{theorem}
\label{thm:denominator}
    Let us write
    \begin{align}
    &\#\V_{d,n}^{\loc}(A;\xi,\s) = 
    \frac{V_{ N_{d,n}}\prod_{p \in M_{\Q}} \rho_{p}(\xi, \s)}{2\zeta( N_{d,n})} A^{N_{d,n}}
    \left(1 + R(A;\xi,\s)\right).
    \end{align}
    If 
    \begin{align}\label{eq:asymp-formula-vdn-condition}
        \frac{\s_\infty A}{q} \gg 1 \quad \text{and} \quad
        A \geq q \left((\log \log 3q)(\log 2A)\right)^{\frac{1}{N_{d,n}-2}},
    \end{align}
    then we have
    \begin{align}\label{eq:asymp-formula-vdn-bound}
        R(A;\xi,\s) \ll \frac{1}{\log \frac{\s_\infty A}{q} \log \log \frac{\s_\infty A}{q}}
        + \frac{q \log\log 3A}{\s_{\infty}A}.
    \end{align}
    Here the implicit constant in \cref{eq:asymp-formula-vdn-condition}
    is an absolute constant and that of \cref{eq:asymp-formula-vdn-bound}
    depends only on $d,n$.
    Moreover, the size of the coefficient of the main term is
    \begin{align}
        \frac{V_{ N_{d,n}}\prod_{p \in M_{\Q}} \rho_{p}(\xi, \s)}{2\zeta( N_{d,n})} 
    \asymp_{d,n} \frac{\s_{\infty}}{q}.
    \end{align}
\end{theorem}

\begin{proof}
    We want to find an upper bound of the infimum of the right hand side of
    \cref{eq:bound-of-R} when $V$ and $P$ vary.
    We set $V = P^2$. (Larger $V$ does not improve the upper bound because
    of the first term and prime number theorem.
    It turns out that smaller $V$ does not improve the bound as well
    because of the form of the third and fourth terms.)
    Note that by enlarging $P_0$ in advance, we have
    $P^2 \geq C_0 \pi(P)$ for all $P \geq P_0$, which we need for setting $V = P^2$. 
    Then for $A \geq 1$ and $P \geq P_0$, we have
    \begin{align}
    R(A;\xi,\s) \ll &\frac{1}{P\log 2P} 
    + \frac{qP^{7\pi(P)}}{\s_{\infty}A}\\
    &+\frac{q^{N_{d,n}-1}(\log\log 3q)P^{7\pi(P)N_{d,n}}\log 2A}{\s_{\infty}A^{N_{d,n}-1}} 
    + \frac{q \log\log 3A}{\s_{\infty}A}.
    \end{align}
    By the prime number theorem, further enlarging $P_0$, we have
    \begin{align}
        \pi(P) \leq \frac{2P}{\log P} \quad \text{for all $P \geq P_0$}.
    \end{align}
    Thus we have
    \begin{align}
        P^{7\pi(P)} \leq e^{14P}\and
        P^{7\pi(P)N_{d,n}} \leq e^{14N_{d,n}P}
    \end{align}
    for $P \geq P_0$.
    Thus
    \begin{align}\label{eq:bound-R-para-P}
    R(A;\xi,\s) \ll &\frac{1}{P\log 2P} 
    + \frac{q}{\s_{\infty}A}e^{14P}\\
    &+\frac{q^{N_{d,n}-1}(\log\log 3q)\log 2A}{\s_{\infty}A^{N_{d,n}-1}} e^{14N_{d,n}P}
    + \frac{q \log\log 3A}{\s_{\infty}A}
    \end{align}
    for $A\geq 1$ and $P \geq P_0$.
    Let us set
    \begin{align}
        &f(P) = \frac{1}{P\log 2P},\ g(P)= \frac{q}{\s_{\infty}A}e^{14P},\\
        &h(P)=\frac{q^{N_{d,n}-1}(\log\log 3q)\log 2A}{\s_{\infty}A^{N_{d,n}-1}} e^{14N_{d,n}P}.
    \end{align}
    Since $f$ is decreasing and $g,h$ are increasing with respect to $P$,
    we have
    \begin{align}\label{eq:bound-f+g+h}
        \inf_{P \geq P_0} \left( f(P) + g(P) + h(P) \right) \ll
        \inf_{P \geq P_0} f(P) + \inf_{P \geq P_0} g(P) + \inf_{P \geq P_0} h(P)
        + \alpha + \beta
    \end{align}
    where 
    \begin{align}
        &\alpha = 
        \begin{cases}
            f(P_*) \quad &\text{if $\exists P_* \geq P_0$ such that $f(P_*) = g(P_*)$}\\
            0 \quad &\text{otherwise}
        \end{cases}\\
        &\beta = 
        \begin{cases}
            f(P_*) \quad &\text{if $\exists P_* \geq P_0$ such that $f(P_*) = h(P_*)$}\\
            0 \quad &\text{otherwise}
        \end{cases}.
    \end{align}
    If $\s_{\infty}A/q \gg 1$, we can take log of the equality $f(P)=g(P)$ for 
    several times and see that 
    \begin{align}
        &\exists P \geq 3,\ \text{s.t.}\ f(P)=g(P)\ \text{and}\\
        &14P + \log P + \log \log 2P = \log \frac{\s_{\infty}A}{q};\\
        &\log P + O(1) = \log \log \frac{\s_{\infty}A}{q};\\
        &\log \log P + O(1) = \log \log \log \frac{\s_{\infty}A}{q}.
    \end{align}
    Thus we get
    \begin{align}
        \alpha \ll \frac{1}{\log \frac{\s_{\infty}A}{q} \log \log \frac{\s_{\infty}A}{q}}.
    \end{align}
    Note that under the second condition in \cref{eq:asymp-formula-vdn-condition},
    we have 
    \begin{align}
        K:=\frac{\s_{\infty}A^{N_{d,n}-1}}{q^{N_{d,n}-1}(\log\log 3q)\log 2A} \geq
        \frac{\s_{\infty}A}{q}.
    \end{align}
    Therefore, we can solve $f(P)=h(P)$ in the same way and get
    \begin{align}
        \beta \ll \frac{1}{\log K \log \log K} \leq 
        \frac{1}{\log \frac{\s_{\infty}A}{q} \log \log \frac{\s_{\infty}A}{q}}.
    \end{align}
    Going back to \cref{eq:bound-R-para-P} and \cref{eq:bound-f+g+h}, we get
    \begin{align}
        R(A;\xi,\s) \ll &\frac{q}{\s_{\infty}A} 
        + \frac{q^{N_{d,n}-1}(\log\log 3q)\log 2A}{\s_{\infty}A^{N_{d,n}-1}}\\
        &+ \frac{1}{\log \frac{\s_{\infty}A}{q} \log \log \frac{\s_{\infty}A}{q}}
        + \frac{q \log\log 3A}{\s_{\infty}A}.
    \end{align}
    By condition \cref{eq:asymp-formula-vdn-condition}, 
    the second term is bounded by the first term.
    Since the first term is obviously bounded by the last term, we get
    \begin{align}
        R(A;\xi,\s) \ll 
        \frac{1}{\log \frac{\s_{\infty}A}{q} \log \log \frac{\s_{\infty}A}{q}}
        + \frac{q \log\log 3A}{\s_{\infty}A}.
    \end{align}
\end{proof}

\appendix

\section{Newton method}

We give proofs of Newton methods here for completeness.

\begin{lemma}[Newton method --Archimedean case--]
\label{lem:Hensel_archimedean}
Consider
%%%%%
\begin{itemize}
%%%%%
	\item\label{lem:Hensel_archimedean:field}
		A field $K=\mathbb{R}$ or $\mathbb{C}$.
	\item\label{lem:Hensel_archimedean:function}
		A point $\alpha_0\in K$ and a function
		\begin{equation}
			f\colon\overline{B}(\alpha_0)\to K
		\end{equation}
        which is $C^2$ when $K= \R$ and holomorphic when $K=\C$,
		where $\overline{B}(\alpha_0)$ is the closed ball
		of radius $1$ centered at $\alpha_0$.
\end{itemize}
%%%%%
Let
\[
F\coloneqq
2\max\biggl(
\max_{x\in\overline{B}(\alpha_0)}|f'(x)|,
\max_{x\in\overline{B}(\alpha_0)}|f''(x)|,
1\biggr).
%F\coloneqq
%2\max\biggl(
%\max_{x\in\overline{B}(\alpha_0)}|f''(x)|,
%1\biggr).
\]
If the function $f(x)$ satisfies
\begin{equation}
\label{lem:Hensel_archimedean:initial_cond}
	f'(\alpha_0)\neq0
	\and
	\eta\coloneqq\frac{|f(\alpha_0)|}{|f'(\alpha_0)|^2}<\frac{1}{F},
\end{equation}
then there exists $\alpha\in \overline{B}(\alpha_0)$ such that
\begin{enumerate}[label=\textup{(\arabic*)}]
	\item\label{lem:Hensel_archimedean:zero}
		We have $f(\alpha)=0$, i.e.\ $\alpha$ is a zero of $f(x)$.
	\item\label{lem:Hensel_archimedean:close}
		We have
		\[
			|\alpha_0-\alpha|
			<(1+F\eta)\frac{|f(\alpha_0)|}{|f'(\alpha_0)|}
			<(1+F\eta)|f'(\alpha_0)|,
		\]
		i.e.\ $\alpha$ is close to $\alpha_0$.
\end{enumerate}
\end{lemma}
%%%%%%%%%%%%%%%%%%%%%%%%%%%%%%%%%%%%%%%%
\begin{proof}
Let us define a sequence of positive real numbers $(\eta_n)_{n=1}^{\infty}$ inductively by
\begin{equation}
	\label{lem:Hensel_archimedean:eta_def}
	\eta_0\coloneqq\eta
	\and
	\eta_{n+1}\coloneqq F\eta_n^2.
\end{equation}
Since $0\le\eta<\frac{1}{F}$, we find that $\eta$ is non-increasing.
We then define a sequence $(\alpha_n)_{n=0}^{\infty}$ so that $\alpha_0$ is 
the given one and
\begin{equation}
\label{lem:Hensel_archimedean:alpha_def}
	\alpha_{n+1}\coloneqq\alpha_n-\frac{f(\alpha_n)}{f'(\alpha_n)},
\end{equation}
where the well-definedness, i.e.\ 
$\alpha_n\in\overline{B}(\alpha_0)$ and $f'(\alpha_n)\neq0$ will be proven in the induction below.
We prove the following claims by induction on $n$:
%%%%%
\begin{enumerate}[label=\textup{(\alph*)}]
%%%%%
	\item\label{lem:Hensel_archimedean:alpha_bound}
		We have
		\begin{equation}
			|\alpha_m-\alpha_n|
			\le\biggl|\frac{f(\alpha_0)}{f'(\alpha_0)}\biggr|
			\sum_{m\le i<n}\bigg(\frac{F\eta}{2}\bigg)^i
			\le1
		\end{equation}
		for all $0\le m\le n$. In particular, $\alpha_n\in\overline{B}(\alpha_0)$.
	\item\label{lem:Hensel_archimedean:dev_bound}
		We have
		\begin{equation}
			|f'(\alpha_n)|\ge\frac{1}{2}|f'(\alpha_{n-1})|
		\end{equation}
		for all $n\ge1$. Also, $|f'(\alpha_n)|>0$.
	\item\label{lem:Hensel_archimedean:frac_bound2}
		We have
		\begin{equation}
			\frac{|f(\alpha_n)|}{|f'(\alpha_n)|^2}
			\le\eta_n\le\eta
		\end{equation}
		for all $n\ge0$.
	\item\label{lem:Hensel_archimedean:frac_bound1}
		We have
		\begin{equation}
			\biggl|\frac{f(\alpha_n)}{f'(\alpha_n)}\biggr|
			\le\biggl|\frac{f(\alpha_0)}{f'(\alpha_0)}\biggr|\biggl(\frac{F\eta}{2}\biggr)^{n}.
		\end{equation}
\end{enumerate}
%%%%%
\medskip

%%%%%
\prooftitle{Initial case $n=0$.}
The assertion \cref{lem:Hensel_archimedean:alpha_bound}
and \cref{lem:Hensel_archimedean:frac_bound1} are trivial
since in this case, $0\le m\le n$ implies $n=m=0$.
The former part of assertion \cref{lem:Hensel_archimedean:dev_bound} is not relevant to the case $n=0$.
The latter part of assertion \cref{lem:Hensel_archimedean:dev_bound},
$f'(\alpha_0)\neq0$ is assumed in \cref{lem:Hensel_archimedean:initial_cond}.
The assertion \cref{lem:Hensel_archimedean:frac_bound2} is obvious from the definition of $\eta$.
\medskip

%%%%%
\prooftitle{Induction step from $n$ to $n+1$.}
Assume the assertion holds up to the $n$-th case, where $n\ge0$.
We then show the $(n+1)$-th case.
By \cref{lem:Hensel_archimedean:dev_bound} for the $n$-th case,
the denominator $f'(\alpha_{n})$ in the definition \cref{lem:Hensel_archimedean:alpha_def} is non-zero.
For any $0\le m\le n$,
by \cref{lem:Hensel_archimedean:alpha_def}
and \cref{lem:Hensel_archimedean:frac_bound1} for the $1,\ldots,n$-th case, we have
\[
|\alpha_m-\alpha_{n+1}|
\le
\sum_{m\le i<n+1}\biggl|\frac{f(\alpha_i)}{f'(\alpha_i)}\biggr|
\le
\biggl|\frac{f(\alpha_0)}{f'(\alpha_0)}\biggr|
\sum_{m\le i<n+1}\biggl(\frac{F\eta}{2}\biggr)^{i}.
\]
By $0\le\eta<\frac{1}{F}$ and \cref{lem:Hensel_archimedean:initial_cond}, we thus have
\[
\biggl|\frac{f(\alpha_0)}{f'(\alpha_0)}\biggr|
\sum_{m\le i<n+1}\biggl(\frac{F\eta}{2}\biggr)^{i}
\le 2\biggl|\frac{f(\alpha_0)}{f'(\alpha_0)}\biggr|
= 2|f'(\alpha_0)|\eta
\le F\eta
\le 1.
\]
By considering the case $m=0$, we also find $\alpha_{n+1}\in\overline{B}(\alpha_0)$.
This proves \cref{lem:Hensel_archimedean:alpha_bound} for the $(n+1)$-th case.
By applying the Taylor theorem to $f'(x)$
(when $K=\mathbb{C}$, we use the Taylor theorem on the line segment between $\alpha_n$ and $\alpha_{n+1}$)
with \cref{lem:Hensel_archimedean:alpha_def},
\begin{align}
|f'(\alpha_{n+1})|
&=
\biggl|f'(\alpha_{n})+\int_{\alpha_{n}}^{\alpha_{n+1}}f''(u)du\biggr|\\
&\ge
|f'(\alpha_{n})|-\frac{F}{2}|\alpha_{n+1}-\alpha_n|
=
|f'(\alpha_{n})|-\frac{F}{2}\biggl|\frac{f(\alpha_{n})}{f'(\alpha_n)}\biggr|.
\end{align}
By \cref{lem:Hensel_archimedean:frac_bound2} for the $n$-th case,
since $0\le\eta<\frac{1}{F}$, we have
\[
|f'(\alpha_{n+1})|
\ge|f'(\alpha_{n})|\biggl(1-\frac{F\eta}{2}\biggr)
\ge\frac{1}{2}|f'(\alpha_{n})|.
\]
By \cref{lem:Hensel_archimedean:dev_bound} for the $n$-th case,
we also find from this that $f'(\alpha_{n+1})\neq0$.
This proves \cref{lem:Hensel_archimedean:dev_bound} for the $(n+1)$-th case.
By applying the Taylor theorem to $f(x)$
with \cref{lem:Hensel_archimedean:alpha_def}, we get
\begin{equation}
\label{lem:Hensel_archimedean:Taylor_quadratic}
\begin{aligned}
|f(\alpha_{n+1})|
&=
\biggl|
f(\alpha_{n})+f'(\alpha_n)(\alpha_{n+1}-\alpha_{n})+\int_{\alpha_{n}}^{\alpha_{n+1}}f''(u)(\alpha_{n+1}-u)du\biggr|\\
&\le
\biggl|f(\alpha_{n})-f'(\alpha_{n})\frac{f(\alpha_n)}{f'(\alpha_{n})}\biggr|
+
\frac{F}{4}\biggl|\frac{f(\alpha_n)}{f'(\alpha_{n})}\biggr|^2
=
\frac{F}{4}\biggl|\frac{f(\alpha_n)}{f'(\alpha_{n})}\biggr|^2.
\end{aligned}
\end{equation}
By using \cref{lem:Hensel_archimedean:dev_bound} for the $(n+1)$-th case proven above, this gives
\[
\frac{|f(\alpha_{n+1})|}{|f'(\alpha_{n+1})|^2}
\le
F\biggl(\frac{|f(\alpha_n)|}{|f'(\alpha_{n})|^2}\biggr)^2.
\]
By \cref{lem:Hensel_archimedean:frac_bound2} for the $n$-th case,
we get
\[
\frac{|f(\alpha_{n+1})|}{|f'(\alpha_{n+1})|^2}
\le F\eta_n^2
=\eta_{n+1}.
\]
i.e.\ \cref{lem:Hensel_archimedean:frac_bound2} for the $(n+1)$-th case.
By \cref{lem:Hensel_archimedean:Taylor_quadratic}
and \cref{lem:Hensel_archimedean:dev_bound}, \cref{lem:Hensel_archimedean:frac_bound2} for the $(n+1),n$-th case,
\[
\biggl|\frac{f(\alpha_{n+1})}{f'(\alpha_{n+1})}\biggr|
\le
\frac{F}{4}\frac{1}{|f'(\alpha_{n+1})|}\biggl|\frac{f(\alpha_{n})}{f'(\alpha_{n})}\biggr|^2
\le
\frac{F}{2}\frac{|f(\alpha_{n})|}{|f'(\alpha_{n})|^2}\biggl|\frac{f(\alpha_{n})}{f'(\alpha_{n})}\biggr|
\le
\frac{F\eta}{2}\biggl|\frac{f(\alpha_{n})}{f'(\alpha_{n})}\biggr|
\]
By \cref{lem:Hensel_archimedean:frac_bound1} for the $n$-th case, we obtain
\cref{lem:Hensel_archimedean:frac_bound1} for the $(n+1)$-th case.
This completes the induction.

%%%%%%%%%%%%%%%%%%%%%%%%%%%%%%%%%%%%%%%%
By \cref{lem:Hensel_archimedean:alpha_bound},
$(\alpha_{n})_{n=1}^{\infty}$
is Cauchy and so converges to a point in $\overline{B}(\alpha_0)$.
Let
\[
\alpha\coloneqq\lim_{n\to\infty}\alpha_n.
\]
By \cref{lem:Hensel_archimedean:alpha_bound} and \cref{lem:Hensel_archimedean:frac_bound2}, we then have
\[
|f(\alpha_n)|
\le
\biggl(\frac{F}{2}\biggr)^2\eta_n
\to0\quad(n\to\infty)
\]
and so $f(\alpha)=0$, i.e.\ \cref{lem:Hensel_archimedean:zero} holds.
By \cref{lem:Hensel_archimedean:alpha_bound}, we have
\[
|\alpha-\alpha_0|
=\lim_{m\to\infty}|\alpha_m-\alpha_0|
\le\biggl|\frac{f(\alpha_0)}{f'(\alpha_0)}\biggr|
\sum_{i=0}^{\infty}\bigg(\frac{F\eta}{2}\bigg)^i
\le(1+F\eta)\biggl|\frac{f(\alpha_0)}{f'(\alpha_0)}\biggr|.
\]
This proves \cref{lem:Hensel_archimedean:close} and completes the proof.
\end{proof}

\begin{lemma}[Newoton method --Non-archimedean case--]\label{lem:newton-method-na-appendix}
Let $(K,|\cdot |)$ be a field with complete non-archimedean absolute value.
Let $R = \{a \in K \mid |a|\leq 1\}$.
For any polynomial $f \in R[t]$ and $ \alpha_{0} \in R$, if
\begin{align}
|f( \alpha_{0})| < |f'( \alpha_{0})|^{2},
\end{align}
then there is $ \alpha \in R$ such that
\begin{align}
f( \alpha)=0 \quad \text{and}\quad
| \alpha - \alpha_{0}| \leq \left| \frac{f( \alpha_{0})}{ f'( \alpha_{0})} \right|.
\end{align}
\end{lemma}
\begin{proof}
We inductively construct $ \alpha_{n} \in K$ for $n =0, 1,2,3,\dots$ with $ \alpha_{0}$ being the given one.
Let 
\begin{align}
c_{0} = \frac{|f( \alpha_{0})|}{|f'( \alpha_{0})|^{2}} \quad \text{and} \quad c_{n} = c_{n-1}^{2} \quad \text{for $n=1,2,\dots$}.
\end{align}
Note that $c_{n} < 1$ and $c_{n} \to 0$.
Suppose we have constructed $ \alpha_{0},\dots, \alpha_{n}$ so that for $i=0,\dots,n$, we have
\begin{enumerate}
\item[$(1)_{i}$] $| \alpha_{i}| \leq 1$;
\item[$(2)_{i}$] $|f'( \alpha_{i})| = |f'( \alpha_{0})|$;
\item[$(3)_{i}$] $|f( \alpha_{i})| \le c_{i}|f'( \alpha_{i})|^{2}$
\end{enumerate}
and
\begin{align}
\alpha_{i} = \alpha_{i-1} - \frac{f( \alpha_{i-1})}{f'( \alpha_{i-1})} \quad \text{for $i=1,\dots,n$.}
\end{align}
Now we set 
\begin{align}
\alpha_{n+1} = \alpha_{n} - \frac{f( \alpha_{n})}{f'( \alpha_{n})}.
\end{align}
We are going to show $(1)_{n+1},(2)_{n+1},(3)_{n+1}$ hold.
First by the definition of $ \alpha_{n+1}$, $(1)_{n}, (3)_{n}$, and the fact that $f \in R[t]$, we have
\begin{align}
| \alpha_{n+1}| = \left| \alpha_{n} - \frac{f( \alpha_{n})}{f'( \alpha_{n})} \right| \leq \max\{ | \alpha_{n}| , c_{n}|f'( \alpha_{n})| \} \leq 1,
\end{align}
that is $(1)_{n+1}$ holds.
Next, using $(1)_{n},(1)_{n+1},(3)_{n}$, we can show that
\begin{align}
|f'( \alpha_{n+1})-f'( \alpha_{n})| \leq | \alpha_{n+1} - \alpha_{n}| = \left|  \frac{f( \alpha_{n})}{f'( \alpha_{n})}\right|
\le c_{n}|f'( \alpha_{n})| < |f'( \alpha_{n})|.
\end{align}
This implies $|f'( \alpha_{n+1})| = |f'( \alpha_{n})|$ and hence $|f'( \alpha_{n+1})| = |f'( \alpha_{0})| $.
This proves $(2)_{n+1}$.
Finally, by the definition of $ \alpha_{n+1}$, we have
\begin{align}
|f( \alpha_{n+1})|& = | f( \alpha_{n+1}) - f( \alpha_{n}) - f'( \alpha_{n})( \alpha_{n+1} - \alpha_{n}) | \\
& \leq | \alpha_{n+1} - \alpha_{n}| ^{2} = \left|  \frac{f( \alpha_{n})}{f'( \alpha_{n})}\right|^{2}.
\end{align}
Here the inequality follows from the fact that
\begin{align}
\frac{f(x) - f(y) - f'(y)(x-y)}{(x-y)^{2}}
\end{align}
is a polynomial in $x,y$ over $R$.
Therefore, by $(2)_{n+1}, (3)_{n}$ we get
\begin{align}
\left|  \frac{f( \alpha_{n+1})}{f'( \alpha_{n+1})^{2}}\right| \leq \left|  \frac{f( \alpha_{n})}{f'( \alpha_{n})^{2}}\right|^{2} \le c_{n}^{2} = c_{n+1}.
\end{align}
This proves $(3)_{n+1}$.

By this construction, we get a sequence $\{ \alpha_{n}\}_{n=0}^{\infty}$.
Since 
\begin{align}
| \alpha_{n+1} - \alpha_{n}| = \left|  \frac{f( \alpha_{n})}{f'( \alpha_{n})}\right| \le c_{n}|f'( \alpha_{0})|,
\end{align}
there is a limit $ \alpha = \lim_{n \to \infty} \alpha_{n}$ in $R$.
We show this $ \alpha$ is what we wanted.
First, we have $|f( \alpha)| = \lim_{n\to \infty} |f( \alpha_{n})|  \leq \lim_{n \to \infty}c_{n}|f'( \alpha_{0})|^{2} = 0$ and hence $f( \alpha)=0$.

Note that for $n \geq 0$, we have
\begin{align}
\left|  \frac{f( \alpha_{n})}{f'( \alpha_{n})}\right| \le c_{n}|f'( \alpha_{n})| \leq c_{0}|f'( \alpha_{0})| = \left|  \frac{f( \alpha_{0})}{f'( \alpha_{0})}\right|.
\end{align}
Therefore, if we take $n \geq 0$ large enough so that $| \alpha - \alpha_{n}| \leq |f( \alpha_{0})/f'( \alpha_{0})|$,
then we get
\begin{align}
| \alpha - \alpha_{0}| \leq \max\left\{| \alpha - \alpha_{n}|, \left|  \frac{f( \alpha_{n-1})}{f'( \alpha_{n-1})}\right|, \dots,
\left|  \frac{f( \alpha_{0})}{f'( \alpha_{0})}\right|  \right\}
\leq \left|  \frac{f( \alpha_{0})}{f'( \alpha_{0})}\right|.
\end{align}
\end{proof}

%%%%%%%%%%%%%%%%%%%%%%%%%%%%%%%%%%%%%%%%
\bibliographystyle{amsplain}
\bibliography{FirstMoment}
\bigskip

%%%%%%%%%%%%%%%%%%%%
\begin{flushleft}
{\textsc{%
\small
Yohsuke Matsuzawa\\[.3em]
\footnotesize
Department of Mathematics, Graduate School of Science,\\
Osaka Metropolitan University,
3-3-138, Sugimoto, Sumiyoshi, Osaka, 558-8585, Japan.}

\small
\textit{Email address}: \texttt{matsuzaway@omu.ac.jp}
}
\bigskip

{\textsc{%
\small
Yuta Suzuki\\[.3em]
\footnotesize
Department of Mathematics, Rikkyo University,\\
3-34-1 Nishi-Ikebukuro, Toshima-ku, Tokyo 171-8501, Japan.
}

\small
\textit{Email address}: \texttt{suzuyu@rikkyo.ac.jp}
}
\end{flushleft}

\end{document}